\documentclass[preprint,10pt]{elsarticle}

\usepackage{amssymb}
\usepackage{amsmath}
\usepackage{graphicx}%
\usepackage{multirow}%
\usepackage{amsfonts}%
\usepackage{amsthm}%
\usepackage{mathrsfs}%
\usepackage[title]{appendix}%
\usepackage{xcolor}%
\usepackage{textcomp}%
\usepackage{manyfoot}%
\usepackage{booktabs}%
\usepackage{algorithm}%
\usepackage{algorithmicx}%
\usepackage{algpseudocode}%
\usepackage{listings}%
\usepackage{tabularx}
\usepackage{makecell}
\usepackage{ragged2e}
\usepackage{lineno}
\usepackage{subfigure}
\usepackage{float}
\usepackage{caption}
\newtheorem{theorem}{Theorem}[section]
\newtheorem{lemma}[theorem]{Lemma}%
\newtheorem{proposition}[theorem]{Proposition}%
\newtheorem{corollary}[theorem]{Corollary}%
\newtheorem{example}{Example}[section]%
\newtheorem{remark}{Remark}[section]%
\newtheorem{definition}{Definition}[section]%

\begin{document}

\begin{frontmatter}


\title{Optimal Fluctuations for Discrete-time Markov Jump Processes}
\author[1]{Feng Zhao}
\ead{zhaofengnuaa@163.com}
\author[1]{Jinjie Zhu}
\ead{jinjiezhu@nuaa.edu.cn}
\author[2]{Yang Li}
\ead{liyangbx5433@163.com}
\author[1]{Xianbin Liu\corref{cor1}}
\ead{xbliu@nuaa.edu.cn}
\author[1]{Dongping Jin\corref{cor1}}
\ead{jindp@nuaa.edu.cn}
\cortext[cor1]{Corresponding author.}

\affiliation[1]{organization={State Key Laboratory of Mechanics and Control for Aerospace Structures, College of Aerospace Engineering, Nanjing University of Aeronautics and Astronautics},
            addressline={29 Yudao Street}, 
            city={Nanjing},
            postcode={210016}, 
            country={China}}
\affiliation[2]{organization={School of Automation,Nanjing University of Science and Technology},addressline={200 Xiaolingwei Street},city={Nanjing},postcode={210094},country={China}}            
            
\begin{abstract}
	In the last few decades, noise‐induced large fluctuations and transition phenomena have garnered significant attention in a variety of scientific contexts. The concept of prehistory probability has been proposed within the framework of Langevin dynamics to illustrate the focusing effect of large fluctuation paths onto a deterministic trajectory known as the optimal path. The present paper is devoted to showing that such a focusing effect persists within the framework of discrete-time Markov jump processes. Our proof leverages large deviation theory and the concept of time reversal for Markov jump processes. A key finding is the relationship identified between the optimal path and the time reversal of a specific family of probability distributions. This theoretical framework elucidates how an essentially deterministic mechanism can emerge from rare stochastic events in discrete-time Markov jump systems.
\end{abstract}

\end{frontmatter}
\section{Introduction}\label{sec01}
Consider a class of $d$-dimensional right-continuous random step functions $\{\boldsymbol{x}^\varepsilon(t):t\in[0,T]\}$ defined by
\begin{equation}\label{eq010101}
	\left\{\begin{aligned}
		&\boldsymbol{x}^\varepsilon(t)=\boldsymbol{x}^\varepsilon(n\varepsilon), &&t\in[n\varepsilon,(n+1)\varepsilon),\\
		&\boldsymbol{x}^\varepsilon(t)=\boldsymbol{x}^\varepsilon(\lfloor{T/\varepsilon}\rfloor\varepsilon), &&t\in[\lfloor{T/\varepsilon}\rfloor\varepsilon,T],\\
		&\boldsymbol{x}^\varepsilon((n+1)\varepsilon)=\boldsymbol{x}^\varepsilon(n\varepsilon)+\varepsilon\boldsymbol{Z}_{n}, &&n=0,1,\cdots,\lfloor{T/\varepsilon}\rfloor-1,
	\end{aligned}\right.
\end{equation}
where $\{\boldsymbol{Z}_{n}\}$ is a jump process whose conditional jump probability at time $n$ is stationary (i.e., independent of the variable $n$), independent of $\boldsymbol{Z}_{k}$ for $k<n$, dependent only on the current state of $\boldsymbol{x}^\varepsilon(n\varepsilon)$ (i.e., independent of previous states of $\boldsymbol{x}^\varepsilon(k\varepsilon)$ for $k<n$), and given by the probability measure $\mu_{\boldsymbol{x}}$. Specifically, for any Borel set $D\subset\mathbb{R}^d$ and $n=0,1,\cdots,\lfloor{T/\varepsilon}\rfloor-1$,
\begin{equation*}
	P^\varepsilon\left(\boldsymbol{Z}_{n}\in D\mid\boldsymbol{x}^\varepsilon(n\varepsilon)=\boldsymbol{x},\mathcal{F}_{<n}\right)=\mu_{\boldsymbol{x}}(D),
\end{equation*}
 where $\mathcal{F}_{<n}$ denotes the history of $\boldsymbol{x}^\varepsilon(k\varepsilon)$ before $n$. We assume $0<\varepsilon\ll{1}$ throughout. Consequently, $\boldsymbol{x}^{\varepsilon}(t)$ exhibits $\varepsilon^{-1}$ jumps per unit time, while the magnitude of each jump is scaled by $\varepsilon$ as compared to that of the base process $\boldsymbol{x}^{1}(t)$ (i.e., a process corresponding to the discrete-time Markov chain $\boldsymbol{X}_{n+1}=\boldsymbol{X}_n+\boldsymbol{Z}_n$ for $n=0,1,\cdots$). This process admits the unified representation:
\begin{equation}\label{eq010102}
	\left\{\begin{aligned}
		&\boldsymbol{x}^\varepsilon(t)=\boldsymbol{x}^\varepsilon(0), &&t\in[0,\varepsilon),\\
		&\boldsymbol{x}^\varepsilon(t)=\boldsymbol{x}^\varepsilon(0)+\varepsilon\sum_{n=0}^{\lfloor{t/\varepsilon}\rfloor-1}\boldsymbol{Z}_{n}, &&t\in[\varepsilon,T].
	\end{aligned}\right.
\end{equation}

Most Markov models arising in statistical physics \cite{Gardiner_1985,Van_1992}, chemistry \cite{Kurtz_1972a,Gillespie_2007,Kurtz_2015}, queueing theory \cite{Newell_2013,Shwartz_1995}, and communications theory \cite{Shwartz_1995} can be formulated in this way. For instance, if $\mu_{\boldsymbol{x}}$ is a state-dependent Gaussian measure with mean vector $\boldsymbol{b}(\boldsymbol{x})\in\mathbb{R}^d$ and covariance matrix $\boldsymbol{\sigma}(\boldsymbol{x})\boldsymbol{\sigma}(\boldsymbol{x})^{\top}/2\in\mathbb{R}^{d\times{d}}$, then the process $\boldsymbol{x}^{\varepsilon}(t)$ reduces to the Euler discretization \cite{Oksendal_2013} of the diffusion process:
\begin{equation}\label{eq010103}
	\boldsymbol{y}^{\varepsilon}(t)=\boldsymbol{y}^{\varepsilon}(0)+\int_{0}^{t}\boldsymbol{b}(\boldsymbol{y}^{\varepsilon}(s))\mathrm{d}s+\sqrt{\varepsilon/2}\int_{0}^{t}\boldsymbol{\sigma}(\boldsymbol{y}^{\varepsilon}(s))\mathrm{d}\boldsymbol{w}(s),\quad t\in[0,T],
\end{equation}
where $\boldsymbol{w}(t)$ is a standard $d$-dimensional Wiener process. Alternatively, if $\boldsymbol{Z}_{n}$ can be expressed as a linear combination of $N$ mutually independent Poisson-distributed random variables with state-dependent intensity parameters $R_{i}(\boldsymbol{x})>0$ and coefficients $\boldsymbol{\nu}_{i}\in\mathbb{R}^d$ ($i=1,2,\cdots,N$), then we obtain the tau-leaping discretization \cite{Kurtz_2015} of the stochastic biochemical reaction model:
\begin{equation}\label{eq010104}
	\boldsymbol{y}^{\varepsilon}(t)=\boldsymbol{y}^{\varepsilon}(0)+\varepsilon\sum_{i=1}^{N}Y_{i}\left(\varepsilon^{-1}\int_{0}^{t}R_{i}(\boldsymbol{y}^{\varepsilon}(s))\mathrm{d}s\right)\boldsymbol{\nu}_{i},\quad t\in[0,T],
\end{equation}
where $\{Y_{i}\}$ denote $N$ independent standard Poisson processes. Unless explicitly stated otherwise, we do not specify $\mu_{\boldsymbol{x}}$'s exact form. At this point, the aforementioned process $\boldsymbol{x}^{\varepsilon}(t)$ can be regarded as a natural discrete-time extension of both the diffusion model (\ref{eq010103}) and the continuous-time Markov chain (\ref{eq010104}).

Let $\boldsymbol{b}(\boldsymbol{x})\triangleq\int_{\mathbb{R}^d}\boldsymbol{z}\mu_{\boldsymbol{x}}(\mathrm{d}\boldsymbol{z})$. Under suitable conditions, it is not too difficult to show that as $\varepsilon\to {0}$, the process $\{\boldsymbol{x}^{\varepsilon}(t):t\in[0,T]\}$ converges in probability to a deterministic trajectory $\{\boldsymbol{x}^{0}(t):t\in[0,T]\}$, which is defined as the solution to the ordinary differential equation
\begin{equation}\label{eq010105}
	\boldsymbol{x}^{0}(t)=\boldsymbol{x}^{0}(0)+\int_{0}^{t}\boldsymbol{b}(\boldsymbol{x}^{0}(s))\mathrm{d}s,\quad t\in[0,T],
\end{equation} 
with initial condition $\boldsymbol{x}^{\varepsilon}(0)=\boldsymbol{x}^{0}(0)=\boldsymbol{x}_{0}$ (cf. \cite[Theorem 4.7]{Kurtz_1970}). That is, for sufficiently small $\varepsilon$, $\boldsymbol{x}^{\varepsilon}(t)$ predominantly oscillates stochastically in the vicinity of $\boldsymbol{x}^{0}(t)$, with large excursions occurring occasionally. This is a consequence of the type of the law of large numbers. Given this, it is reasonable to interpret the stochastic model (\ref{eq010102}) as a weak random perturbation of the dynamical system (\ref{eq010105}).

In contrast to the law of large numbers, the large deviation principle (LDP) (cf. \cite{Venttsel_1977a,Venttsel_1977b,Venttsel_1980,Venttsel_1983,Wentzell_2012}) provides a tool to estimate the probabilities of rare events more accurately. To be specific, let $\mathcal{G}$ be a regular set (cf. \cite[Chapter 3, Theorem 3.4]{Freidlin_2012}) consisting of $\mathbb{R}^d$-valued functions on $[0,T]$. Then, under appropriate conditions,
\begin{equation*}
	P^{\varepsilon}_{\boldsymbol{x}_0}\left(\{\boldsymbol{x}^{\varepsilon}(t):t\in[0,T]\}\in\mathcal{G}\right)\simeq\exp\left(-\varepsilon^{-1}\inf_{\boldsymbol{\phi}\in\mathcal{G}}I_{[0,T]}\left(\{\boldsymbol{\phi}(t):t\in[0,T]\}\right)\right),
\end{equation*}
as $\varepsilon\to{0}$, where $P_{\boldsymbol{x}_0}^{\varepsilon}$ denotes the probability measure conditioned on $\boldsymbol{x}^{\varepsilon}(0)=\boldsymbol{x}_0$, and $I_{[0,T]}$ is a good rate function to be defined explicitly later. Moreover, if $\inf_{\boldsymbol{\phi}\in\text{cl}(\mathcal{G})}I_{[0,T]}\left(\boldsymbol{\phi}\right)$ is attained at a unique path $\boldsymbol{\phi}^{*}$, then for any $\delta>0$,
\begin{equation*}
	\lim_{\varepsilon\to{0}}\frac{P^{\varepsilon}_{\boldsymbol{x}_0}\left(\{\boldsymbol{x}^{\varepsilon}(t):t\in[0,T]\}\in\mathcal{G}\cap\{\boldsymbol{\phi}:\sup_{t\in[0,T]}\left\vert\boldsymbol{\phi}(t)-\boldsymbol{\phi}^{*}(t)\right\vert<\delta\}\right)}{P^{\varepsilon}_{\boldsymbol{x}_0}\left(\{\boldsymbol{x}^{\varepsilon}(t):t\in[0,T]\}\in\mathcal{G}\right)}=1.
\end{equation*}
Since $I_{[0,T]}(\boldsymbol{\phi})\geq{0}$ with equality if and only if $\boldsymbol{\phi}(t)=\boldsymbol{x}^{0}(t)$ for all $t\in[0,T]$, it follows that $\inf_{\boldsymbol{\phi}\in\mathcal{G}}I_{[0,T]}\left(\boldsymbol{\phi}\right)>0$ whenever the unperturbed trajectory $\boldsymbol{x}^{0}$ lies outside the closure $\text{cl}(\mathcal{G})$ of $\mathcal{G}$. In such cases, $\mathcal{G}$ represents a rare event with an exponentially small probability dominated by $\inf_{\boldsymbol{\phi}\in\mathcal{G}}I_{[0,T]}\left(\boldsymbol{\phi}\right)$, and any rare large fluctuations in $\mathcal{G}$, when they occur, are more likely to be close to the minimizer $\boldsymbol{\phi}^{*}$. Consequently, $\boldsymbol{\phi}^{*}$ is highlighted as it unveils the nearly deterministic mechanism concealed behind rare noisy phenomena. The variational principle can hence be used to identify the optimal fluctuation in $\mathcal{G}$ as the solution to $\boldsymbol{\phi}^{*}=\arg\min_{\boldsymbol{\phi}\in\mathcal{G}}I_{[0,T]}\left(\boldsymbol{\phi}\right)$. Numerous numerical methods have been developed to solve this optimization problem. For further details on this topic, see References \cite{E_2004,Matthias_2008,Tobias_2019,Li_2021,Ruben_2023,Nocedal_2006}.

The present paper further explores the optimal fluctuation in $\mathcal{G}$ for the specific case where  $\mathcal{G}=\{\boldsymbol{\phi}:\boldsymbol{\phi}(0)=\boldsymbol{x}_0,\boldsymbol{\phi}(T)=\boldsymbol{x}_T\}$. Instead of solving the corresponding constrained optimization problem directly, our primary objective is to provide an alternative characterization of the optimal fluctuation in this non-stationary setting. The idea originates from the work of Dykman et al. \cite{Dykman_1992,Dykman_1996,ANISHCHENKO_2001}, who first proposed the concept of prehistory probability and showed that it can be used not only to pinpoint the location of the optimal path $\boldsymbol{\phi}^{*}$ but also to capture the statistical features of nearby trajectories. The original prehistorical description of optimal fluctuations was formulated exclusively for stationary Langevin models, whose extension to non-stationary cases remained undeveloped until recently. In \cite{Zhao_2022}, we leveraged a result from \cite{Haussmann_1986,Millet_1989,Cattiaux_2023} to derive a time-reversed diffusion process so that it is possible to interpret the prehistory probability as the conditional probability of the reversed process. In doing so, the focusing effect of large fluctuation paths onto the optimal path $\boldsymbol{\phi}^{*}$ can then be proved by means of limit theorems for the reversed process. This approach naturally reveals a profound connection between the optimal path and a time-reversed process, thereby facilitating a more intuitive understanding of the essence of the optimal fluctuation. Furthermore, in light of recent advances in the time reversal of Markov jump processes \cite{Conforti_2022}, analogous conclusions have also been shown to be applicable to stochastic processes of the form (\ref{eq010104}) \cite{Zhao_2025}.

In this paper, we generalize the aforementioned results in \cite{Zhao_2022,Zhao_2025} to a broad class of discrete-time Markov jump processes of the form (\ref{eq010102}) and show how the optimal fluctuation in $\mathcal{G}$ is closely related to a time-reversed discrete-time Markov jump process. The remainder of this paper is organized as follows. A concise exposition of the large deviation principle and several corollaries pertinent to these Markov jump processes is presented in Sec. 2. The definition of prehistory probability and its relationship with the time reversal of a specific family of distributions are provided in Sec. 3. The prehistorical description of optimal fluctuations is subsequently established in Sec. 4 by means of the law of large numbers. Several illustrative examples are given in Sec. 5, and the conclusions of this study are set out in Sec. 6. The proof of Proposition \ref{theo040101} and an algorithm for calculating the prehistory probability are placed in the Appendix.

\section{Preliminaries}\label{sec02}
\subsection{Large Deviation Principle}\label{sec0201}
For a fixed $T>0$, let $\mathcal{D}([0,T];\mathbb{R}^d)$ denote the space of all $\mathbb{R}^d$-valued functions on $[0,T]$ that are right-continuous with left limits. We equip this space with the uniform-convergence metric $\rho_{[0,T]}$, which is defined by
\begin{equation*}
	\rho_{[0,T]}\left(\left\{\boldsymbol{x}(t):t\in [0,T]\right\},\left\{\boldsymbol{y}(t):t\in [0,T]\right\}\right)\triangleq\sup_{t\in[0,T]}\left\vert\boldsymbol{x}(t)-\boldsymbol{y}(t)\right\vert,
\end{equation*}
where $\vert\cdot\vert$ denotes the Euclidean norm. Then $\left(\mathcal{D}([0,T];\mathbb{R}^d),\rho_{[0,T]}\right)$ forms a metric space.

Define the Hamiltonian $H$ and Lagrangian $L$ as:
\begin{equation*}
	H(\boldsymbol{x},\boldsymbol{\alpha})\triangleq\ln\int_{\mathbb{R}^d}e^{\boldsymbol{\alpha}\cdot\boldsymbol{z}}\mu_{\boldsymbol{x}}(\mathrm{d}\boldsymbol{z}),
\end{equation*}
\begin{equation*}
	L(\boldsymbol{x},\boldsymbol{\beta})\triangleq\sup_{\boldsymbol{\alpha}\in\mathbb{R}^d}\left[\boldsymbol{\alpha}\cdot\boldsymbol{\beta}-H(\boldsymbol{x},\boldsymbol{\alpha})\right].
\end{equation*}
It can be shown that $H(\boldsymbol{x},\boldsymbol{\alpha})$ is convex and analytic in $\boldsymbol{\alpha}$ at interior points of the set $\{\boldsymbol{\alpha}:H(\boldsymbol{x},\boldsymbol{\alpha})<\infty\}$, with a strictly positive definite Hessian matrix in $\boldsymbol{\alpha}$. Similarly, $L(\boldsymbol{x},\boldsymbol{\beta})$ possesses analogous properties in $\boldsymbol{\beta}$ on the interior of $\{\boldsymbol{\beta}:L(\boldsymbol{x},\boldsymbol{\beta})<\infty\}$ (cf. \cite[p. 232]{Venttsel_1977a}).

The Freidlin-Wentzell large deviation principle for $\boldsymbol{x}^{\varepsilon}(t)$ can be expressed in the following way (cf. \cite{Venttsel_1977a,Venttsel_1977b,Venttsel_1980,Venttsel_1983,Wentzell_2012}), with the associated rate function on $\mathcal{D}([0,T];\mathbb{R}^d)$ given by 

\begin{equation*}
	I_{[0,T]}\left(\left\{\boldsymbol{\phi}(t):t\in[0,T]\right\}\right)\triangleq\begin{cases}
		\int_{0}^{T}L(\boldsymbol{\phi}(t),\dot{\boldsymbol{\phi}}(t))\mathrm{d}t, &\text{if $\boldsymbol{\phi}$ is absolutely continuous}, \\ 
		\infty,&\text{otherwise}.
	\end{cases}
\end{equation*}
\begin{theorem}\label{theo020101}
	Suppose that:
	\begin{itemize}
		\item[(a)] There exists an everywhere finite non-negative convex function $\bar{H}(\boldsymbol{\alpha})$ such that $\bar{H}(\boldsymbol{0})=0$ and $H(\boldsymbol{x},\boldsymbol{\alpha})\leq\bar{H}(\boldsymbol{\alpha})$ for all $\boldsymbol{x},\,\boldsymbol{\alpha}\in\mathbb{R}^d$;
		\item[(b)] The function $L(\boldsymbol{x},\boldsymbol{\beta})$ is finite for all values of the arguments; for any positive $R$, there exist positive constants $M_R$ and $m_R$ such that for all $\boldsymbol{x},\boldsymbol{\lambda}\in\mathbb{R}^d$ and all $\boldsymbol{\beta}\in\mathbb{R}^d$ with $\vert\boldsymbol{\beta}\vert\leq{R}$,
		\begin{equation*}
			L(\boldsymbol{x},\boldsymbol{\beta})\leq M_R,\;\; \vert\nabla_{\boldsymbol{\beta}}L(\boldsymbol{x},\boldsymbol{\beta})\vert\leq M_R,\;\; \sum_{i,j}\frac{\partial^{2}L}{\partial\beta_{i}\partial\beta_{j}}(\boldsymbol{x},\boldsymbol{\beta})\lambda_{i}\lambda_{j}\geq m_R\sum_{i}\lambda_{i}^2;A
		\end{equation*}
		\item[(c)] The modulus of continuity
		\begin{equation*}
			\Delta L(\delta)\triangleq\sup_{\vert\boldsymbol{x}-\boldsymbol{y}\vert<\delta}\sup_{\boldsymbol{\beta}}\frac{L(\boldsymbol{x},\boldsymbol{\beta})-L(\boldsymbol{y},\boldsymbol{\beta})}{1+L(\boldsymbol{y},\boldsymbol{\beta})}\to{0},\quad\text{as}\;\delta\to{0}.
		\end{equation*}
	\end{itemize}
	Then the family of processes $\{\boldsymbol{x}^{\varepsilon}(t),P^{\varepsilon}_{\boldsymbol{x}_{0}}\}$ satisfies a large deviation principle with rate function $I_{[0,T]}$ (also called the normalized action functional) and normalizing coefficient $\varepsilon^{-1}$, uniformly with respect to the initial point $\boldsymbol{x}_0\in\mathbb{R}^d$. Specifically:
	\begin{itemize}
		\item[(a)] $I_{[0,T]}$ is lower semi-continuous on $\left(\mathcal{D}([0,T];\mathbb{R}^d),\rho_{[0,T]}\right)$;
		\item[(b)] The set $\bigcup_{\boldsymbol{x}_{0}\in{K}}\Phi_{\boldsymbol{x}_{0}}(s)$ is compact for any compact subset $K\subset\mathbb{R}^d$, where $\Phi_{\boldsymbol{x}_{0}}(s)$ is defined for each $\boldsymbol{x}_{0}\in\mathbb{R}^d$ and $s>0$ by
		\begin{equation*}
			\Phi_{\boldsymbol{x}_{0}}(s)\triangleq\{\boldsymbol{\phi}\in\mathcal{D}([0,T];\mathbb{R}^d):\boldsymbol{\phi}(0)=\boldsymbol{x}_{0},I_{[0,T]}\left(\boldsymbol{\phi}\right)\leq{s}\};
		\end{equation*}
		\item[(c)] For any positive $\delta,\gamma,s_{0}$, there exists $\varepsilon_{0}>0$ such that for all $\varepsilon<\varepsilon_{0}$, $\boldsymbol{x}_{0}\in\mathbb{R}^d$ and $\boldsymbol{\phi}\in\Phi_{\boldsymbol{x}_{0}}(s_{0})$,
		\begin{equation}\label{eq020101}
			P^{\varepsilon}_{\boldsymbol{x}_{0}}\left(\rho_{[0,T]}(\boldsymbol{x}^{\varepsilon},\boldsymbol{\phi})<\delta\right)\geq\exp\left(-\varepsilon^{-1}(I_{[0,T]}\left(\boldsymbol{\phi}\right)+\gamma)\right);
		\end{equation}
		\item[(d)] For any positive $\delta,\gamma,s_{0}$, there exists $\varepsilon_{0}>0$ such that for all $\varepsilon<\varepsilon_{0}$, $\boldsymbol{x}_{0}\in\mathbb{R}^d$ and $s\leq{s}_{0}$,
		\begin{equation}\label{eq020102}
			P^{\varepsilon}_{\boldsymbol{x}_{0}}\left(\rho_{[0,T]}(\boldsymbol{x}^{\varepsilon},\Phi_{\boldsymbol{x}_{0}}(s))\geq\delta\right)\leq\exp\left(-\varepsilon^{-1}(s-\gamma)\right).
		\end{equation}
	\end{itemize}
\end{theorem}
\begin{remark}\label{rem020101}
	In accordance with the preceding assumptions, it is possible to verify that: (i) $L(\boldsymbol{x},\boldsymbol{\beta})\geq{0}$; (ii) $L(\boldsymbol{x},\boldsymbol{\beta})={0}$ if and only if $\boldsymbol{\beta}=\boldsymbol{b}(\boldsymbol{x})$; (iii) $L(\boldsymbol{x},\boldsymbol{\beta})/\vert\boldsymbol{\beta}\vert\to\infty$ as $\vert\boldsymbol{\beta}\vert\to\infty$, uniformly for all $\boldsymbol{x}\in\mathbb{R}^d$ (cf. \cite[p. 231]{Venttsel_1977a}, \cite[Sections 1.1, 2.1, 3.1]{Wentzell_2012}, \cite[Sections 5.1, 5.2]{Freidlin_2012} or \cite{Rockafellar_1997}). Consequently, $I_{[0,T]}(\boldsymbol{\phi})=0$ if and only if $\dot{\boldsymbol{\phi}}(t)=\boldsymbol{b}(\boldsymbol{\phi}(t))$ for almost every $t\in[0,T]$, which implies that $\boldsymbol{\phi}(t)=\boldsymbol{x}^{0}(t)$ for all $t\in[0,T]$. By Theorem \ref{theo020101}, for any $\delta>0$, there exist constants $C>0$ and $\varepsilon_{0}>0$ such that for all $\boldsymbol{x}_{0}\in\mathbb{R}^d$ and $\varepsilon\leq\varepsilon_{0}$,
	\begin{equation*}
		P^{\varepsilon}_{\boldsymbol{x}_{0}}\left(\sup_{t\in[0,T]}\left\vert\boldsymbol{x}^{\varepsilon}(t)-\boldsymbol{x}^{0}(t)\right\vert\geq\delta\right)\leq e^{-C/\varepsilon}.
	\end{equation*} 
	Therefore, the large deviation principle can be viewed as a result that goes beyond the law of large numbers, as it not only identifies the deterministic limit $\{\boldsymbol{x}^{0}(t):t\in[0,T]\}$ to which $\{\boldsymbol{x}^\varepsilon(t):t\in[0,T]\}$ converges in probability \cite{Kurtz_1970}, but also quantifies the asymptotic probabilities of events that $\{\boldsymbol{x}^\varepsilon(t):t\in[0,T]\}$ falls in some sets of functions far away from this deterministic trajectory \cite{Freidlin_2012}. 
\end{remark}
\begin{remark}\label{rem020102}
	Assume that $\boldsymbol{b}(\boldsymbol{x}):\mathbb{R}^d\to\mathbb{R}^d$ is bounded and Lipschitz continuous. For any $\kappa\in(1,\infty)$ and $\sigma_{i}\in(0,\infty)\;(i=1,\cdots,d)$, let $\mu_{\boldsymbol{x}}$ be the $d$-dimensional generalized Gaussian measure with mean $\boldsymbol{b}(\boldsymbol{x})$, shape parameter $\kappa$, and scale parameters $(\sigma_1,\cdots,\sigma_d)$ \cite{Fabian_2009}, i.e.,
	\begin{equation*}
		\mu_{\boldsymbol{x}}(D)=\frac{\kappa^d}{(2\Gamma(1/\kappa))^d\prod_{i=1}^{d}\sigma_{i}}\int_{D}\exp\left(-\sum_{i=1}^{d}\frac{\left\vert z_i-b_i(\boldsymbol{x})\right\vert^\kappa}{\sigma_{i}^\kappa}\right)\lambda(\mathrm{d}\boldsymbol{z}),
	\end{equation*}
	where $D\subset\mathbb{R}^d$ is a Borel set, $\lambda$ denotes the Lebesgue measure, and $\Gamma$ represents the gamma function. Under this construction, the discrete dynamics in Eq. (\ref{eq010101}) can be rewritten as 
	\begin{equation*}
		\boldsymbol{x}^{\varepsilon}((n+1)\varepsilon)=\boldsymbol{x}^{\varepsilon}(n\varepsilon)+\varepsilon\boldsymbol{b}(\boldsymbol{x}^{\varepsilon}(n\varepsilon))+\varepsilon\boldsymbol{Z}_{n}^{*},
	\end{equation*}
	where $\{\boldsymbol{Z}_{n}^{*}\}$ is a sequence of independent and identically distributed $d$-dimensional random jump vectors, whose probability measure is given by
	\begin{equation*}
		\mu^{*}(D)=\frac{\kappa^d}{(2\Gamma(1/\kappa))^d\prod_{i=1}^{d}\sigma_{i}}\int_{D}\exp\left(-\sum_{i=1}^{d}\frac{\left\vert z_i\right\vert^\kappa}{\sigma_{i}^\kappa}\right)\lambda(\mathrm{d}\boldsymbol{z}).
	\end{equation*}
	The corresponding Hamiltonian and Lagrangian functions then admit the following representations:
	\begin{equation*}
		H(\boldsymbol{x},\boldsymbol{\alpha})=\boldsymbol{b}(\boldsymbol{x})\cdot\boldsymbol{\alpha}+H^{*}(\boldsymbol{\alpha}),\quad L(\boldsymbol{x},\boldsymbol{\beta})=L^{*}(\boldsymbol{\beta}-\boldsymbol{b}(\boldsymbol{x})),
	\end{equation*}
	where
	\begin{equation*}
		H^{*}(\boldsymbol{\alpha})\triangleq\ln\int_{\mathbb{R}^d}e^{\boldsymbol{\alpha}\cdot\boldsymbol{z}}\mu^{*}(\mathrm{d}\boldsymbol{z}),\quad L^{*}(\boldsymbol{\beta})\triangleq\sup_{\boldsymbol{\alpha}\in\mathbb{R}^d}[\boldsymbol{\alpha}\cdot\boldsymbol{\beta}-H^{*}(\boldsymbol{\alpha})].
	\end{equation*}
	One may verify that $H(\boldsymbol{x},\boldsymbol{\alpha})$ and $L(\boldsymbol{x},\boldsymbol{\beta})$ satisfy all the conditions of Theorem \ref{theo020101}. Hence, this defines a reasonably large class of models for which the large deviation principle holds.
	
	In particular, when $\kappa=2$, Eq. (\ref{eq010101}) becomes
	\begin{equation*}
		\boldsymbol{x}^{\varepsilon}((n+1)\varepsilon)-\boldsymbol{x}^{\varepsilon}(n\varepsilon)\overset{p}{=}\varepsilon\boldsymbol{b}(\boldsymbol{x}^{\varepsilon}(n\varepsilon))+\sqrt{\varepsilon/2}\boldsymbol{\sigma}(\boldsymbol{w}((n+1)\varepsilon)-\boldsymbol{w}(n\varepsilon)),
	\end{equation*}
	where $\overset{p}{=}$ denotes equality in finite-dimensional distributions. This corresponds precisely to the Euler discretization of the diffusion process (\ref{eq010103}) for the special case where $\boldsymbol{\sigma}=\text{diag}(\sigma_1,\cdots,\sigma_d)$. Therefore, the large deviation principle for the discrete-time Markov jump processes considered here naturally extends the corresponding principle for diffusion processes.
\end{remark}
\begin{remark}\label{rem020103}
	For simplicity, the large deviation principle for $\boldsymbol{x}^{\varepsilon}(t)$ is formulated in terms of strong assumptions analogous to those used for locally infinitely divisible processes \cite[Section 5.2]{Freidlin_2012}. These conditions can be relaxed if the finiteness of $L(\boldsymbol{x},\boldsymbol{\beta})$ for all values of the arguments or the uniformity for all initial points is not required. For a more comprehensive treatment of this topic, we refer to \cite{Venttsel_1977a,Venttsel_1977b,Venttsel_1980,Venttsel_1983,Wentzell_2012}. 
\end{remark}
\subsection{Optimal Fluctuations on Finite Time Intervals}\label{sec0202}
The following proposition follows straightforwardly from the contraction principle \cite[Chapter 3, Theorem 3.1]{Freidlin_2012} and Borovkov's characterization of the large deviation principle \cite[Chapter 3, Theorem 3.4]{Freidlin_2012}.
\begin{proposition}\label{theo020201}
	Suppose that the conditions in Theorem \ref{theo020101} hold. For each $\boldsymbol{x}_0,\,\boldsymbol{x}_T\in\mathbb{R}^d$ and $T>0$, define
	\begin{equation}\label{eq020201}
		S(\boldsymbol{x}_T,T\vert\boldsymbol{x}_0)\triangleq\inf_{\boldsymbol{\phi}(0)=\boldsymbol{x}_0,\;\boldsymbol{\phi}(T)=\boldsymbol{x}_T}I_{[0,T]}\left(\left\{\boldsymbol{\phi}(t):t\in[0,T]\right\}\right).
	\end{equation}
	Then, for any $\boldsymbol{x}_{0}\in\mathbb{R}^d$ and any Borel set $D\subset\mathbb{R}^d$ satisfying
	\begin{equation*}
		\inf_{\boldsymbol{x}_T\in\text{cl}(D)}S(\boldsymbol{x}_T,T\vert\boldsymbol{x}_0)=\inf_{\boldsymbol{x}_T\in\text{int}(D)}S(\boldsymbol{x}_T,T\vert\boldsymbol{x}_0),
	\end{equation*}
	we have
	\begin{equation}\label{eq020202}
		\lim_{\varepsilon\to{0}}\varepsilon\ln{P}^{\varepsilon}_{\boldsymbol{x}_0}\left(\boldsymbol{x}^{\varepsilon}(T)\in D\right)=-\inf_{\boldsymbol{x}_T\in{D}}S(\boldsymbol{x}_T,T\vert\boldsymbol{x}_0),
	\end{equation}
	where $\text{cl}(D)$ and $\text{int}(D)$ are the closure and interior of $D$, respectively.
\end{proposition} 
The concept of a non-stationary optimal path is articulated as follows.
\begin{definition}\label{def020201}
	For given $\boldsymbol{x}_0, \boldsymbol{x}_T\in\mathbb{R}^d$ and $T>0$, a path $\{\boldsymbol{\phi}(t):t\in[0,T]\}$ is said to be a \textbf{non-stationary optimal path (NOP)} that connects $\boldsymbol{x}_0$ with $\boldsymbol{x}_T$ in the time span $T$ if it is a minimizer of the variational problem in Eq. (\ref{eq020201}). That is, it satisfies $\boldsymbol{\phi}(0)=\boldsymbol{x}_0$, $\boldsymbol{\phi}(T)=\boldsymbol{x}_T$, and
	\begin{equation*}
		S(\boldsymbol{x}_T,T\vert\boldsymbol{x}_0)=I_{[0,T]}\left(\left\{\boldsymbol{\phi}(t):t\in[0,T]\right\}\right)<\infty.
	\end{equation*}
	We denote such a path by $\{\boldsymbol{\phi}_{\text{NOP}}(t;\boldsymbol{x}_T,T;\boldsymbol{x}_0):t\in[0,T]\}$ when it exists.
\end{definition}
We now summarize several properties required for the subsequent analysis. Detailed proofs of these results can be found in \cite[Appendix A]{Zhao_2025}.
\begin{proposition}\label{theo020202}
	Suppose that the conditions in Theorem \ref{theo020101} hold. Then:
	\begin{itemize}
		\item[(a)] $S(\boldsymbol{x}_T,T\vert\boldsymbol{x}_0)<\infty$ for all $\boldsymbol{x}_0, \boldsymbol{x}_T\in\mathbb{R}^d$ and $T>0$; That is, there exists at least one (possibly not unique) NOP that connects $\boldsymbol{x}_0$ with $\boldsymbol{x}_T$ in the time span $T$;
		\item[(b)] For fixed $T>0$, $S(\boldsymbol{x}_T,T\vert\boldsymbol{x}_0)\to\infty$ as $\vert\boldsymbol{x}_T-\boldsymbol{x}_0\vert\to\infty$; For any $\boldsymbol{x}_0,\boldsymbol{x}_T\in\mathbb{R}^d$ satisfying $\boldsymbol{x}_0\neq\boldsymbol{x}_T$, $S(\boldsymbol{x}_T,T\vert\boldsymbol{x}_0)\to\infty$ as $T\to{0}$;
		\item[(c)] If $\{\boldsymbol{\phi}(t):t\in[0,T]\}$ is an (respectively, the unique) NOP that connects $\boldsymbol{x}_0$ with $\boldsymbol{x}_T$ in the time span $T$, then for any $t_1,t_2\in[0,T]$ with $t_1<t_2$, the subpath $\{\boldsymbol{\phi}(t+t_1):t\in[0,t_2-t_1]\}$ is also an (respectively, the unique) NOP that connects $\boldsymbol{\phi}(t_1)$ with $\boldsymbol{\phi}(t_2)$ in the time span $t_2-t_1$.
	\end{itemize}
	If, in addition, $H(\boldsymbol{x},\boldsymbol{\alpha})$ is $\mathcal{C}^{k+1}$ (i.e., a function possessing continuous partial derivatives up to order $k+1$ inclusive) for all values of the arguments with some integer $k\geq 1$, then:
	\begin{itemize}
		\item[(d)] For each NOP $\{\boldsymbol{\phi}_{\text{NOP}}(t;\boldsymbol{x}_T,T;\boldsymbol{x}_0):t\in[0,T]\}$, the functions
		\begin{equation*}
			\left\{\begin{aligned}
				&\boldsymbol{x}(t)\triangleq\boldsymbol{\phi}_{\text{NOP}}(t;\boldsymbol{x}_T,T;\boldsymbol{x}_0),\\
				&\boldsymbol{\alpha}(t)\triangleq\nabla_{\boldsymbol{\beta}}L(\boldsymbol{\phi}_{\text{NOP}}(t;\boldsymbol{x}_T,T;\boldsymbol{x}_0),\dot{\boldsymbol{\phi}}_{\text{NOP}}(t;\boldsymbol{x}_T,T;\boldsymbol{x}_0)),
			\end{aligned}\right.
		\end{equation*}
		are $\mathcal{C}^{k+1}$ and satisfy the following Hamilton's system of equations
		\begin{equation}\label{eq020203}
			\left\{\begin{aligned}
				&\dot{\boldsymbol{x}}(t)=\nabla_{\boldsymbol{\alpha}}H(\boldsymbol{x}(t),\boldsymbol{\alpha}(t)),\\
				&\dot{\boldsymbol{\alpha}}(t)=-\nabla_{\boldsymbol{x}}H(\boldsymbol{x}(t),\boldsymbol{\alpha}(t)),
			\end{aligned}\right.			
		\end{equation}
		with initial condition $\boldsymbol{x}(0)=\boldsymbol{x}_0$ and terminal condition $\boldsymbol{x}(T)=\boldsymbol{x}_T$. Consequently, for each $t\in[0,T]$,
		\begin{equation}\label{eq020204}
			S(\boldsymbol{x}(t),t\vert\boldsymbol{x}_0)=\int_{0}^{t}\left[\dot{\boldsymbol{x}}(u)\cdot\boldsymbol{\alpha}(u)-H(\boldsymbol{x}(u),\boldsymbol{\alpha}(u))\right]\mathrm{d}u,
		\end{equation}
		and
		\begin{equation}\label{eq020205}
			S(\boldsymbol{x}_T,T-t\vert\boldsymbol{x}(t))=\int_{t}^{T}\left[\dot{\boldsymbol{x}}(u)\cdot\boldsymbol{\alpha}(u)-H(\boldsymbol{x}(u),\boldsymbol{\alpha}(u))\right]\mathrm{d}u.
		\end{equation}
		\item[(e)] Suppose that the NOP $\{\boldsymbol{\phi}_{\text{NOP}}(t;\boldsymbol{x}_T,T;\boldsymbol{x}_0):t\in[0,T]\}$ is a proper subarc of another NOP, i.e., there exist a constant $\gamma_0>0$ and an NOP $\{\bar{\boldsymbol{\phi}}(t):t\in[0,T+2\gamma_0]\}$ such that $\boldsymbol{\phi}_{\text{NOP}}(t;\boldsymbol{x}_T,T;\boldsymbol{x}_0)=\bar{\boldsymbol{\phi}}(t+\gamma_0)$ for all $t\in[0,T]$. Let $B_{\delta_0}(\boldsymbol{x})$ denote the open $\delta_0$-neighborhood of $\boldsymbol{x}$ in $\mathbb{R}^d$. For any $t_1,t_2\in[0,T]$ with $t_1<t_2$, define
		\begin{equation*}
			\Omega_{\delta_0,[t_1,t_2]}\triangleq\left\{(\boldsymbol{z},t)\in\mathbb{R}^d\times[t_1,t_2]:\boldsymbol{z}\in{B}_{\delta_0}(\boldsymbol{\phi}_{\text{NOP}}(t;\boldsymbol{x}_T,T;\boldsymbol{x}_0))\right\}.
		\end{equation*}
		Then, for each $T^*\in(0,T)$, there exists a constant $\delta_0$ such that $S(\boldsymbol{x},t\vert\boldsymbol{x}_0)$ and $S(\boldsymbol{x}_T,T-t\vert\boldsymbol{x})$ are $\mathcal{C}^{k+1}$ with respect to $(\boldsymbol{x},t)$ on the sets $\Omega_{\delta_0,[T-T^*,T]}$ and $\Omega_{\delta_0,[0,T^*]}$ respectively. In this case, they satisfy the Hamilton-Jacobi equations
		\begin{equation}\label{eq020206}
			\frac{\partial}{\partial{t}}S(\boldsymbol{x},t\vert\boldsymbol{x}_0)+H(\boldsymbol{x},\nabla_{\boldsymbol{x}}S(\boldsymbol{x},t\vert\boldsymbol{x}_0))=0,
		\end{equation} 
		and 
		\begin{equation}\label{eq020207}
			\frac{\partial}{\partial{t}}S(\boldsymbol{x}_T,T-t\vert\boldsymbol{x})-H(\boldsymbol{x},-\nabla_{\boldsymbol{x}}{S}(\boldsymbol{x}_T,T-t\vert\boldsymbol{x}))=0,
		\end{equation}
		in the classical sense at each point $(\boldsymbol{x},t)$ in $\Omega_{\delta_0,[T-T^*,T]}$ and $\Omega_{\delta_0,[0,T^*]}$, respectively. The NOP is the unique solution of
		\begin{equation}\label{eq020208}
			\dot{\boldsymbol{x}}(t)=\nabla_{\boldsymbol{\alpha}}H(\boldsymbol{x}(t),\nabla_{\boldsymbol{x}}S(\boldsymbol{x}(t),t\vert\boldsymbol{x}_0)),\;\;t\in[T-T^*,T],
		\end{equation}
		with terminal condition $\boldsymbol{x}(T)=\boldsymbol{x}_T$, or equivalently
		\begin{equation}\label{eq020209}
			\dot{\boldsymbol{x}}(t)=\nabla_{\boldsymbol{\alpha}}H(\boldsymbol{x}(t),-\nabla_{\boldsymbol{x}}S(\boldsymbol{x}_T,T-t\vert\boldsymbol{x}(t))),\;\;t\in[0,T^*],
		\end{equation}
		with initial condition $\boldsymbol{x}(0)=\boldsymbol{x}_0$. Moreover, for any $(\boldsymbol{x},t)\in\Omega_{\delta_0,[T-T^*,T]}$ and sufficiently small $\delta>0$,
		\begin{equation}\label{eq020210}
			\lim_{\varepsilon\to{0}}\varepsilon\ln{P^{\varepsilon}_{\boldsymbol{x}_0}\left(\boldsymbol{x}^{\varepsilon}(t)\in{B}_{\delta}(\boldsymbol{x})\right)}=-\inf_{\boldsymbol{y}\in{B}_{\delta}(\boldsymbol{x})}S(\boldsymbol{y},t\vert\boldsymbol{x}_0).
		\end{equation}
		Similarly, for any $(\boldsymbol{x},t)\in\Omega_{\delta_0,[0,T^*]}$ and sufficiently small $\delta>0$,
		\begin{equation}\label{eq020211}
			\lim_{\varepsilon\to{0}}\varepsilon\ln{P^{\varepsilon}_{\boldsymbol{x}_0}\left(\boldsymbol{x}^{\varepsilon}(T)\in{B}_{\delta}(\boldsymbol{x}_T)\vert\boldsymbol{x}^{\varepsilon}(t)=\boldsymbol{x}\right)}=-\inf_{\boldsymbol{y}\in{B}_{\delta}(\boldsymbol{x}_T)}S(\boldsymbol{y},T-t\vert\boldsymbol{x}).
		\end{equation}
	\end{itemize}
\end{proposition}
\begin{remark}\label{rem020201}
	The function $S(\boldsymbol{x},t\vert\boldsymbol{y})$ is discontinuous at $t=0$, as $S(\boldsymbol{x},0\vert\boldsymbol{y})=0$ if $\boldsymbol{x}=\boldsymbol{y}$ and $S(\boldsymbol{x},0\vert\boldsymbol{y})=\infty$ otherwise. The parameter $T^{*}$ is introduced to circumvent this discontinuity.
	
	If the assumptions in part (e) are satisfied, then the NOP $\{\boldsymbol{\phi}_{\text{NOP}}(t;\boldsymbol{x}_T,T;\boldsymbol{x}_0):t\in[0,T]\}$ is the unique path connecting $\boldsymbol{x}_0$ with $\boldsymbol{x}_T$ in the time span $T$. The converse, however, is not true in general. Hence, the conditions in part (e) are slightly stronger than those required solely for the uniqueness of the NOP.
	
	Part (e) of this proposition provides a sufficient condition under which the segments $\{\boldsymbol{\phi}_{\text{NOP}}(t;\boldsymbol{x}_T,T;\boldsymbol{x}_0):t\in[T-T^*,T]\}$ and $\{\boldsymbol{\phi}_{\text{NOP}}(t;\boldsymbol{x}_T,T;\boldsymbol{x}_0):t\in[0,T^*]\}$ can be characterized as the unique solutions of Eqs. (\ref{eq020208}) and (\ref{eq020209}), respectively. This characterization is a prerequisite for the limit theorems established in Sec. \ref{sec04}.
\end{remark}

The following lemma is used in Sec. \ref{sec04}. Its proof is given in \ref{secA1}.
\begin{lemma}\label{theo020203}
	Suppose that the conditions in Theorem \ref{theo020101} and Proposition \ref{theo020202}(e) hold. For each $T^{*}\in(0,T)$ and a sufficiently small $\delta_{0}>0$, define 
	\begin{equation*}
		\Sigma_{\delta_{0},[T-T^{*},T]}\triangleq\bigcup_{t\in[T-T^{*},T]}B_{\delta_{0}}(\boldsymbol{\phi}_{\text{NOP}}(t;\boldsymbol{x}_T,T;\boldsymbol{x}_0)).
	\end{equation*}
	Assume that for any $T^*\in(0,T)$, there exists a constant $\delta_0>0$, so that for all sufficiently small $\delta$ and all $t\in[T-T^{*},T]$, the prefactor
	\begin{equation*}
		k^{\varepsilon,\delta}(\boldsymbol{x},t\vert\boldsymbol{x}_0,0)\triangleq{P}^\varepsilon_{\boldsymbol{x}_0}\left(\boldsymbol{x}^{\varepsilon}(t)\in{B}_{\delta}(\boldsymbol{x})\right)\exp\left(\varepsilon^{-1}\inf_{\boldsymbol{y}\in{B}_{\delta}(\boldsymbol{x})}S(\boldsymbol{y},t\vert\boldsymbol{x}_0)\right),
	\end{equation*}
	is continuous on $\Sigma_{\delta_{0},[T-T^{*},T]}$, and that there exists a continuous function $f^{\varepsilon,\delta}$ for which the limit
	\begin{equation*}
		\lim_{\varepsilon\to{0},\delta\to{0}}\frac{k^{\varepsilon,\delta}(\boldsymbol{x},t\vert\boldsymbol{x}_0,0)}{f^{\varepsilon,\delta}}=K(\boldsymbol{x},t\vert\boldsymbol{x}_0),
	\end{equation*}
	exists, uniformly for $(\boldsymbol{x},t)\in\Sigma_{\delta_{0},[T-T^{*},T]}\times[T-T^{*},T]$. We further assume that both $K(\boldsymbol{x},t\vert\boldsymbol{x}_0)$ and $S(\boldsymbol{x},t\vert\boldsymbol{x}_0)$ are at least twice continuously differentiable on $\Sigma_{\delta_{0},[T-T^{*},T]}\times[T-T^{*},T]$ , and that $K(\boldsymbol{x},t\vert\boldsymbol{x}_0)>0$. Then 
	\begin{equation}\label{eq020212}
		\frac{{P}^\varepsilon_{\boldsymbol{x}_0}\left(\boldsymbol{x}^{\varepsilon}(t-\varepsilon)\in{B}_{{\delta}}(\boldsymbol{x}+\varepsilon\boldsymbol{z})\right)}{{P}^\varepsilon_{\boldsymbol{x}_0}\left(\boldsymbol{x}^{\varepsilon}(t)\in{B}_{{\delta}}(\boldsymbol{x})\right)}\to\frac{\exp\left(-\boldsymbol{z}\cdot\nabla_{\boldsymbol{x}}S(\boldsymbol{x},t\vert\boldsymbol{x}_0)\right)}{\exp\left(-\frac{\partial}{\partial{t}}S(\boldsymbol{x},t\vert\boldsymbol{x}_0)\right)},
	\end{equation}
	as $\varepsilon\to{0}$ and $\varepsilon^{-1}\delta^2\to{0}$, uniformly for $(\boldsymbol{x},t)\in\Sigma_{\delta_{0},[T-T^{*},T]}\times[T-T^{*},T]$ and $\boldsymbol{z}$ in any compact subset of $\mathbb{R}^d$.
	
	Moreover, if for each $t\in[T-T^{*},T]$, the probability $P_{\boldsymbol{x}_0}^\varepsilon(\boldsymbol{x}^{\varepsilon}(t)\in{\cdot})$ admits a continuous density function $p^{\varepsilon}(\boldsymbol{x},t\vert\boldsymbol{x}_0,0)$ such that $p^{\varepsilon}(\boldsymbol{x},t\vert\boldsymbol{x}_0,0)>0$ and
	\begin{equation*}
		\lim_{\delta\to{0}}\frac{P_{\boldsymbol{x}_0}^\varepsilon(\boldsymbol{x}^{\varepsilon}(t)\in{B}_{\delta}(\boldsymbol{x}))}{\lambda({B}_{\delta}(\boldsymbol{x}))}=p^{\varepsilon}(\boldsymbol{x},t\vert\boldsymbol{x}_0,0),
	\end{equation*}
	uniformly for $(\boldsymbol{x},t)\in\Sigma_{\delta_{0},[T-T^{*},T]}\times[T-T^{*},T]$, then
	\begin{equation}\label{eq020213}
		\lim_{\varepsilon\to{0}}\frac{p^{\varepsilon}(\boldsymbol{x}+\varepsilon\boldsymbol{z},t-\varepsilon\vert\boldsymbol{x}_0,0)}{p^{\varepsilon}(\boldsymbol{x},t\vert\boldsymbol{x}_0,0)}=\frac{\exp\left(-\boldsymbol{z}\cdot\nabla_{\boldsymbol{x}}S(\boldsymbol{x},t\vert\boldsymbol{x}_0)\right)}{\exp\left(-\frac{\partial}{\partial{t}}S(\boldsymbol{x},t\vert\boldsymbol{x}_0)\right)},
	\end{equation}
	uniformly for $(\boldsymbol{x},t)\in\Sigma_{\delta_{0},[T-T^{*},T]}\times[T-T^{*},T]$ and $\boldsymbol{z}$ in any compact subset of $\mathbb{R}^d$.
\end{lemma}
Similarly, the following result can be obtained in essentially the same way as Lemma \ref{theo020203}, and its proof is therefore omitted.
\begin{lemma}\label{theo020204}
	 Assume that the conditions in Theorem \ref{theo020101} and Proposition \ref{theo020202}(e) hold. For each $T^{*}\in(0,T)$ and a sufficiently small $\delta_{0}>0$, define 
	\begin{equation*}
		\Sigma_{\delta_{0},[0,T^*]}\triangleq\bigcup_{t\in[0,T^*]}B_{\delta_{0}}(\boldsymbol{\phi}_{\text{NOP}}(t;\boldsymbol{x}_T,T;\boldsymbol{x}_0)).
	\end{equation*}
	Assume that for any $T^*\in(0,T)$, there exists a constant $\delta_0>0$, so that for all sufficiently small $\delta$ and all $t\in[0,T^*]$, the prefactor
	\begin{equation*}
		\begin{aligned}
			k^{\varepsilon,\delta}(\boldsymbol{x}_T,T\vert\boldsymbol{x},t)\triangleq{P}^\varepsilon_{\boldsymbol{x}_0}(\boldsymbol{x}^{\varepsilon}(T)&\in{B}_{\delta}(\boldsymbol{x}_T)\vert\boldsymbol{x}^{\varepsilon}(t)=\boldsymbol{x})\\
			&\times\exp\left(\varepsilon^{-1}\inf_{\boldsymbol{y}\in{B}_{\delta}(\boldsymbol{x}_T)}S(\boldsymbol{y},T-t\vert\boldsymbol{x})\right),
		\end{aligned}
	\end{equation*}
	is continuous on $\Sigma_{\delta_{0},[0,T^*]}$, and that there exists a continuous function $f^{\varepsilon,\delta}$ for which the limit
	\begin{equation*}
		\lim_{\varepsilon\to{0},\delta\to{0}}\frac{k^{\varepsilon,\delta}(\boldsymbol{x}_T,T\vert\boldsymbol{x},t)}{f^{\varepsilon,\delta}}=K(\boldsymbol{x}_T,T-t\vert\boldsymbol{x}),
	\end{equation*}
	exists, uniformly for $(\boldsymbol{x},t)\in\Sigma_{\delta_{0},[0,T^*]}\times[0,T^*]$. We further assume that both $K(\boldsymbol{x}_T,T-t\vert\boldsymbol{x})$ and $S(\boldsymbol{x}_T,T-t\vert\boldsymbol{x})$ are at least twice continuously differentiable on $\Sigma_{\delta_{0},[0,T^*]}\times[0,T^*]$, and that $K(\boldsymbol{x}_T,T-t\vert\boldsymbol{x})>0$. Then
		\begin{equation}\label{eq020214}
		\frac{{P}^\varepsilon_{\boldsymbol{x}_0}\left(\boldsymbol{x}^{\varepsilon}(T)\in{B}_{{\delta}}(\boldsymbol{x}_T)\vert\boldsymbol{x}^{\varepsilon}(t+\varepsilon)=\boldsymbol{x}+\varepsilon\boldsymbol{z}\right)}{{P}^\varepsilon_{\boldsymbol{x}_0}\left(\boldsymbol{x}^{\varepsilon}(T)\in{B}_{{\delta}}(\boldsymbol{x}_T)\vert\boldsymbol{x}^{\varepsilon}(t)=\boldsymbol{x}\right)}\to\frac{\exp\left(-\boldsymbol{z}\cdot\nabla_{\boldsymbol{x}}S(\boldsymbol{x}_T,T-t\vert\boldsymbol{x})\right)}{\exp\left(\frac{\partial}{\partial{t}}S(\boldsymbol{x}_T,T-t\vert\boldsymbol{x})\right)},
	\end{equation}
	as $\varepsilon\to{0}$ and $\varepsilon^{-1}\delta^2\to{0}$, uniformly for $(\boldsymbol{x},t)\in\Sigma_{\delta_{0},[0,T^*]}\times[0,T^*]$ and $\boldsymbol{z}$ in any compact subset of $\mathbb{R}^d$. 
	
	Moreover, if for each $t\in[0,T^*]$, the conditional probability $P_{\boldsymbol{x}_0}^\varepsilon(\boldsymbol{x}^{\varepsilon}(T)\in{\cdot}\vert\boldsymbol{x}^{\varepsilon}(t)=\boldsymbol{x})$ admits a continuous density function $p^{\varepsilon}(\boldsymbol{x}_T,T\vert\boldsymbol{x},t)$ such that $p^{\varepsilon}(\boldsymbol{x}_T,T\vert\boldsymbol{x},t)>0$, and
	\begin{equation*}
		\lim_{\delta\to{0}}\frac{P_{\boldsymbol{x}_0}^\varepsilon(\boldsymbol{x}^{\varepsilon}(T)\in{B}_{\delta}(\boldsymbol{x}_T)\vert\boldsymbol{x}^{\varepsilon}(t)=\boldsymbol{x})}{\lambda({B}_{\delta}(\boldsymbol{x}_T))}=p^{\varepsilon}(\boldsymbol{x}_T,T\vert\boldsymbol{x},t),
	\end{equation*}
	uniformly for $(\boldsymbol{x},t)\in\Sigma_{\delta_{0},[0,T^*]}\times[0,T^*]$, then
	\begin{equation}\label{eq020215}
		\lim_{\varepsilon\to{0}}\frac{p^{\varepsilon}(\boldsymbol{x}_T,T\vert\boldsymbol{x}+\varepsilon\boldsymbol{z},t+\varepsilon)}{p^{\varepsilon}(\boldsymbol{x}_T,T\vert\boldsymbol{x},t)}=\frac{\exp\left(-\boldsymbol{z}\cdot\nabla_{\boldsymbol{x}}S(\boldsymbol{x}_T,T-t\vert\boldsymbol{x})\right)}{\exp\left(\frac{\partial}{\partial{t}}S(\boldsymbol{x}_T,T-t\vert\boldsymbol{x})\right)},
	\end{equation}
	uniformly for $(\boldsymbol{x},t)\in\Sigma_{\delta_{0},[0,T^*]}\times[0,T^*]$ and $\boldsymbol{z}$ in any compact subset of $\mathbb{R}^d$.
\end{lemma}
\begin{remark}\label{rem020202}
	Note that for each $0<\varepsilon\ll{1}$, the identity $\boldsymbol{x}^{\varepsilon}(t)\equiv\boldsymbol{x}^{\varepsilon}(\lfloor{t/\varepsilon}\rfloor\varepsilon)$ holds for all $t\in[0,T]$. Consequently, for fixed $\boldsymbol{x}$ and $\boldsymbol{x}_{s}$, the quantity $p^{\varepsilon}(\boldsymbol{x},t\vert\boldsymbol{x}_{s},s)$ depends only on $(\lfloor{t/\varepsilon}\rfloor-\lfloor{s/\varepsilon}\rfloor)\varepsilon$, rather than on $t-s$. This explains the use of $k^{\varepsilon,\delta}(\boldsymbol{x},t\vert\boldsymbol{x}_{s},s)$ and $p^{\varepsilon}(\boldsymbol{x},t\vert\boldsymbol{x}_{s},s)$ in Lemmas \ref{theo020203} and \ref{theo020204}, instead of the forms $k^{\varepsilon,\delta}(\boldsymbol{x},t-s\vert\boldsymbol{x}_{s})$ and $p^{\varepsilon}(\boldsymbol{x},t-s\vert\boldsymbol{x}_{s})$. Furthermore, since $(\lfloor{t/\varepsilon}\rfloor-\lfloor{s/\varepsilon}\rfloor)\varepsilon\to t-s$ as $\varepsilon\to{0}$, the appearance of  $K(\boldsymbol{x},t-s\vert\boldsymbol{x}_s)$ and $S(\boldsymbol{x},t-s\vert\boldsymbol{x}_s)$ in these lemmas is justified.
	
	In Lemmas \ref{theo020203} and \ref{theo020204}, we assume that $S(\boldsymbol{x},t\vert\boldsymbol{x}_0)$ and $S(\boldsymbol{x}_T,T-t\vert\boldsymbol{x})$ are at least twice continuously differentiable on $\Sigma_{\delta_{0},[T-T^{*},T]}\times[T-T^*,T]$ and $\Sigma_{\delta_{0},[0,T^*]}\times[0,T^*]$, respectively. These assumptions are stronger than those in Proposition \ref{theo020202}(e), since $\Omega_{\delta_{0},[T-T^*,T]}\subset\Sigma_{\delta_{0},[T-T^{*},T]}\times[T-T^*,T]$ and $\Omega_{\delta_{0},[0,T^*]}\subset\Sigma_{\delta_{0},[0,T^*]}\times[0,T^*]$. The time intervals are restricted to $[T-T^*,T]$ and $[0,T^*]$ because $p^{\varepsilon}(\boldsymbol{x},t\vert\boldsymbol{x}_0)$ and $p^{\varepsilon}(\boldsymbol{x}_T,T\vert\boldsymbol{x},t)$ are singular at $t=0$ and $t=T$, respectively. The spatial domains are confined to $\Sigma_{\delta_{0},[T-T^{*},T]}$ and $\Sigma_{\delta_{0},[0,T^*]}$ as they constitute the fundamental conditions required for the limit theorems in Sec. \ref{sec04}.
\end{remark}


\section{Time Reversal of Discrete-time Markov Jump Processes}\label{Sec03}
For each $t\in[0,T]$, assume that $\boldsymbol{x}^\varepsilon(t)$ possesses a probability density $p^\varepsilon(\boldsymbol{x},t)$. The time evolution of $p^\varepsilon(\boldsymbol{x},t)$ is described by
\begin{equation}\label{eq030001}
	\left\{\begin{aligned}
		&p^\varepsilon(\boldsymbol{x},t)=p^\varepsilon(\boldsymbol{x},n\varepsilon), &&t\in[n\varepsilon,(n+1)\varepsilon),\\
		&p^\varepsilon(\boldsymbol{x},t)=p^\varepsilon(\boldsymbol{x},\lfloor{T/\varepsilon}\rfloor\varepsilon), &&t\in[\lfloor{T/\varepsilon}\rfloor\varepsilon,T],
	\end{aligned}\right.
\end{equation}
and the Kolmogorov forward equation:
\begin{equation}\label{eq030002}
	p^\varepsilon(\boldsymbol{x},(n+1)\varepsilon)-p^\varepsilon(\boldsymbol{x},n\varepsilon)=\mathscr{J}_{\boldsymbol{x}}^{\varepsilon,*}p^\varepsilon(\boldsymbol{x},n\varepsilon),
\end{equation}
for $n=0,1,\cdots,\lfloor{T/\varepsilon}\rfloor-1$, where the operator $\mathscr{J}_{\boldsymbol{x}}^{\varepsilon,*}$ is defined for each bounded continuous function $f(\boldsymbol{x})$ by
\begin{equation*}
	\mathscr{J}_{\boldsymbol{x}}^{\varepsilon,*}f(\boldsymbol{x})\triangleq\int_{\mathbb{R}^d}\left[f(\boldsymbol{x}-\varepsilon\boldsymbol{z})\mu_{\boldsymbol{x}-\varepsilon\boldsymbol{z}}(\mathrm{d}\boldsymbol{z})-f(\boldsymbol{x})\mu_{\boldsymbol{x}}(\mathrm{d}\boldsymbol{z})\right].
\end{equation*}
In particular, the conditional probability density $p^\varepsilon(\boldsymbol{x},t\vert\boldsymbol{x}_s,s)$ ($0\leq{s}\leq{t}\leq{T}$) satisfies
\begin{equation}\label{eq030003}
	\left\{\begin{aligned}
		&p^\varepsilon(\boldsymbol{x},t\vert\boldsymbol{x}_s,s)=\delta(\boldsymbol{x}-\boldsymbol{x}_s), && t\in[s,(\lfloor{s/\varepsilon}\rfloor+1)\varepsilon),\\
		&p^\varepsilon(\boldsymbol{x},t\vert\boldsymbol{x}_s,s)=p^\varepsilon(\boldsymbol{x},n\varepsilon\vert\boldsymbol{x}_s,s), && t\in[n\varepsilon,(n+1)\varepsilon),\\
		&p^\varepsilon(\boldsymbol{x},t\vert\boldsymbol{x}_s,s)=p^\varepsilon(\boldsymbol{x},\lfloor{T/\varepsilon}\rfloor\varepsilon\vert\boldsymbol{x}_s,s), && t\in[\lfloor{T/\varepsilon}\rfloor\varepsilon,T],
	\end{aligned}\right.
\end{equation}
and
\begin{equation}\label{eq030004}
	p^\varepsilon(\boldsymbol{x},(n+1)\varepsilon\vert\boldsymbol{x}_s,s)-p^\varepsilon(\boldsymbol{x},n\varepsilon\vert\boldsymbol{x}_s,s)=\mathscr{J}_{\boldsymbol{x}}^{\varepsilon,*}p^\varepsilon(\boldsymbol{x},n\varepsilon\vert\boldsymbol{x}_s,s),
\end{equation}
for $n=\lfloor{s/\varepsilon}\rfloor+1,\cdots,\lfloor{T/\varepsilon}\rfloor-1$. Moreover, it also satisfies
\begin{equation}\label{eq030005}
	\left\{\begin{aligned}
		&p^\varepsilon(\boldsymbol{x},t\vert\boldsymbol{x}_s,s)=p^\varepsilon(\boldsymbol{x},t\vert\boldsymbol{x}_s,n\varepsilon), && s\in[n\varepsilon,(n+1)\varepsilon),\\
		&p^\varepsilon(\boldsymbol{x},t\vert\boldsymbol{x}_s,s)=\delta(\boldsymbol{x}_s-\boldsymbol{x}), && s\in[\lfloor{t/\varepsilon}\rfloor\varepsilon,t],\\
	\end{aligned}\right.
\end{equation}
and the Kolmogorov backward equation
\begin{equation}\label{eq030006}
	p^\varepsilon(\boldsymbol{x},t\vert\boldsymbol{x}_s,n\varepsilon)-p^\varepsilon(\boldsymbol{x},t\vert\boldsymbol{x}_s,(n+1)\varepsilon)=\mathscr{J}_{\boldsymbol{x}_s}^\varepsilon{p}^\varepsilon(\boldsymbol{x},t\vert\boldsymbol{x}_s,(n+1)\varepsilon),
\end{equation}
for $n=\lfloor{t/\varepsilon}\rfloor-1,\cdots,1,0$, where $\mathscr{J}^\varepsilon_{\boldsymbol{x}}$ is the adjoint operator of $\mathscr{J}_{\boldsymbol{x}}^{\varepsilon,*}$ and is given by
\begin{equation*}
	\mathscr{J}^\varepsilon_{\boldsymbol{x}}f(\boldsymbol{x})\triangleq\int_{\mathbb{R}^d}\left[f(\boldsymbol{x}+\varepsilon\boldsymbol{z})-f(\boldsymbol{x})\right]\mu_{\boldsymbol{x}}(\mathrm{d}\boldsymbol{z}).
\end{equation*}
\subsection{Time Reversal of a Given Family of Probability Densities}\label{sec0301}
For a given family of densities $\{p^\varepsilon(\boldsymbol{x},t)\}_{t\in[0,T]}$ satisfying equations (\ref{eq030001}) and (\ref{eq030002}), we call $\{\bar{p}^\varepsilon(\boldsymbol{x},t)\triangleq{p}^\varepsilon(\boldsymbol{x},T-t)\}_{t\in[0,T]}$ the time reversal of $\{p^\varepsilon(\boldsymbol{x},t)\}_{t\in[0,T]}$. The following proposition is straightforward to verify. 
\begin{proposition}\label{theo030101}
	The time evolution of $\bar{p}^\varepsilon(\boldsymbol{x},t)$ is described by
	\begin{equation}\label{eq030101}
		\left\{\begin{aligned}
			&\bar{p}^\varepsilon(\boldsymbol{x},t)=\bar{p}^\varepsilon(\boldsymbol{x},\Delta), && t\in[0,\Delta],\\
			&\bar{p}^\varepsilon(\boldsymbol{x},t)=\bar{p}^\varepsilon(\boldsymbol{x},\Delta+(m+1)\varepsilon), && t\in(\Delta+m\varepsilon,\Delta+(m+1)\varepsilon],
		\end{aligned}\right.
	\end{equation}
	and the Kolmogorov forward equation
	\begin{equation}\label{eq030102}
		\bar{p}^\varepsilon(\boldsymbol{x},\Delta+(m+1)\varepsilon)-\bar{p}^\varepsilon(\boldsymbol{x},\Delta+m\varepsilon)=\mathscr{K}^{\varepsilon,*}_{\boldsymbol{x},\Delta+m\varepsilon}\bar{p}^\varepsilon(\boldsymbol{x},\Delta+m\varepsilon),
	\end{equation}
	where $\Delta=T-\lfloor{T/\varepsilon}\rfloor\varepsilon$, $m=0,1,\cdots,\lfloor{T/\varepsilon}\rfloor-1$, and the operator $\mathscr{K}^{\varepsilon,*}_{\boldsymbol{x},t}$ is defined for each $t=\Delta+m\varepsilon\in[0,T-\varepsilon]$ and bounded continuous function $f(\boldsymbol{x},t)$ by
	\begin{equation*}
		\mathscr{K}^{\varepsilon,*}_{\boldsymbol{x},t}f(\boldsymbol{x},t)\triangleq\int_{\mathbb{R}^d}\left[f(\boldsymbol{x}-\varepsilon\boldsymbol{z},t)\bar{\mu}^\varepsilon_{\boldsymbol{x}-\varepsilon\boldsymbol{z},t}(\mathrm{d}\boldsymbol{z})-f(\boldsymbol{x},t)\bar{\mu}^\varepsilon_{\boldsymbol{x},t}(\mathrm{d}\boldsymbol{z})\right],
	\end{equation*}
	with
	\begin{equation}\label{eq030103}
		\bar{\mu}^\varepsilon_{\boldsymbol{x},t}(\mathrm{d}\boldsymbol{z})\triangleq\frac{p^\varepsilon(\boldsymbol{x}+\varepsilon\boldsymbol{z},T-t-\varepsilon)\mu_{\boldsymbol{x}+\varepsilon\boldsymbol{z}}(-\mathrm{d}\boldsymbol{z})}{p^\varepsilon(\boldsymbol{x},T-t)}.
	\end{equation}
\end{proposition}
\begin{remark}\label{rem030101}
	Here, the convention ${0}/{0}=0$ is adopted to ensure that the probability measure $\bar{\mu}^\varepsilon_{\boldsymbol{x},t}$ is well-defined. The notation $\mu_{\boldsymbol{x}}(-\mathrm{d}\boldsymbol{z})$ denotes the push-forward probability measure under the map $\boldsymbol{z}\to-\boldsymbol{z}$, i.e., $\int_D\mu_{\boldsymbol{x}}(-\mathrm{d}\boldsymbol{z})=\mu_{\boldsymbol{x}}(-D)$ for any Borel set $D\subset\mathbb{R}^d$.
\end{remark}

Let $\{\bar{\boldsymbol{x}}^\varepsilon(t):t\in[0,T]\}$ be a left-continuous stochastic step function depending on the density family $\{p^\varepsilon(\boldsymbol{x},t)\}_{t\in[0,T]}$, which is defined by 
\begin{equation}\label{eq030104}
	\left\{\begin{aligned}
		&\bar{\boldsymbol{x}}^\varepsilon(t)=\bar{\boldsymbol{x}}^\varepsilon(\Delta), &&t\in[0,\Delta],\\
		&\bar{\boldsymbol{x}}^\varepsilon(t)=\bar{\boldsymbol{x}}^\varepsilon(\Delta)+\varepsilon\sum_{m=0}^{\lceil{(t-\Delta)/\varepsilon}\rceil-1}\bar{\boldsymbol{Z}}^\varepsilon_{m}, &&t\in(\Delta,T].
	\end{aligned}\right.
\end{equation}
where $\{\Bar{\boldsymbol{Z}}^\varepsilon_{m}\}$ is a jump process whose conditional jump probability at time $m$ is non-stationary (i.e., dependent on the variable $m$), independent of $\Bar{\boldsymbol{Z}}^\varepsilon_{k}$ for $k<m$, dependent only on the current state of $\bar{\boldsymbol{x}}^\varepsilon(\Delta+m\varepsilon)$ (i.e., independent of previous states of $\bar{\boldsymbol{x}}^\varepsilon(\Delta+k\varepsilon)$ for $k<m$), and is given by $\bar{\mu}^\varepsilon_{\boldsymbol{x},\Delta+m\varepsilon}$. Specifically, for any Borel set $D\subset\mathbb{R}^d$ and $m=0,1,\cdots,\lfloor{T/\varepsilon}\rfloor-1$,
\begin{equation*}
	\begin{aligned}
			P^\varepsilon\left(\bar{\boldsymbol{Z}}^\varepsilon_{m}\in{D}\vert\bar{\boldsymbol{x}}^\varepsilon(\Delta+m\varepsilon)=\boldsymbol{x},\bar{\mathcal{F}}_{<m}\right)=\bar{\mu}^\varepsilon_{\boldsymbol{x},\Delta+m\varepsilon}(D),
		\end{aligned}
\end{equation*}
where $\bar{\mathcal{F}}_{<m}$ denotes the history of $\bar{\boldsymbol{x}}^\varepsilon(\Delta+k\varepsilon)$ before $m$. Comparing the Kolmogorov forward equation for $\bar{\boldsymbol{x}}^\varepsilon(t)$ with that obtained in Proposition \ref{theo030101}, we deduce: 
\begin{corollary}\label{theo030102}
	If $\bar{\boldsymbol{x}}^\varepsilon(0)$ has initial density $\bar{p}^\varepsilon(\boldsymbol{x},0)$, then $\bar{p}^\varepsilon(\boldsymbol{x},t)$ is the probability density of the process $\bar{\boldsymbol{x}}^\varepsilon(t)$ for every $t\in[0,T]$.
\end{corollary}
\begin{remark}\label{rem030102}
	It is well known that the time reversal of a Markov process remains a Markov process \cite{Haussmann_1986,Millet_1989,Cattiaux_2023}. Under suitable conditions, a broad class of Markov processes with jumps in a continuous time setting also retains this property under the time reversal transformation $t\to {T-t}$ \cite{Conforti_2022}. The process $\bar{\boldsymbol{x}}^\varepsilon(t)$, introduced here to characterize the time evolution of the reversed family $\{\bar{p}^\varepsilon(\boldsymbol{x},t)\}_{t\in[0,T]}$, constitutes a discrete-time analogue of the time reversal of continuous-time Markov jump processes. One can consult Reference \cite{Conforti_2022} for a similar but more rigorous proof in the continuous time context.
\end{remark}
\subsection{Conditional Probability Density of the Reversed Process}\label{sec0302}
For a fixed $\boldsymbol{x}_T$, define
\begin{equation}\label{eq030201}
	q^\varepsilon(\boldsymbol{x},t;\boldsymbol{x}_T,T)\triangleq\frac{p^\varepsilon(\boldsymbol{x},t)p^\varepsilon(\boldsymbol{x}_T,T\vert\boldsymbol{x},t)}{p^\varepsilon(\boldsymbol{x}_T,T)},\quad {t}\in[0,T],
\end{equation}
and its time reversal as
\begin{equation}\label{eq030202}
	\bar{q}^\varepsilon(\boldsymbol{x},t;\boldsymbol{x}_T,T)\triangleq{q}^\varepsilon(\boldsymbol{x},T-t;\boldsymbol{x}_T,T)=\frac{p^\varepsilon(\boldsymbol{x},T-t)p^\varepsilon(\boldsymbol{x}_T,T\vert\boldsymbol{x},T-t)}{p^\varepsilon(\boldsymbol{x}_T,T)},\quad {t}\in[0,T].
\end{equation}
Clearly, $q^\varepsilon(\boldsymbol{x},t;\boldsymbol{x}_T,T)\geq{0}$, $\bar{q}^\varepsilon(\boldsymbol{x},t;\boldsymbol{x}_T,T)\geq{0}$, and 
\begin{equation*}
	\int_{\mathbb{R}^d}q^\varepsilon(\boldsymbol{x},t;\boldsymbol{x}_T,T)\lambda(\mathrm{d}\boldsymbol{x})=1,\quad \int_{\mathbb{R}^d}\bar{q}^\varepsilon(\boldsymbol{x},t;\boldsymbol{x}_T,T)\lambda(\mathrm{d}\boldsymbol{x})=1.
\end{equation*}
Thus, both $q^\varepsilon$ and $\bar{q}^\varepsilon$ can be interpreted as probability densities of certain stochastic processes, respectively. We proceed to define these processes explicitly.

Consider the difference equations
\begin{equation*}
	\begin{aligned}
		q^\varepsilon(\boldsymbol{x},t+\varepsilon&;\boldsymbol{x}_T,T)-q^\varepsilon(\boldsymbol{x},t;\boldsymbol{x}_T,T)\\
		&=\frac{p^\varepsilon(\boldsymbol{x},t+\varepsilon)p^\varepsilon(\boldsymbol{x}_T,T\vert\boldsymbol{x},t+\varepsilon)}{p^\varepsilon(\boldsymbol{x}_T,T)}-\frac{p^\varepsilon(\boldsymbol{x},t)p^\varepsilon(\boldsymbol{x}_T,T\vert\boldsymbol{x},t)}{p^\varepsilon(\boldsymbol{x}_T,T)},
	\end{aligned}
\end{equation*}
and 
\begin{equation*}
	\begin{aligned}
		\bar{q}^\varepsilon&(\boldsymbol{x},t+\varepsilon;\boldsymbol{x}_T,T)-\bar{q}^\varepsilon(\boldsymbol{x},t;\boldsymbol{x}_T,T)\\
		&=\frac{p^\varepsilon(\boldsymbol{x},T-t-\varepsilon)p^\varepsilon(\boldsymbol{x}_T,T\vert\boldsymbol{x},T-t-\varepsilon)}{p^\varepsilon(\boldsymbol{x}_T,T)}-\frac{p^\varepsilon(\boldsymbol{x},T-t)p^\varepsilon(\boldsymbol{x}_T,T\vert\boldsymbol{x},T-t)}{p^\varepsilon(\boldsymbol{x}_T,T)}.
	\end{aligned}
\end{equation*}
Substituting equations (\ref{eq030001}), (\ref{eq030002}), (\ref{eq030005}), and (\ref{eq030006}) into these expressions and regrouping terms leads to the following proposition:
\begin{proposition}\label{theo030201}
	(a) The density $\bar{q}^\varepsilon(\boldsymbol{x},t;\boldsymbol{x}_T,T)$ satisfies
	\begin{equation}\label{eq030203}
		\left\{\begin{aligned}
			&\bar{q}^\varepsilon(\boldsymbol{x},t;\boldsymbol{x}_T,T)=\bar{q}^\varepsilon(\boldsymbol{x},\Delta;\boldsymbol{x}_T,T), && t\in[0,\Delta],\\
			&\bar{q}^\varepsilon(\boldsymbol{x},t;\boldsymbol{x}_T,T)=\bar{q}^\varepsilon(\boldsymbol{x},\Delta+(m+1)\varepsilon;\boldsymbol{x}_T,T), && t\in(\Delta+m\varepsilon,\Delta+(m+1)\varepsilon],
		\end{aligned}\right.
	\end{equation}
	and 
	\begin{equation}\label{eq030204}
		\begin{aligned}
			\bar{q}^\varepsilon(\boldsymbol{x},\Delta+(m+1)\varepsilon;\boldsymbol{x}_T,T)&-\bar{q}^\varepsilon(\boldsymbol{x},\Delta+m\varepsilon;\boldsymbol{x}_T,T)\\
			&=\mathscr{K}_{\boldsymbol{x},\Delta+m\varepsilon}^{\varepsilon,*}\bar{q}^\varepsilon(\boldsymbol{x},\Delta+m\varepsilon;\boldsymbol{x}_T,T),
		\end{aligned}
	\end{equation}
	for $m=0,1,\cdots,\lfloor{T/\varepsilon}\rfloor-1$.
	
	(b) The density $q^\varepsilon(\boldsymbol{x},t;\boldsymbol{x}_T,T)$ satisfies
	\begin{equation}\label{eq030205}
		\left\{\begin{aligned}
			&q^\varepsilon(\boldsymbol{x},t;\boldsymbol{x}_T,T)=q^\varepsilon(\boldsymbol{x},n\varepsilon;\boldsymbol{x}_T,T), &&t\in[n\varepsilon,(n+1)\varepsilon),\\
			&q^\varepsilon(\boldsymbol{x},t;\boldsymbol{x}_T,T)=q^\varepsilon(\boldsymbol{x},\lfloor{T/\varepsilon}\rfloor\varepsilon;\boldsymbol{x}_T,T), && t\in[\lfloor{T/\varepsilon}\rfloor\varepsilon,T],
		\end{aligned}\right.
	\end{equation}
	and 
	\begin{equation}\label{eq030206}
		q^\varepsilon(\boldsymbol{x},(n+1)\varepsilon;\boldsymbol{x}_T,T)-q^\varepsilon(\boldsymbol{x},n\varepsilon;\boldsymbol{x}_T,T)=\mathscr{L}_{\boldsymbol{x},n\varepsilon}^{\varepsilon,*}q^\varepsilon(\boldsymbol{x},n\varepsilon;\boldsymbol{x}_T,T),
	\end{equation}
	for $n=0,1,\cdots,\lfloor{T/\varepsilon}\rfloor-1$. Here, the operator $\mathscr{L}_{\boldsymbol{x},t}^{\varepsilon,*}$ is defined for each bounded continuous function $f(\boldsymbol{x},t)$ and each $t=n\varepsilon\in[0,T-\Delta-\varepsilon]$ by
	\begin{equation*}
		\mathscr{L}_{\boldsymbol{x},t}^{\varepsilon,*}f(\boldsymbol{x},t)\triangleq\int_{\mathbb{R}^d}\left[f(\boldsymbol{x}-\varepsilon\boldsymbol{z},t)\tilde{\mu}_{\boldsymbol{x}-\varepsilon\boldsymbol{z},t}^\varepsilon(\mathrm{d}\boldsymbol{z})-f(\boldsymbol{x},t)\tilde{\mu}_{\boldsymbol{x},t}^\varepsilon(\mathrm{d}\boldsymbol{z})\right],
	\end{equation*}
	with
	\begin{equation}\label{eq030207}
		\tilde{\mu}_{\boldsymbol{x},t}^\varepsilon(\mathrm{d}\boldsymbol{z})\triangleq\frac{p^\varepsilon(\boldsymbol{x}_T,T\vert\boldsymbol{x}+\varepsilon\boldsymbol{z},t+\varepsilon)\mu_{\boldsymbol{x}}(\mathrm{d}\boldsymbol{z})}{p^\varepsilon(\boldsymbol{x}_T,T\vert\boldsymbol{x},t)}.
	\end{equation}
\end{proposition}

Comparing equations (\ref{eq030203}), (\ref{eq030204}) with (\ref{eq030101}), (\ref{eq030102}), we obtain:
\begin{corollary}\label{theo030202}
	(a) For each $t\in[0,T]$, $\bar{q}^\varepsilon(\boldsymbol{x},t;\boldsymbol{x}_T,T)$ is the probability density of $\bar{\boldsymbol{x}}^\varepsilon(t)$ when the initial density of $\bar{\boldsymbol{x}}^\varepsilon(0)$ is $\bar{q}^\varepsilon(\boldsymbol{x},0;\boldsymbol{x}_T,T)=\delta(\boldsymbol{x}-\boldsymbol{x}_T)$, i.e., $\bar{\boldsymbol{x}}^\varepsilon(0)=\boldsymbol{x}_T$ almost surely.
	
	(b) $\bar{q}^\varepsilon(\boldsymbol{x},t;\boldsymbol{x}_T,T)$ is the conditional probability density of the reversed family $\{\bar{p}^\varepsilon(\boldsymbol{x},t)\}_{t\in[0,T]}$ in the sense that
	\begin{equation*}
		\int_{\mathbb{R}^d}{\bar{p}^\varepsilon(\boldsymbol{x}_T,0)}\bar{q}^\varepsilon(\boldsymbol{x},t;\boldsymbol{x}_T,T)\lambda(\mathrm{d}\boldsymbol{x}_T)=\bar{p}^\varepsilon(\boldsymbol{x},t),
	\end{equation*}
	i.e., the Chapman-Kolmogorov-type equation holds.
\end{corollary}

Let $\{\tilde{\boldsymbol{x}}^\varepsilon(t):t\in[0,T]\}$ be another right-continuous stochastic step function defined by
\begin{equation}\label{eq030208}
	\left\{\begin{aligned}
		&\tilde{\boldsymbol{x}}^\varepsilon(t)=\tilde{\boldsymbol{x}}^\varepsilon(0), &&t\in[0,\varepsilon),\\
		&\tilde{\boldsymbol{x}}^\varepsilon(t)=\tilde{\boldsymbol{x}}^\varepsilon(0)+\varepsilon\sum_{n=0}^{\lfloor{t/\varepsilon}\rfloor-1}\tilde{\boldsymbol{Z}}^\varepsilon_{n}, &&t\in[\varepsilon,T],
	\end{aligned}\right.
\end{equation}
	where $\{\tilde{\boldsymbol{Z}}^\varepsilon_{n}\}$ is a jump process whose conditional jump probability at time $n$ is non-stationary (i.e., dependent on the variable $n$), independent of $\tilde{\boldsymbol{Z}}^\varepsilon_{k}$ for $k<n$, dependent only on the current state of $\tilde{\boldsymbol{x}}^\varepsilon(n\varepsilon)$ (i.e., independent of
	previous states of $\tilde{\boldsymbol{x}}^\varepsilon(k\varepsilon)$ for $k<n$), and is given by $\tilde{\mu}_{\boldsymbol{x},n\varepsilon}^\varepsilon$. That is, for any Borel set $D\subset\mathbb{R}^d$ and $n=0,1,\cdots,\lfloor{T/\varepsilon}\rfloor-1$, 
	\begin{equation*}
		P^{\varepsilon}\left(\tilde{\boldsymbol{Z}}^\varepsilon_{n}\in{D}\vert\tilde{\boldsymbol{x}}^\varepsilon(n\varepsilon)=\boldsymbol{x},\tilde{\mathcal{F}}_{<n}\right)=\tilde{\mu}_{\boldsymbol{x},n\varepsilon}^\varepsilon(D),
	\end{equation*}
	where $\tilde{\mathcal{F}}_{<n}$ denotes the history of $\tilde{\boldsymbol{x}}^\varepsilon(k\varepsilon)$ before $n$. This leads to:
	\begin{corollary}\label{theo030203}
		For each $t\in[0,T]$, $q^\varepsilon(\boldsymbol{x},t;\boldsymbol{x}_T,T)$ is the probability density of $\tilde{\boldsymbol{x}}^\varepsilon(t)$ when the initial density of $\tilde{\boldsymbol{x}}^\varepsilon(0)$ is 
		\begin{equation*}
			q^\varepsilon(\boldsymbol{x},0;\boldsymbol{x}_T,T)=\frac{p^\varepsilon(\boldsymbol{x},0)p^\varepsilon(\boldsymbol{x}_T,T\vert\boldsymbol{x},0)}{p^\varepsilon(\boldsymbol{x}_T,T)}.
		\end{equation*}
	\end{corollary}
	\begin{remark}\label{rem030201}
		The process $\tilde{\boldsymbol{x}}^\varepsilon(t)$ can be described alternatively as follows. By Proposition \ref{theo030201}(a), the family $\{\bar{q}^\varepsilon(\boldsymbol{x},t;\boldsymbol{x}_T,T)\}_{t\in[0,T]}$ consists of probability densities of the process $\bar{\boldsymbol{x}}^\varepsilon(t)$. Repeating the derivation presented in Sec. \ref{sec0301}, it is easy to show that the reversed family $\{\bar{\bar{q}}^\varepsilon(\boldsymbol{x},t;\boldsymbol{x}_T,T)=q^\varepsilon(\boldsymbol{x},t;\boldsymbol{x}_T,T)\}_{t\in[0,T]}$ satisfies Eq. (\ref{eq030205}) and another Kolmogorov forward equation with conditional jump probability:
		\begin{equation*}
			\begin{aligned}
				\bar{\bar{\mu}}^\varepsilon_{\boldsymbol{x},t}(\mathrm{d}\boldsymbol{z})&\triangleq\frac{\bar{q}^\varepsilon(\boldsymbol{x}+\varepsilon\boldsymbol{z},T-t-\varepsilon;\boldsymbol{x}_T,T)\bar{\mu}^\varepsilon_{\boldsymbol{x}+\varepsilon\boldsymbol{z},T-t-\varepsilon}(-\mathrm{d}\boldsymbol{z})}{\bar{q}^\varepsilon(\boldsymbol{x},T-t;\boldsymbol{x}_T,T)}\\
				&=\frac{p^\varepsilon(\boldsymbol{x}_T,T\vert\boldsymbol{x}+\varepsilon\boldsymbol{z},t+\varepsilon)\mu_{\boldsymbol{x}}(\mathrm{d}\boldsymbol{z})}{p^\varepsilon(\boldsymbol{x}_T,T\vert\boldsymbol{x},t)},
			\end{aligned}
		\end{equation*} 
		which is exactly the conditional probability defined in Eq. (\ref{eq030207}). Hence, $\tilde{\boldsymbol{x}}^\varepsilon(t)$ defined here is a stochastic process characterizing how the time reversal of the family $\{\bar{q}^\varepsilon(\boldsymbol{x},t;\boldsymbol{x}_T,T)\}_{t\in[0,T]}$ evolves. We can therefore view $\tilde{\boldsymbol{x}}^\varepsilon(t)$ as a time reversal of the process $\bar{\boldsymbol{x}}^\varepsilon(t)$, and to some extent, as a double time reversal of the original process $\boldsymbol{x}^\varepsilon(t)$.
	\end{remark}
\subsection{Non-stationary Prehistory Probability Density}\label{sec0304}
\begin{definition}\label{def030401}
	The \textbf{non-stationary prehistory probability density (NPPD)} is defined for each $t\in[0,T]$ as
	\begin{equation}\label{eq030401}
		q^{\varepsilon,\text{NPPD}}(\boldsymbol{x},t;\boldsymbol{x}_T,T;\boldsymbol{x}_0,0)\triangleq\frac{p^\varepsilon(\boldsymbol{x},t\vert\boldsymbol{x}_0,0)p^\varepsilon(\boldsymbol{x}_T,T\vert\boldsymbol{x},t)}{p^\varepsilon(\boldsymbol{x}_T,T\vert\boldsymbol{x}_0,0)},
	\end{equation}
	with its time reversal given by
	\begin{equation}\label{eq030402}
		\begin{aligned}
			\bar{q}^{\varepsilon,\text{NPPD}}(\boldsymbol{x},t;\boldsymbol{x}_T,T;\boldsymbol{x}_0,0)&\triangleq{q}^{\varepsilon,\text{NPPD}}(\boldsymbol{x},T-t;\boldsymbol{x}_T,T;\boldsymbol{x}_0)\\
			&=\frac{p^\varepsilon(\boldsymbol{x},T-t\vert\boldsymbol{x}_0,0)p^\varepsilon(\boldsymbol{x}_T,T\vert\boldsymbol{x},T-t)}{p^\varepsilon(\boldsymbol{x}_T,T\vert\boldsymbol{x}_0,0)}.
		\end{aligned}
	\end{equation}
\end{definition}

Similarly, we obtain:
\begin{corollary}\label{theo030401}
	(a) $q^{\varepsilon,\text{NPPD}}(\boldsymbol{x},t;\boldsymbol{x}_T,T;\boldsymbol{x}_0,0)$ satisfies Eq. (\ref{eq030205}) and 
	\begin{equation}\label{eq030403}
		\begin{aligned}
			q^{\varepsilon,\text{NPPD}}(\boldsymbol{x},(n+1)\varepsilon;\boldsymbol{x}_T,T;\boldsymbol{x}_0,0)&-q^{\varepsilon,\text{NPPD}}(\boldsymbol{x},n\varepsilon;\boldsymbol{x}_T,T;\boldsymbol{x}_0,0)\\
			&=\mathscr{L}_{\boldsymbol{x},n\varepsilon}^{\varepsilon,*}q^{\varepsilon,\text{NPPD}}(\boldsymbol{x},n\varepsilon;\boldsymbol{x}_T,T;\boldsymbol{x}_0,0).
		\end{aligned}
	\end{equation}
	Therefore, for each $t\in[0,T]$, $q^{\varepsilon,\text{NPPD}}(\boldsymbol{x},t;\boldsymbol{x}_T,T;\boldsymbol{x}_0,0)$ is the probability density of $\tilde{\boldsymbol{x}}^\varepsilon(t)$ when the initial density of $\tilde{\boldsymbol{x}}^\varepsilon(0)$ is $q^{\varepsilon,\text{NPPD}}(\boldsymbol{x},0;\boldsymbol{x}_T,T;\boldsymbol{x}_0,0)=\delta(\boldsymbol{x}-\boldsymbol{x}_0)$, i.e., $\tilde{\boldsymbol{x}}^\varepsilon(0)=\boldsymbol{x}_0$ almost surely.
	
	(b) $\bar{q}^{\varepsilon,\text{NPPD}}(\boldsymbol{x},t;\boldsymbol{x}_T,T;\boldsymbol{x}_0,0)$ is the conditional probability density of the reversed family $\{\bar{p}^\varepsilon(\boldsymbol{x},t)=p^\varepsilon(\boldsymbol{x},T-t\vert\boldsymbol{x}_0,0)\}_{t\in[0,T]}$. It satisfies Eq. (\ref{eq030203}) and 
	\begin{equation}\label{eq030404}
		\begin{aligned}
			\bar{q}^{\varepsilon,\text{NPPD}}(\boldsymbol{x},\Delta+(m+1)\varepsilon;\boldsymbol{x}_T,T;\boldsymbol{x}_0,0)&-\bar{q}^{\varepsilon,\text{NPPD}}(\boldsymbol{x},\Delta+m\varepsilon;\boldsymbol{x}_T,T;\boldsymbol{x}_0,0)\\
			=\mathscr{N}^{\varepsilon,*}_{\boldsymbol{x},\Delta+m\varepsilon}&\bar{q}^{\varepsilon,\text{NPPD}}(\boldsymbol{x},\Delta+m\varepsilon;\boldsymbol{x}_T,T;\boldsymbol{x}_0,0),
		\end{aligned}
	\end{equation}
	where
	\begin{equation*}
		\mathscr{N}^{\varepsilon,*}_{\boldsymbol{x},t}f(\boldsymbol{x},t)\triangleq\int_{\mathbb{R}^d}\left[f(\boldsymbol{x}-\varepsilon\boldsymbol{z},t)\bar{\mu}^{\varepsilon,\text{NPPD}}_{\boldsymbol{x}-\varepsilon\boldsymbol{z},t}(\mathrm{d}\boldsymbol{z})-f(\boldsymbol{x},t)\bar{\mu}^{\varepsilon,\text{NPPD}}_{\boldsymbol{x},t}(\mathrm{d}\boldsymbol{z})\right],
	\end{equation*}
	with
	\begin{equation}\label{eq030405}
		\bar{\mu}^{\varepsilon,\text{NPPD}}_{\boldsymbol{x},t}(\mathrm{d}\boldsymbol{z})\triangleq\frac{p^\varepsilon(\boldsymbol{x}+\varepsilon\boldsymbol{z},T-t-\varepsilon\vert\boldsymbol{x}_0,0)\mu_{\boldsymbol{x}+\varepsilon\boldsymbol{z}}(-\mathrm{d}\boldsymbol{z})}{p^\varepsilon(\boldsymbol{x},T-t\vert\boldsymbol{x}_0,0)},
	\end{equation}
	for $t=\Delta+m\varepsilon\in[0,T-\varepsilon]$ and $m=0,1,\cdots,\lfloor{T/\varepsilon}\rfloor-1$. Now define a new left-continuous stochastic step function $\{\bar{\boldsymbol{x}}^{\varepsilon,\text{NPPD}}(t):t\in[0,T]\}$ as
	\begin{equation}\label{eq030406}
		\left\{\begin{aligned}
			&\bar{\boldsymbol{x}}^{\varepsilon,\text{NPPD}}(t)=\bar{\boldsymbol{x}}^{\varepsilon,\text{NPPD}}(\Delta), &&t\in[0,\Delta],\\
			&\bar{\boldsymbol{x}}^{\varepsilon,\text{NPPD}}(t)=\bar{\boldsymbol{x}}^{\varepsilon,\text{NPPD}}(\Delta)+\varepsilon\sum_{m=0}^{\lceil{(t-\Delta)/\varepsilon}\rceil-1}\bar{\boldsymbol{Z}}^{\varepsilon,\text{NPPD}}_{m}, &&t\in(\Delta,T],
		\end{aligned}\right.
	\end{equation}
	where $\{\bar{\boldsymbol{Z}}^{\varepsilon,\text{NPPD}}_{m}\}$ is a jump process whose conditional jump probability at time $m$ is non-stationary (i.e., dependent on the variable $m$), independent of $\bar{\boldsymbol{Z}}^{\varepsilon,\text{NPPD}}_{k}$ for $k<m$, dependent only on the current state of $\bar{\boldsymbol{x}}^{\varepsilon,\text{NPPD}}(\Delta+m\varepsilon)$ (i.e., independent of previous states of $\bar{\boldsymbol{x}}^{\varepsilon,\text{NPPD}}(\Delta+k\varepsilon)$ for $k<m$), and is given by $\bar{\mu}^{\varepsilon,\text{NPPD}}_{\boldsymbol{x},\Delta+m\varepsilon}$. That is, for any Borel set $D\subset\mathbb{R}^d$ and $m=0,1,\cdots,\lfloor{T/\varepsilon}\rfloor-1$, 
	\begin{equation*}
		P^{\varepsilon}\left(\bar{\boldsymbol{Z}}^{\varepsilon,\text{NPPD}}_{m}\in{D}\vert\bar{\boldsymbol{x}}^{\varepsilon,\text{NPPD}}(\Delta+m\varepsilon)=\boldsymbol{x},\bar{\mathcal{F}}^{\text{NPPD}}_{<m}\right)=\bar{\mu}^{\varepsilon,\text{NPPD}}_{\boldsymbol{x},\Delta+m\varepsilon}(D),
	\end{equation*}
	where $\bar{\mathcal{F}}^{\text{NPPD}}_{<m}$ denotes the history of $\bar{\boldsymbol{x}}^{\varepsilon,\text{NPPD}}(\Delta+k\varepsilon)$ before $m$. Then, the function $\bar{q}^{\varepsilon,\text{NPPD}}(\boldsymbol{x},t;\boldsymbol{x}_T,T;\boldsymbol{x}_0,0)$ is the probability density of $\bar{\boldsymbol{x}}^{\varepsilon,\text{NPPD}}(t)$ when $\bar{\boldsymbol{x}}^{\varepsilon,\text{NPPD}}(0)=\boldsymbol{x}_T$ almost surely.
\end{corollary}
\begin{remark}\label{rem030401}
	In fact, the density $q^{\varepsilon,\text{NPPD}}(\boldsymbol{x},t;\boldsymbol{x}_T,T;\boldsymbol{x}_0,0)$ also serves as the conditional probability density of the family $\{q^\varepsilon(\boldsymbol{x},t;\boldsymbol{x}_T,T)\}_{t\in[0,T]}$, since it satisfies the Chapman-Kolmogorov-type equation:
	\begin{equation*}
		\int_{\mathbb{R}^d}{q^\varepsilon(\boldsymbol{x}_0,0;\boldsymbol{x}_T,T)q^{\varepsilon,\text{NPPD}}(\boldsymbol{x},t;\boldsymbol{x}_T,T;\boldsymbol{x}_0,0)}\lambda(\mathrm{d}\boldsymbol{x}_0)=q^\varepsilon(\boldsymbol{x},t;\boldsymbol{x}_T,T).
	\end{equation*}
\end{remark}

\section{Prehistorical Description of Optimal Fluctuations on Finite Time Intervals}\label{sec04} 
 For given $\boldsymbol{x}_0,\,\boldsymbol{x}_T,\,T$, assume that the conditions in Theorem \ref{theo020101}, Proposition \ref{theo020202}(e), Lemmas \ref{theo020203} and \ref{theo020204} hold. Under these conditions, the NOP that connects $\boldsymbol{x}_0$ with $\boldsymbol{x}_T$ in the time span $T$ is unique. Define $\bar{\boldsymbol{x}}^{0}(t)\triangleq\boldsymbol{\phi}_{\text{NOP}}(T-t;\boldsymbol{x}_T,T;\boldsymbol{x}_0)$ for $t\in[0,T]$. It follows from Eqs. (\ref{eq020206}) and (\ref{eq020208}) that for any $T^*\in(0,T)$, $\bar{\boldsymbol{x}}^{0}(t)$ satisfies
\begin{equation}\label{eq040101}
	\dot{\bar{\boldsymbol{x}}}^{0}(t)=\bar{\boldsymbol{b}}(\bar{\boldsymbol{x}}^{0}(t),t),\;\;t\in[0,T^*],
\end{equation}
with the initial condition $\bar{\boldsymbol{x}}^{0}(0)=\boldsymbol{x}_T$, where
\begin{equation*}
	\begin{aligned}
		\bar{\boldsymbol{b}}(\boldsymbol{x},t)&\triangleq-\nabla_{\boldsymbol{\alpha}}H(\boldsymbol{x},\nabla_{\boldsymbol{x}}S(\boldsymbol{x},T-t\vert\boldsymbol{x}_0))\\
		&=-\frac{\int_{\mathbb{R}^d}{\boldsymbol{z}e^{\boldsymbol{z}\cdot\nabla_{\boldsymbol{x}}S(\boldsymbol{x},T-t\vert\boldsymbol{x}_0)}\mu_{\boldsymbol{x}}(\mathrm{d}\boldsymbol{z})}}{e^{H(\boldsymbol{x},\nabla_{\boldsymbol{x}}S(\boldsymbol{x},T-t\vert\boldsymbol{x}_0))}}\\
		&=-\frac{\int_{\mathbb{R}^d}{\boldsymbol{z}e^{\boldsymbol{z}\cdot\nabla_{\boldsymbol{x}}S(\boldsymbol{x},T-t\vert\boldsymbol{x}_0)}\mu_{\boldsymbol{x}}(\mathrm{d}\boldsymbol{z})}}{e^{\frac{\partial}{\partial{t}}S(\boldsymbol{x},T-t\vert\boldsymbol{x}_0)}}.
	\end{aligned}
\end{equation*}

Now consider the left-continuous stochastic step function $\bar{\boldsymbol{x}}^{\varepsilon,\text{NPPD}}(t)$ defined in Eq. (\ref{eq030406}), with the initial condition $\bar{\boldsymbol{x}}^{\varepsilon,\text{NPPD}}(0)=\boldsymbol{x}_T$. Denote
\begin{equation*}
	\bar{\boldsymbol{b}}^{\varepsilon}(\boldsymbol{x},t)\triangleq\int_{\mathbb{R}^d}{\boldsymbol{z}\bar{\mu}^{\varepsilon,\text{NPPD}}_{\boldsymbol{x},t}(\mathrm{d}\boldsymbol{z})}.
\end{equation*}
For any $T^{*}\in(0,T)$ and a sufficiently small $\delta_{0}$, let $\bar{\Sigma}_{\delta_{0},[0,T^{*}]}\triangleq\bigcup_{t\in[0,T^{*}]}B_{\delta_{0}}(\bar{\boldsymbol{x}}^{0}(t))$. Obviously, $\bar{\Sigma}_{\delta_{0},[0,T^{*}]}=\Sigma_{\delta_{0},[T-T^{*},T]}$ (as defined in Lemma \ref{theo020203}). The following proposition of the type of the law of large numbers can be established. Its proof is given in \ref{secA2}.

\begin{proposition}\label{theo040101}
	Suppose that the conditions in Theorem \ref{theo020101}, Proposition \ref{theo020202}(e) and Lemma \ref{theo020203} hold. We further assume that for any $T^*\in(0,T)$, there exists a constant $\delta_{0}>0$ such that:
	\begin{itemize}
		\item[(a)] for sufficiently small $\varepsilon$, $\sup_{(\boldsymbol{x},t)\in\bar{\Sigma}_{\delta_{0},[0,T^{*}]}\times[0,T^{*}]}\int_{\mathbb{R}^d}{\vert\boldsymbol{z}\vert\bar{\mu}^{\varepsilon,\text{NPPD}}_{\boldsymbol{x},t}(\mathrm{d}\boldsymbol{z})}<G_1<\infty$;
		\item[(b)] $\bar{\boldsymbol{b}}^{\varepsilon}(\boldsymbol{x},t)\to\bar{\boldsymbol{b}}(\boldsymbol{x},t)$ as $\varepsilon\to{0}$, uniformly for $(\boldsymbol{x},t)\in\bar{\Sigma}_{\delta_{0},[0,T^{*}]}\times[0,T^{*}]$;
		\item[(c)] there exists a constant $\varkappa_1\in(0,1]$ such that for sufficiently small $\varepsilon$,
		\begin{equation*}
			\sup_{(\boldsymbol{x},t)\in\bar{\Sigma}_{\delta_{0},[0,T^{*}]}\times[0,T^{*}]}\int_{\mathbb{R}^d}\left\vert\boldsymbol{z}-\bar{\boldsymbol{b}}^{\varepsilon}(\boldsymbol{x},t)\right\vert^{1+\varkappa_1}\bar{\mu}^{\varepsilon,\text{NPPD}}_{\boldsymbol{x},t}(\mathrm{d}\boldsymbol{z})<G_2<\infty.
		\end{equation*}
	\end{itemize}
	Then, for any $T^*\in(0,T)$, there exists a constant $\delta_0$ so that for every $\delta<\delta_{0}$,
	\begin{equation}\label{eq040102}
		\lim_{\varepsilon\to{0}}P^\varepsilon\left(\sup_{t\in[0,T^{*}]}\left\vert\bar{\boldsymbol{x}}^{\varepsilon,\text{NPPD}}(t)-\bar{\boldsymbol{x}}^{0}(t)\right\vert>\delta\right)=0.
	\end{equation} 
\end{proposition}

Define $\tilde{\boldsymbol{x}}^{0}(t)\triangleq\boldsymbol{\phi}_{\text{NOP}}(t;\boldsymbol{x}_T,T;\boldsymbol{x}_0)$ for $t\in[0,T]$. By Eqs. (\ref{eq020207}) and (\ref{eq020209}), for any $T^*\in(0,T)$, $\tilde{\boldsymbol{x}}^{0}(t)$ satisfies
\begin{equation}\label{eq040103}
	\dot{\tilde{\boldsymbol{x}}}^{0}(t)=\tilde{\boldsymbol{b}}(\tilde{\boldsymbol{x}}^{0}(t),t),
\end{equation}
with the initial condition $\tilde{\boldsymbol{x}}^{0}(0)=\boldsymbol{x}_0$, where
\begin{equation*}
	\begin{aligned}
		\tilde{\boldsymbol{b}}(\boldsymbol{x},t)&\triangleq\nabla_{\boldsymbol{\alpha}}H(\boldsymbol{x},-\nabla_{\boldsymbol{x}}S(\boldsymbol{x}_T,T-t\vert\boldsymbol{x}))\\
		&=\frac{\int_{\mathbb{R}^d}{\boldsymbol{z}e^{-\boldsymbol{z}\cdot\nabla_{\boldsymbol{x}}S(\boldsymbol{x}_T,T-t\vert\boldsymbol{x})}\mu_{\boldsymbol{x}}(\mathrm{d}\boldsymbol{z})}}{e^{H(\boldsymbol{x},-\nabla_{\boldsymbol{x}}S(\boldsymbol{x}_T,T-t\vert\boldsymbol{x}))}}\\
		&=\frac{\int_{\mathbb{R}^d}{\boldsymbol{z}e^{-\boldsymbol{z}\cdot\nabla_{\boldsymbol{x}}S(\boldsymbol{x}_T,T-t\vert\boldsymbol{x})}\mu_{\boldsymbol{x}}(\mathrm{d}\boldsymbol{z})}}{e^{\frac{\partial}{\partial{t}}S(\boldsymbol{x}_T,T-t\vert\boldsymbol{x})}}.
	\end{aligned}
\end{equation*}

Similarly, let $\tilde{\boldsymbol{x}}^{\varepsilon,\text{NPPD}}(t)\triangleq\tilde{\boldsymbol{x}}^{\varepsilon}(t)$ (i.e., the right-continuous stochastic step function defined in Eq. (\ref{eq030208})) for $t\in[0,T]$, with the initial condition $\tilde{\boldsymbol{x}}^{\varepsilon,\text{NPPD}}(t)=\boldsymbol{x}_0$. Define
\begin{equation*}
	\tilde{\boldsymbol{b}}^{\varepsilon}(\boldsymbol{x},t)\triangleq\int_{\mathbb{R}^d}{\boldsymbol{z}\tilde{\mu}^{\varepsilon}_{\boldsymbol{x},t}(\mathrm{d}\boldsymbol{z})}
\end{equation*}
(For simplicity, we here use the notation $\Sigma_{\delta_{0},[0,T^{*}]}$ (as defined in Lemma \ref{theo020204}), since $\tilde{\Sigma}_{\delta_{0},[0,T^{*}]}\triangleq\bigcup_{t\in[0,T^{*}]}B_{\delta_{0}}(\tilde{\boldsymbol{x}}^{0}(t))\equiv\Sigma_{\delta_{0},[0,T^{*}]}$.) The following result is indicative of  the law of large numbers. Its proof is largely analogous to that of Proposition \ref{theo040101} and is therefore omitted.
\begin{proposition}\label{theo040102}
	In addition to the conditions in Theorem \ref{theo020101}, Proposition \ref{theo020202}(e) and Lemma \ref{theo020204}, we assume that for any $T^*\in(0,T)$, there exists a constant $\delta_{0}>0$ such that:
	\begin{itemize}
		\item[(a)] for sufficiently small $\varepsilon$, $\sup_{(\boldsymbol{x},t)\in{\Sigma}_{\delta_{0},[0,T^{*}]}\times[0,T^{*}]}\int_{\mathbb{R}^d}{\vert\boldsymbol{z}\vert\tilde{\mu}^{\varepsilon}_{\boldsymbol{x},t}(\mathrm{d}\boldsymbol{z})}<G_3<\infty$;
		
		\item[(b)] $\tilde{\boldsymbol{b}}^{\varepsilon}(\boldsymbol{x},t)\to\tilde{\boldsymbol{b}}(\boldsymbol{x},t)$ as $\varepsilon\to{0}$, uniformly for $(\boldsymbol{x},t)\in{\Sigma}_{\delta_{0},[0,T^{*}]}\times[0,T^{*}]$;
		
		\item[(c)] there exists a constant $\varkappa_2\in(0,1]$ such that for sufficiently small $\varepsilon$,
		\begin{equation*}
			\sup_{(\boldsymbol{x},t)\in{\Sigma}_{\delta_{0},[0,T^{*}]}\times[0,T^{*}]}\int_{\mathbb{R}^d}\left\vert\boldsymbol{z}-\tilde{\boldsymbol{b}}^{\varepsilon}(\boldsymbol{x},t)\right\vert^{1+\varkappa_2}\tilde{\mu}^{\varepsilon}_{\boldsymbol{x},t}(\mathrm{d}\boldsymbol{z})<G_4<\infty.
		\end{equation*}
	\end{itemize}
	Then, for any $T^*\in(0,T)$, there exists a constant $\delta_0$ so that for every $\delta<\delta_{0}$,
	\begin{equation}\label{eq040104}
		\lim_{\varepsilon\to{0}}P\left(\sup_{t\in[0,T^{*}]}\left\vert\tilde{\boldsymbol{x}}^{\varepsilon,\text{NPPD}}(t)-\tilde{\boldsymbol{x}}^{0}(t)\right\vert>\delta\right)=0.
	\end{equation}
\end{proposition}
Combining Corollary \ref{theo030401} with Propositions \ref{theo040101} and \ref{theo040102}, we obtain:
\begin{corollary}\label{theo040103}
	If the assumptions in Theorem \ref{theo020101}, Proposition \ref{theo020202}(e), Lemmas \ref{theo020203}, \ref{theo020204}, and Propositions \ref{theo040101}, \ref{theo040102} are satisfied, then the non-stationary prehistory probability density $q^{\varepsilon,\text{NPPD}}(\boldsymbol{x},t;\boldsymbol{x}_T,T;\boldsymbol{x}_0,0)$ concentrates on the non-stationary optimal path $\boldsymbol{\phi}_{\text{NOP}}(t;\boldsymbol{x}_T,T;\boldsymbol{x}_0)$ as $\varepsilon\to{0}$, uniformly for $t\in[0,T]$. This provides the complete prehistorical description of optimal fluctuations in the non-stationary setting.
\end{corollary}

\section{Examples}\label{sec06}
In this section, the class of stochastic processes delineated in Remark \ref{rem020102} is further considered, with multiple illustrations provided of instances where Propositions \ref{theo040101} and \ref{theo040102} are applicable. For simplicity, we assume that $d=1$ and $\sigma=1$. The present examination is initiated by considering the case where $\kappa=2$ and $b(x)$ is an affine function.
\begin{example}\label{exam050101}
	{$\kappa=2$, $b(x)=a_0$ or $b(x)=a_0+a_1x$.}
	
	Define
	\begin{equation*}
		M(t\vert x_s,s)\triangleq E^{\varepsilon}[x^{\varepsilon}(t)\vert x^{\varepsilon}(s)=x_s],
	\end{equation*}
	\begin{equation*}
		V(t\vert x_s,s)\triangleq E^{\varepsilon}[(x^{\varepsilon}(t)-E^{\varepsilon}x^{\varepsilon}(t))^2\vert x^{\varepsilon}(s)=x_s],
	\end{equation*}
	\begin{equation*}
		\bar{M}(t\vert x_T,0)\triangleq E^{\varepsilon}[\bar{x}^{\varepsilon,\text{NPPD}}(t)\vert \bar{x}^{\varepsilon,\text{NPPD}}(0)=x_T],
	\end{equation*}
	\begin{equation*}
		\bar{V}(t\vert x_T,0)\triangleq E^{\varepsilon}[(\bar{x}^{\varepsilon,\text{NPPD}}(t)-E^{\varepsilon}\bar{x}^{\varepsilon,\text{NPPD}}(t))^2\vert \bar{x}^{\varepsilon,\text{NPPD}}(0)=x_T],
	\end{equation*}
	\begin{equation*}
		\tilde{M}(t\vert x_0,0)\triangleq E^{\varepsilon}[\tilde{x}^{\varepsilon,\text{NPPD}}(t)\vert \tilde{x}^{\varepsilon,\text{NPPD}}(0)=x_0],
	\end{equation*}
	\begin{equation*}
		\tilde{V}(t\vert x_0,0)\triangleq E^{\varepsilon}[(\tilde{x}^{\varepsilon,\text{NPPD}}(t)-E^{\varepsilon}\tilde{x}^{\varepsilon,\text{NPPD}}(t))^2\vert \tilde{x}^{\varepsilon,\text{NPPD}}(0)=x_0].
	\end{equation*}
	As shown in Tables \ref{tab01} and \ref{tab02}, explicit expressions can be derived for all these quantities in both cases. The processes $x^{\varepsilon}(t)$, $\bar{x}^{\varepsilon,\text{NPPD}}(t)$ and $\tilde{x}^{\varepsilon,\text{NPPD}}(t)$ are all Gaussian. Furthermore, it can be directly verified that all conditions in Propositions \ref{theo040101} and \ref{theo040102} are satisfied. Note that $q^\varepsilon(x,t\vert x_T,T;x_0,0)$ is also Gaussian. As $\varepsilon\to{0}$, the NPPD concentrates around its mean $\tilde{M}(t\vert x_0,0)$ because the variance $\tilde{V}(t\vert x_0,0)\to {0}$ uniformly for $t\in[0,T]$. This further implies that the conclusion of Corollary \ref{theo040103} holds, since $\tilde{M}(t\vert x_0,0)\to\phi_{\text{NOP}}(t;x_T,T;x_0)$ uniformly for $t\in[0,T]$. 
	
	\begin{table}[ht]
		\centering
		\caption{A comprehensive exposition of $x^\varepsilon$, $\bar{x}^{\varepsilon,\text{NPPD}}$ and $\tilde{x}^{\varepsilon,\text{NPPD}}$ in the case that $b(x)=a_0$, $\kappa=2$.}
		\label{tab01}
		\resizebox{\linewidth}{!}{
			\begin{tabular}{m{16cm}}
				\toprule[1pt]
				\midrule
				1.\;\emph{Quantities related to} $x^{\varepsilon}$\\
				\hspace{0.5cm}(a)\;$\mu_{x}$ is a Gaussian measure with variance $\frac{\sigma^2}{2}$ and mean $b(x)=a_0$.\\
				\hspace{0.5cm}(b)\;$M(t\vert x_s,s)=x_s+\varepsilon{a_0}(\lfloor{t/\varepsilon}\rfloor-\lfloor{s/\varepsilon}\rfloor)$ for $0\leq{s}\leq{t}\leq{T}$.\\
				\hspace{0.5cm}(c)\;$V(t\vert x_s,s)=\frac{\varepsilon^2\sigma^2(\lfloor{t/\varepsilon}\rfloor-\lfloor{s/\varepsilon}\rfloor)}{2}$ for $0\leq{s}\leq{t}\leq{T}$. \\
				\hspace{0.5cm}(d)\;$p^\varepsilon(x,t\vert x_s,s)=\begin{cases}
					\delta(x-x_s) &\text{for}\; t\in[s,(\lfloor{s/\varepsilon}\rfloor+1)\varepsilon)\; \text{or}\; s\in [\lfloor{t/\varepsilon}\rfloor\varepsilon,t], \\
					\frac{1}{\sqrt{2\pi V(t\vert x_s,s)}}\exp\left(-\frac{(x-M(t\vert x_s,s))^2}{2V(t\vert x_s,s)}\right) &\text{for}\; t\in[(\lfloor{s/\varepsilon}\rfloor+1)\varepsilon,T]\; \text{or}\; s\in[0,\lfloor{t/\varepsilon}\rfloor\varepsilon).  
				\end{cases}$ \\ 
				\hspace{0.5cm}(e)\;$\phi_{\text{NOP}}(t;x_T,T;x_0)=x_0\frac{T-t}{T}+x_T\frac{t}{T}$ for $t\in[0,T]$. \\
				2.\;\emph{Quantities related to} $\bar{x}^{\varepsilon}$\\
				\hspace{0.5cm}(a)\;$\bar{\mu}^{\varepsilon,\text{NPPD}}_{x,\Delta+m\varepsilon}$ is a Gaussian measure with variance $\frac{\sigma^2}{2}\left(1-\frac{\varepsilon}{T-(\Delta+m\varepsilon)}\right)$ and mean $\bar{b}^\varepsilon(x,\Delta+m\varepsilon)\triangleq\frac{x_0-x}{T-(\Delta+m\varepsilon)}$ for $m=0,1,\cdots,\lfloor{T/\varepsilon}\rfloor-1$.\\
				\hspace{0.5cm}(b)\;$\bar{M}(t\vert x_T,0)=x_T\frac{T-\left(\Delta+m\varepsilon\right)}{T-\Delta}+x_0\frac{m\varepsilon}{T-\Delta}$ for $m=\lceil{(t-\Delta)/\varepsilon}\rceil$ and $t\in [0,T]$. \\
				\hspace{0.5cm}(c)\;$\bar{V}(t\vert x_T,0)=\frac{\varepsilon\sigma^2\left(T-\left(\Delta+m\varepsilon\right)\right)\left(m\varepsilon\right)}{2(T-\Delta)}$ for $m=\lceil{(t-\Delta)/\varepsilon}\rceil$ and $t\in [0,T]$. \\
				\hspace{0.5cm}(d)\;$\bar{q}^\varepsilon(x,t\vert x_T,T;x_0,0)=\begin{cases}
					\delta(x-x_T) &\text{for}\; t\in[0,\Delta],\\
					\frac{1}{\sqrt{2\pi\bar{V}(t\vert x_T,0)}}\exp\left(-\frac{(x-\bar{M}(t\vert x_T,0))^2}{2\bar{V}(t\vert x_T,0)}\right) &\text{for}\; t\in(\Delta,\Delta+(\lfloor{T/\varepsilon}\rfloor-1)\varepsilon],\\
					\delta(x-x_0) &\text{for}\; t\in(\Delta+(\lfloor{T/\varepsilon}\rfloor-1)\varepsilon,T].
				\end{cases}$ \\
				3.\;\emph{Quantities related to} $\tilde{x}^{\varepsilon}$\\
				\hspace{0.5cm}(a)\;$\tilde{\mu}^{\varepsilon,\text{NPPD}}_{x,n\varepsilon}$ is a Gaussian measure with variance $\frac{\sigma^2}{2}\left(1-\frac{\varepsilon}{(T-\Delta)-n\varepsilon}\right)$ and mean $\tilde{b}^\varepsilon(x,n\varepsilon)\triangleq\frac{x_T-x}{(T-\Delta)-n\varepsilon}$ for $n=0,1,\cdots,\lfloor{T/\varepsilon}\rfloor-1$.\\
				\hspace{0.5cm}(b)\;$\tilde{M}(t\vert x_0,0)=x_0\frac{(T-\Delta)-n\varepsilon}{T-\Delta}+x_T\frac{n\varepsilon}{T-\Delta}$ for $n=\lfloor{t/\varepsilon}\rfloor$ and $t\in [0,T]$.\\
				\hspace{0.5cm}(c)\;$\tilde{V}(t\vert x_0,0)=\frac{\varepsilon\sigma^2\left((T-\Delta)-n\varepsilon\right)\left(n\varepsilon\right)}{2(T-\Delta)}$ for $n=\lfloor{t/\varepsilon}\rfloor$ and $t\in [0,T]$.\\
				\hspace{0.5cm}(d)\;$q^\varepsilon(x,t\vert x_T,T;x_0,0)=\begin{cases}
					\delta(x-x_0) &\text{for}\; t\in[0,\varepsilon),\\
					\frac{1}{\sqrt{2\pi\tilde{V}(t\vert x_0,0)}}\exp\left(-\frac{(x-\tilde{M}(t\vert x_0,0))^2}{2\tilde{V}(t\vert x_0,0)}\right) &\text{for}\; t\in[\varepsilon,\lfloor{T/\varepsilon}\rfloor\varepsilon),\\
					\delta(x-x_T) &\text{for}\; t\in[\lfloor{T/\varepsilon}\rfloor\varepsilon,T].
				\end{cases}$\\
				\midrule
				\bottomrule[1pt]
		\end{tabular}}
	\end{table}
	\begin{table}[ht]
		\centering
		\caption{A comprehensive exposition of $x^\varepsilon$, $\bar{x}^{\varepsilon,\text{NPPD}}$ and $\tilde{x}^{\varepsilon,\text{NPPD}}$ in the case that $b(x)=a_0+a_1 x$, $\kappa=2$.}
		\label{tab02}
		\resizebox{\linewidth}{!}{
			\begin{tabular}{m{16cm}}
				\toprule[1pt]
				\midrule
				1.\;\emph{Quantities related to} $x^{\varepsilon}$\\
				\hspace{0.5cm}(a)\;$\mu_{x}$ is a Gaussian measure with variance $\frac{\sigma^2}{2}$ and mean $b(x)=a_0+a_1x$.\\
				\hspace{0.5cm}(b)\;$M(t\vert x_s,s)=x_s(1+a_1\varepsilon)^{\lfloor{t/\varepsilon}\rfloor-\lfloor{s/\varepsilon}\rfloor}+\frac{a_0}{a_1}\left((1+a_1\varepsilon)^{\lfloor{t/\varepsilon}\rfloor-\lfloor{s/\varepsilon}\rfloor}-1\right)$ for $0\leq{s}\leq{t}\leq{T}$. \\
				\hspace{0.5cm}(c)\;$V(t\vert x_s,s)=\frac{\varepsilon\sigma^2}{2a_1(2+a_1\varepsilon)}\left((1+a_1\varepsilon)^{2(\lfloor{t/\varepsilon}\rfloor-\lfloor{s/\varepsilon}\rfloor)}-1\right)$ for $0\leq{s}\leq{t}\leq{T}$. \\
				\hspace{0.5cm}(d)\;$p^\varepsilon(x,t\vert x_s,s)=\begin{cases}
					\delta(x-x_s) &\text{for}\; t\in[s,(\lfloor{s/\varepsilon}\rfloor+1)\varepsilon)\; \text{or}\; s\in [\lfloor{t/\varepsilon}\rfloor\varepsilon,t], \\
					\frac{1}{\sqrt{2\pi V(t\vert x_s,s)}}\exp\left(-\frac{(x-M(t\vert x_s,s))^2}{2V(t\vert x_s,s)}\right) &\text{for}\; t\in[(\lfloor{s/\varepsilon}\rfloor+1)\varepsilon,T]\; \text{or}\; s\in[0,\lfloor{t/\varepsilon}\rfloor\varepsilon).  
				\end{cases}$ \\ 
				\hspace{0.5cm}(e)\;$\phi_{\text{NOP}}(t;x_T,T;x_0)=x_0\frac{\sinh(a_1(T-t))}{\sinh(a_1T)}+x_T\frac{\sinh(a_1t)}{\sinh(a_1T)}-\frac{a_0}{a_1}\left(1-\frac{\sinh(a_1(T-t))}{\sinh(a_1T)}-\frac{\sinh(a_1t)}{\sinh(a_1T)}\right)$ for $t\in[0,T]$. \\
				2.\;\emph{Quantities related to} $\bar{x}^{\varepsilon}$\\
				\hspace{0.5cm}(a)\;$\bar{\mu}^{\varepsilon,\text{NPPD}}_{x,\Delta+m\varepsilon}$ is a Gaussian measure with variance $\frac{\sigma^2}{2}\frac{(1+a_1\varepsilon)^{2(\lfloor{T/\varepsilon}\rfloor-(m+1))}-1}{(1+a_1\varepsilon)^{2(\lfloor{T/\varepsilon}\rfloor-m)}-1}$ and mean $\bar{b}^\varepsilon(x,\Delta+m\varepsilon)\triangleq
				\frac{(2+a_1\varepsilon)(1+a_1\varepsilon)^{\lfloor{T/\varepsilon}\rfloor-m-1}}{(1+a_1\varepsilon)^{2(\lfloor{T/\varepsilon}\rfloor-m)}-1}a_1x_0
				-\frac{(1+a_1\varepsilon)^{2(\lfloor{T/\varepsilon}\rfloor-m)-1}+1}{(1+a_1\varepsilon)^{2(\lfloor{T/\varepsilon}\rfloor-m)}-1}a_1x-\frac{(1+a_1\varepsilon)^{\lfloor{T/\varepsilon}\rfloor-m-1}-1}{(1+a_1\varepsilon)^{\lfloor{T/\varepsilon}\rfloor-m}+1}a_0$ for $m=0,1,\cdots,\lfloor{T/\varepsilon}\rfloor-1$.\\
				\hspace{0.5cm}(b)\;$\bar{M}(t\vert x_T,0)=x_T\frac{(1+a_1\varepsilon)^{m}\left((1+a_1\varepsilon)^{2(\lfloor{T/\varepsilon}\rfloor-m)}-1\right)}{(1+a_1\varepsilon)^{2\lfloor{T/\varepsilon}\rfloor}-1}+x_0\frac{(1+a_1\varepsilon)^{\lfloor{T/\varepsilon}\rfloor-m}\left((1+a_1\varepsilon)^{2m}-1\right)}{(1+a_1\varepsilon)^{2\lfloor{T/\varepsilon}\rfloor}-1}-\frac{a_0}{a_1}\bigg(1-\frac{(1+a_1\varepsilon)^{m}\left((1+a_1\varepsilon)^{2(\lfloor{T/\varepsilon}\rfloor-m)}-1\right)}{(1+a_1\varepsilon)^{2\lfloor{T/\varepsilon}\rfloor}-1}-\frac{(1+a_1\varepsilon)^{\lfloor{T/\varepsilon}\rfloor-m}\left((1+a_1\varepsilon)^{2m}-1\right)}{(1+a_1\varepsilon)^{2\lfloor{T/\varepsilon}\rfloor}-1}\bigg)$ for $m=\lceil{(t-\Delta)/\varepsilon}\rceil$ and $t\in [0,T]$.\\
				\hspace{0.5cm}(c)\;$\bar{V}(t\vert x_T,0)=\frac{\varepsilon\sigma^2}{2a_1(2+a_1\varepsilon)}\frac{\left((1+a_1\varepsilon)^{2(\lfloor{T/\varepsilon}\rfloor-m)}-1\right)\left((1+a_1\varepsilon)^{2m}-1\right)}{(1+a_1\varepsilon)^{2\lfloor{T/\varepsilon}\rfloor}-1}$ for $m=\lceil{(t-\Delta)/\varepsilon}\rceil$ and $t\in [0,T]$. \\
				\hspace{0.5cm}(d)\;$\bar{q}^\varepsilon(x,t\vert x_T,T;x_0,0)=\begin{cases}
					\delta(x-x_T) &\text{for}\; t\in[0,\Delta],\\
					\frac{1}{\sqrt{2\pi\bar{V}(t\vert x_T,0)}}\exp\left(-\frac{(x-\bar{M}(t\vert x_T,0))^2}{2\bar{V}(t\vert x_T,0)}\right) &\text{for}\; t\in(\Delta,\Delta+(\lfloor{T/\varepsilon}\rfloor-1)\varepsilon],\\
					\delta(x-x_0) &\text{for}\; t\in(\Delta+(\lfloor{T/\varepsilon}\rfloor-1)\varepsilon,T].
				\end{cases}$ \\
				3.\;\emph{Quantities related to} $\tilde{x}^{\varepsilon}$\\
				\hspace{0.5cm}(a)\;$\tilde{\mu}^{\varepsilon,\text{NPPD}}_{x,n\varepsilon}$ is a Gaussian measure with variance $\frac{\sigma^2}{2}\frac{(1+a_1\varepsilon)^{2(\lfloor{T/\varepsilon}\rfloor-(n+1))}-1}{(1+a_1\varepsilon)^{2(\lfloor{T/\varepsilon}\rfloor-n)}-1}$ and mean $\tilde{b}^\varepsilon(x,n\varepsilon)\triangleq\frac{(2+a_1\varepsilon)(1+a_1\varepsilon)^{\lfloor{T/\varepsilon}\rfloor-n-1}}{(1+a_1\varepsilon)^{2(\lfloor{T/\varepsilon}\rfloor-n)}-1}a_1x_T
				-\frac{(1+a_1\varepsilon)^{2(\lfloor{T/\varepsilon}\rfloor-n)-1}+1}{(1+a_1\varepsilon)^{2(\lfloor{T/\varepsilon}\rfloor-n)}-1}a_1x-\frac{(1+a_1\varepsilon)^{\lfloor{T/\varepsilon}\rfloor-n-1}-1}{(1+a_1\varepsilon)^{\lfloor{T/\varepsilon}\rfloor-n}+1}a_0$ for $n=0,1,\cdots,\lfloor{T/\varepsilon}\rfloor-1$.\\
				\hspace{0.5cm}(b)\;$\tilde{M}(t\vert x_0,0)=x_0\frac{(1+a_1\varepsilon)^{n}\left((1+a_1\varepsilon)^{2(\lfloor{T/\varepsilon}\rfloor-n)}-1\right)}{(1+a_1\varepsilon)^{2\lfloor{T/\varepsilon}\rfloor}-1}+x_T\frac{(1+a_1\varepsilon)^{\lfloor{T/\varepsilon}\rfloor-n}\left((1+a_1\varepsilon)^{2n}-1\right)}{(1+a_1\varepsilon)^{2\lfloor{T/\varepsilon}\rfloor}-1}-\frac{a_0}{a_1}\bigg(1-\frac{(1+a_1\varepsilon)^{n}\left((1+a_1\varepsilon)^{2(\lfloor{T/\varepsilon}\rfloor-n)}-1\right)}{(1+a_1\varepsilon)^{2\lfloor{T/\varepsilon}\rfloor}-1}-\frac{(1+a_1\varepsilon)^{\lfloor{T/\varepsilon}\rfloor-n}\left((1+a_1\varepsilon)^{2n}-1\right)}{(1+a_1\varepsilon)^{2\lfloor{T/\varepsilon}\rfloor}-1}\bigg)$ for $n=\lfloor{t/\varepsilon}\rfloor$ and $t\in [0,T]$.\\
				\hspace{0.5cm}(c)\;$\tilde{V}(t\vert x_0,0)=\frac{\varepsilon\sigma^2}{2a_1(2+a_1\varepsilon)}\frac{\left((1+a_1\varepsilon)^{2(\lfloor{T/\varepsilon}\rfloor-n)}-1\right)\left((1+a_1\varepsilon)^{2n}-1\right)}{(1+a_1\varepsilon)^{2\lfloor{T/\varepsilon}\rfloor}-1}$ for $n=\lfloor{t/\varepsilon}\rfloor$ and $t\in [0,T]$. \\
				\hspace{0.5cm}(d)\;$q^\varepsilon(x,t\vert x_T,T;x_0,0)=\begin{cases}
					\delta(x-x_0) &\text{for}\; t\in[0,\varepsilon),\\
					\frac{1}{\sqrt{2\pi\tilde{V}(t\vert x_0,0)}}\exp\left(-\frac{(x-\tilde{M}(t\vert x_0,0))^2}{2\tilde{V}(t\vert x_0,0)}\right) &\text{for}\; t\in[\varepsilon,\lfloor{T/\varepsilon}\rfloor\varepsilon),\\
					\delta(x-x_T) &\text{for}\; t\in[\lfloor{T/\varepsilon}\rfloor\varepsilon,T].
				\end{cases}$\\
				\midrule
				\bottomrule[1pt]
		\end{tabular}}
	\end{table}
\end{example}
\begin{remark}\label{rem050101}
	In both cases, the processes $x^\varepsilon$, $\bar{x}^{\varepsilon,\text{NPPD}}$ and $\tilde{x}^{\varepsilon,\text{NPPD}}$ possess continuous-time analogues. One can see Tables \ref{tab03} and \ref{tab04} for details.
	\begin{table}[ht]
		\centering
		\caption{Continuous-time analogues of $x^\varepsilon$, $\bar{x}^{\varepsilon,\text{NPPD}}$ and $\tilde{x}^{\varepsilon,\text{NPPD}}$ in the case that $b(x)=a_0$, $\kappa=2$.}
		\label{tab03}
		\resizebox{\linewidth}{!}{
			\begin{tabular}{m{16cm}}
				\toprule[1pt]
				\midrule
				1.\;\emph{Quantities related to the continuous-time analogue of} $x^{\varepsilon}$ \emph{(with the same notations)}\\
				\hspace{0.5cm}(a)\;In the continuous-time context, $x^{\varepsilon}$ is a Brownian motion defined by the SDE $\mathrm{d}x^{\varepsilon}(t)=a_0\mathrm{d}t+\sqrt{\varepsilon/2}\sigma\mathrm{d}w_t$ and the initial condition $x^{\varepsilon}(0)=x_0$.\\
				\hspace{0.5cm}(b)\;$M(t\vert x_s,s)=x_s+a_0(t-s)$ for $0\leq{s}\leq{t}\leq{T}$.\\
				\hspace{0.5cm}(c)\;$V(t\vert x_s,s)=\frac{\varepsilon\sigma^2(t-s)}{2}$ for $0\leq{s}\leq{t}\leq{T}$. \\
				\hspace{0.5cm}(d)\;$p^\varepsilon(x,t\vert x_s,s)=\begin{cases}
					\delta(x-x_s) &\text{for}\; t=s, \\
					\frac{1}{\sqrt{2\pi V(t\vert x_s,s)}}\exp\left(-\frac{(x-M(t\vert x_s,s))^2}{2V(t\vert x_s,s)}\right) &\text{for}\; t\in(s,T]\; \text{or}\; s\in[0,t).  
				\end{cases}$ \\ 
				\hspace{0.5cm}(e)\;$\phi_{\text{NOP}}(t;x_T,T;x_0)=x_0\frac{T-t}{T}+x_T\frac{t}{T}$ for $t\in[0,T]$. \\
				2.\;\emph{Quantities related to the continuous-time analogue of} $\bar{x}^{\varepsilon,\text{NPPD}}$ \emph{(with the same notations)}\\
				\hspace{0.5cm}(a)\;In the continuous-time context, $\bar{x}^{\varepsilon,\text{NPPD}}$ is a Brownian bridge defined by the SDE $\mathrm{d}\bar{x}^{\varepsilon,\text{NPPD}}(t)=\frac{x_0-\bar{x}^{\varepsilon,\text{NPPD}}(t)}{T-t}\mathrm{d}t+\sqrt{\varepsilon/2}\sigma\mathrm{d}w_t$ and the initial condition $\bar{x}^{\varepsilon,\text{NPPD}}(0)=x_T$.\\
				\hspace{0.5cm}(b)\;$\bar{M}(t\vert x_T,0)=x_T\frac{T-t}{T}+x_0\frac{t}{T}$ for $t\in [0,T]$. \\
				\hspace{0.5cm}(c)\;$\bar{V}(t\vert x_T,0)=\frac{\varepsilon\sigma^2\left(T-t\right)t}{2T}$ for $t\in [0,T]$. \\
				\hspace{0.5cm}(d)\;$\bar{q}^\varepsilon(x,t\vert x_T,T;x_0,0)=\begin{cases}
					\delta(x-x_T) &\text{for}\; t=0,\\
					\frac{1}{\sqrt{2\pi\bar{V}(t\vert x_T,0)}}\exp\left(-\frac{(x-\bar{M}(t\vert x_T,0))^2}{2\bar{V}(t\vert x_T,0)}\right) &\text{for}\; t\in(0,T),\\
					\delta(x-x_0) &\text{for}\; t=T.
				\end{cases}$ \\
				3.\;\emph{Quantities related to the continuous-time analogue of} $\tilde{x}^{\varepsilon,\text{NPPD}}$ \emph{(with the same notations)}\\
				\hspace{0.5cm}(a)\;In the continuous-time context, $\tilde{x}^{\varepsilon,\text{NPPD}}$ is a Brownian bridge defined by the SDE $\mathrm{d}\tilde{x}^{\varepsilon,\text{NPPD}}(t)=\frac{x_T-\tilde{x}^{\varepsilon,\text{NPPD}}(t)}{T-t}\mathrm{d}t+\sqrt{\varepsilon/2}\sigma\mathrm{d}w_t$ and the initial condition $\tilde{x}^{\varepsilon,\text{NPPD}}(0)=x_0$.\\
				\hspace{0.5cm}(b)\;$\tilde{M}(t\vert x_0,0)=x_0\frac{T-t}{T}+x_T\frac{t}{T}$ for $t\in [0,T]$.\\
				\hspace{0.5cm}(c)\;$\tilde{V}(t\vert x_0,0)=\frac{\varepsilon\sigma^2\left(T-t\right)t}{2T}$ for $t\in [0,T]$.\\
				\hspace{0.5cm}(d)\;$q^\varepsilon(x,t\vert x_T,T;x_0,0)=\begin{cases}
					\delta(x-x_0) &\text{for}\; t=0,\\
					\frac{1}{\sqrt{2\pi\tilde{V}(t\vert x_0,0)}}\exp\left(-\frac{(x-\tilde{M}(t\vert x_0,0))^2}{2\tilde{V}(t\vert x_0,0)}\right) &\text{for}\; t\in(0,T),\\
					\delta(x-x_T) &\text{for}\; t=T.
				\end{cases}$\\
				\midrule
				\bottomrule[1pt]
		\end{tabular}}
	\end{table}
	\begin{table}[ht]
		\centering
		\caption{Continuous-time analogues of $x^\varepsilon$, $\bar{x}^{\varepsilon,\text{NPPD}}$ and $\tilde{x}^{\varepsilon,\text{NPPD}}$ in the case that $b(x)=a_0+a_1 x$, $\kappa=2$.}
		\label{tab04}
		\resizebox{\linewidth}{!}{
			\begin{tabular}{m{16cm}}
				\toprule[1pt]
				\midrule
				1.\;\emph{Quantities related to the continuous-time analogue of} $x^{\varepsilon}$ \emph{(with the same notations)}\\
				\hspace{0.5cm}(a)\;In the continuous-time context, $x^{\varepsilon}$ is an Ornstein-Uhlenbeck process defined by the SDE $\mathrm{d}x^{\varepsilon}(t)=(a_0+a_1  x^{\varepsilon}(t))\mathrm{d}t+\sqrt{\varepsilon/2}\sigma\mathrm{d}w_t$ and the initial condition $x^{\varepsilon}(0)=x_0$.\\
				\hspace{0.5cm}(b)\;$M(t\vert x_s,s)=x_se^{a_1(t-s)}+\frac{a_0}{a_1}\left(e^{a_1(t-s)}-1\right)$ for $0\leq{s}\leq{t}\leq{T}$.\\
				\hspace{0.5cm}(c)\;$V(t\vert x_s,s)=\frac{\varepsilon\sigma^2}{4a_1}\left(e^{2a_1(t-s)}-1\right)$ for $0\leq{s}\leq{t}\leq{T}$. \\
				\hspace{0.5cm}(d)\;$p^\varepsilon(x,t\vert x_s,s)=\begin{cases}
					\delta(x-x_s) &\text{for}\; t=s, \\
					\frac{1}{\sqrt{2\pi V(t\vert x_s,s)}}\exp\left(-\frac{(x-M(t\vert x_s,s))^2}{2V(t\vert x_s,s)}\right) &\text{for}\; t\in(s,T]\; \text{or}\; s\in[0,t).  
				\end{cases}$ \\ 
				\hspace{0.5cm}(e)\;$\phi_{\text{NOP}}(t;x_T,T;x_0)=x_0\frac{\sinh(a_1(T-t))}{\sinh(a_1T)}+x_T\frac{\sinh(a_1t)}{\sinh(a_1T)}-\frac{a_0}{a_1}\left(1-\frac{\sinh(a_1(T-t))}{\sinh(a_1T)}-\frac{\sinh(a_1t)}{\sinh(a_1T)}\right)$ for $t\in[0,T]$. \\
				2.\;\emph{Quantities related to the continuous-time analogue of} $\bar{x}^{\varepsilon,\text{NPPD}}$ \emph{(with the same notations)}\\
				\hspace{0.5cm}(a)\;In the continuous-time context, $\bar{x}^{\varepsilon,\text{NPPD}}$ is an Ornstein-Uhlenbeck bridge defined by the SDE $\mathrm{d}\bar{x}^{\varepsilon,\text{NPPD}}(t)=
				\Big(\frac{1}{\sinh(a_1(T-t))}a_1x_0-\frac{\cosh(a_1(T-t))}{\sinh(a_1(T-t))}a_1\bar{x}^{\varepsilon,\text{NPPD}}(t)-a_0\frac{\cosh(a_1(T-t))-1}{\sinh(a_1(T-t))}\Big)\mathrm{d}t+\sqrt{\varepsilon/2}\sigma\mathrm{d}w_t$ and the initial condition $\bar{x}^{\varepsilon,\text{NPPD}}(0)=x_0$.\\
				\hspace{0.5cm}(b)\;$\bar{M}(t\vert x_T,0)=x_T\frac{\sinh(a_1(T-t))}{\sinh(a_1T)}+x_0\frac{\sinh(a_1t)}{\sinh(a_1T)}-\frac{a_0}{a_1}\left(1-\frac{\sinh(a_1(T-t))}{\sinh(a_1T)}-\frac{\sinh(a_1t)}{\sinh(a_1T)}\right)$ for $t\in [0,T]$. \\
				\hspace{0.5cm}(c)\;$\bar{V}(t\vert x_T,0)=\frac{\varepsilon\sigma^2}{2a_1}\frac{\sinh(a_1(T-t))\sinh(a_1t)}{\sinh(a_1T)}$ for $t\in [0,T]$. \\
				\hspace{0.5cm}(d)\;$\bar{q}^\varepsilon(x,t\vert x_T,T;x_0,0)=\begin{cases}
					\delta(x-x_T) &\text{for}\; t=0,\\
					\frac{1}{\sqrt{2\pi\bar{V}(t\vert x_T,0)}}\exp\left(-\frac{(x-\bar{M}(t\vert x_T,0))^2}{2\bar{V}(t\vert x_T,0)}\right) &\text{for}\; t\in(0,T),\\
					\delta(x-x_0) &\text{for}\; t=T.
				\end{cases}$ \\
				3.\;\emph{Quantities related to the continuous-time analogue of} $\tilde{x}^{\varepsilon,\text{NPPD}}$ \emph{(with the same notations)}\\
				\hspace{0.5cm}(a)\;In the continuous-time context, $\tilde{x}^{\varepsilon,\text{NPPD}}$ is an Ornstein-Uhlenbeck bridge defined by the SDE $\mathrm{d}\tilde{x}^{\varepsilon,\text{NPPD}}(t)=\Big(\frac{1}{\sinh(a_1(T-t))}a_1x_T-\frac{\cosh(a_1(T-t))}{\sinh(a_1(T-t))}a_1\tilde{x}^{\varepsilon,\text{NPPD}}(t)-a_0\frac{\cosh(a_1(T-t))-1}{\sinh(a_1(T-t))}\Big)\mathrm{d}t+\sqrt{\varepsilon/2}\sigma\mathrm{d}w_t$ and the initial condition $\tilde{x}^{\varepsilon,\text{NPPD}}(0)=x_0$.\\
				\hspace{0.5cm}(b)\;$\tilde{M}(t\vert x_0,0)=x_0\frac{\sinh(a_1(T-t))}{\sinh(a_1T)}+x_T\frac{\sinh(a_1t)}{\sinh(a_1T)}-\frac{a_0}{a_1}\left(1-\frac{\sinh(a_1(T-t))}{\sinh(a_1T)}-\frac{\sinh(a_1t)}{\sinh(a_1T)}\right)$ for $t\in [0,T]$.\\
				\hspace{0.5cm}(c)\;$\tilde{V}(t\vert x_0,0)=\frac{\varepsilon\sigma^2}{2a_1}\frac{\sinh(a_1(T-t))\sinh(a_1t)}{\sinh(a_1T)}$ for $t\in [0,T]$.\\
				\hspace{0.5cm}(d)\;$q^\varepsilon(x,t\vert x_T,T;x_0,0)=\begin{cases}
					\delta(x-x_0) &\text{for}\; t=0,\\
					\frac{1}{\sqrt{2\pi\tilde{V}(t\vert x_0,0)}}\exp\left(-\frac{(x-\tilde{M}(t\vert x_0,0))^2}{2\tilde{V}(t\vert x_0,0)}\right) &\text{for}\; t\in(0,T),\\
					\delta(x-x_T) &\text{for}\; t=T.
				\end{cases}$\\
				\midrule
				\bottomrule[1pt]
		\end{tabular}}
	\end{table}
\end{remark}

In instances where $b(x)$ is not affine or $\kappa\neq 2$, the explicit expression for $p^\varepsilon(x,t\vert x_s,s)$ is unavailable. Consequently, the necessary information regarding the measures $\bar{\mu}^{\varepsilon,\text{NPPD}}_{x,\Delta+m\varepsilon}$ and $\tilde{\mu}^{\varepsilon,\text{NPPD}}_{x,n\varepsilon}$ is lacking, making it difficult to verify whether the conditions stipulated in Propositions \ref{theo040101} and \ref{theo040102} are satisfied. In such a scenario, the emphasis shifts from theoretical proofs to numerically validating the conclusions of Proposition \ref{theo040103} by examining several examples, as presented below.
\begin{example}\label{exam050103}
	{$b(x)=\frac{x-x^3}{1+\vert x\vert^3}$, $\kappa=2$.}
	
	When $\kappa=2$, the measure $\mu_x$ is Gaussian. The Hamiltonian function can be expressed explicitly as 
	\begin{equation*}
		H(x,\alpha)=b(x)\alpha+\frac{1}{4}\sigma\alpha^2.
	\end{equation*}
	Note that $b(x)=\frac{x-x^3}{1+\vert x\vert^3}\in\mathcal{C}^2$. The deterministic system (\ref{eq010105}) has two stable equilibria $x^{sl}_{eq}=-1$ and $x^{sr}_{eq}=1$, separated by an unstable equilibrium $x^{u}_{eq}=0$. Take $x_0=-1$ and $x_T=1$. As illustrated in Fig. \ref{fig01}, the numerical results indicate that for each $T>0$, there exists a unique initial momentum $\alpha_0$, corresponding to a unique solution of the constrained Hamiltonian problem (\ref{eq020203}), and therefore corresponding to a unique NOP connecting $x_0$ with $x_T$ in the time span $T$. For $T=5$, the value of $\alpha_0$ is $0.671$. The associated solution is exhibited in phase space by the magenta solid line. Its projection is exactly the NOP and is displayed in Fig. \ref{fig02d} by the black dashed line. The NPPDs and their peak trajectories for $\varepsilon=0.2$, $0.05$ and $0.01$ are calculated by means of the algorithm provided in \ref{secA3} and are illustrated in Figs \ref{fig02a}, \ref{fig02b}, and \ref{fig02c}, respectively. Comparing these results with the NOP, one can see an apparent focusing effect of the NPPD onto the NOP as $\varepsilon\to{0}$.
	\begin{figure*}[htp]
		\centering
		\subfigure{\begin{minipage}[b]{.3\linewidth}
				
				\includegraphics[scale=0.3]{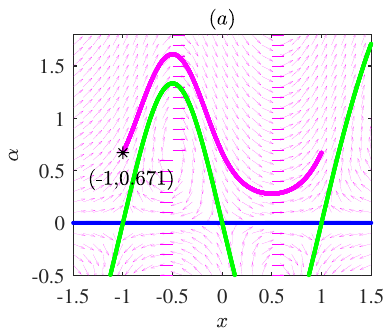}
			\end{minipage}
			\label{fig01a}
		}
		\subfigure{\begin{minipage}[b]{.3\linewidth}
				\includegraphics[scale=0.3]{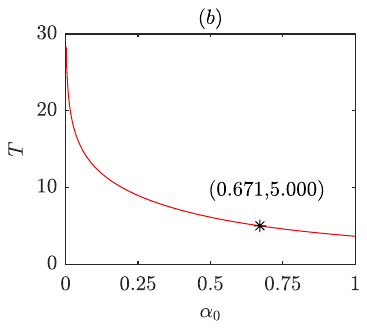}
			\end{minipage}
			\label{fig01b}
		}
		\caption{(a) The Hamiltonian vector field, together with the stable (blue) and unstable (green) manifolds of the fixed points $(-1,0)$ and $(1,0)$. The unique (magenta) solution of the constrained Hamiltonian problem for $x_0=-1$, $x_T=1$, $T=5$. (b) Plot of $T(\alpha_0)$ versus $\alpha_0$ for $x_0=-1$, $x_T=1$. The monotonicity indicates that each $T$ corresponds to a unique NOP.}
		\label{fig01}
	\end{figure*}
	\begin{figure*}[hbp]
		\centering
		\subfigure{\begin{minipage}[b]{.3\linewidth}
				\centering
				\includegraphics[scale=0.3]{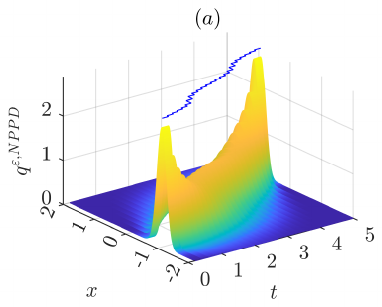}
			\end{minipage}
			\label{fig02a}
		}
		\subfigure{\begin{minipage}[b]{.3\linewidth}
				\centering
				\includegraphics[scale=0.3]{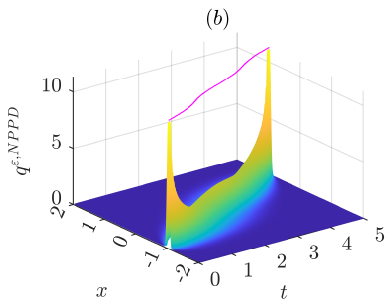}
			\end{minipage}
			\label{fig02b}
		}
		
		\subfigure{\begin{minipage}[b]{.3\linewidth}
				\centering
				\includegraphics[scale=0.3]{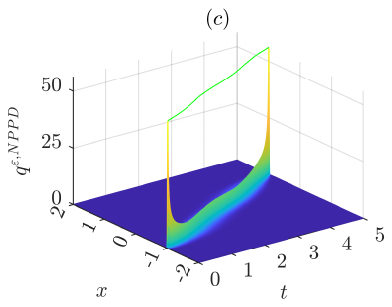}
			\end{minipage}
			\label{fig02c}
		}
		\subfigure{\begin{minipage}[b]{.3\linewidth}
				\centering
				\includegraphics[scale=0.3]{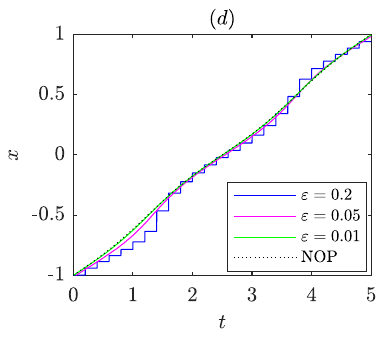}
			\end{minipage}
			\label{fig02d}
		}
		\caption{NPPDs and their peak trajectories for (a) $\varepsilon=0.2$, (b) $\varepsilon=0.05$, (c) $\varepsilon=0.01$. (d) Convergence of these peak trajectories to the NOP. Parameters: $x_0=-1$, $x_T=1$, $T=5$.}
		\label{fig02}
	\end{figure*}
\end{example}

\begin{example}\label{exam050104}
	{$b(x)=0$, $\kappa\neq2$.}
	
	When $\kappa\neq2$, in general, we do not have an explicit expression for the Hamiltonian function $H(x,\alpha)$. Fortunately, in the special case where $b(x)=0$, the NOP connecting $x_0$ with $x_T$ in the time span $T$ can be obtained explicitly. In this situation, $H(x,\alpha)$ does not depend on $x$ and can be written as $H(x,\alpha)=H^{*}(\alpha)$. From Eq. (\ref{eq020203}), we have $\dot{\alpha}(t)\equiv{0}$, so $\alpha(t)$ is constant. This implies that $\dot{\phi}_{\text{NOP}}(t)=\text{const}$. Consequently, for any given $x_0$, $x_T$ and $T$, the NOP is the line segment joining $x_0$ to $x_T$ with uniform speed, and is therefore given by $\phi_{\text{NOP}}(t)=x_0\frac{T-t}{T}+x_T\frac{t}{T}$ for $t\in[0,T]$. We set the parameters to be $x_0=0$, $x_T=1$, and $T=3$. As shown in Figs. \ref{fig03} and \ref{fig04}, an analogous focusing effect of the NPPD onto the deterministic NOP emerges.
		\begin{figure*}[htp]
		\centering
		\subfigure{\begin{minipage}[b]{.3\linewidth}
				\centering
				\includegraphics[scale=0.3]{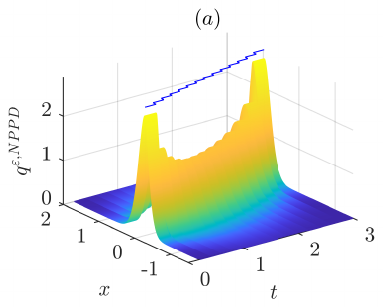}
			\end{minipage}
			\label{fig03a}
		}
		\subfigure{\begin{minipage}[b]{.3\linewidth}
				\centering
				\includegraphics[scale=0.3]{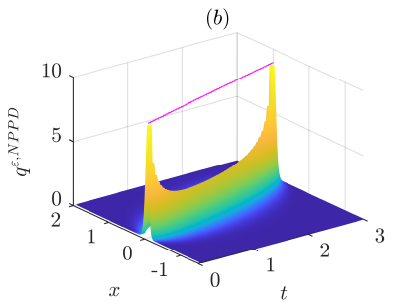}
			\end{minipage}
			\label{fig03b}
		}
		
		\subfigure{\begin{minipage}[b]{.3\linewidth}
				\centering
				\includegraphics[scale=0.3]{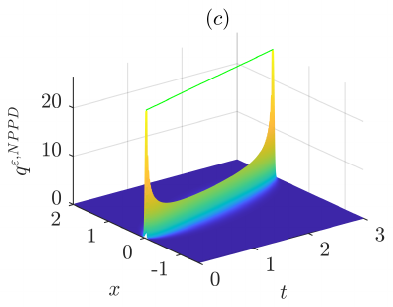}
			\end{minipage}
			\label{fig03c}
		}
		\subfigure{\begin{minipage}[b]{.3\linewidth}
				\centering
				\includegraphics[scale=0.3]{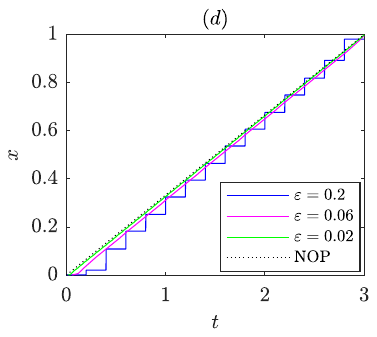}
			\end{minipage}
			\label{fig03d}
		}
		\caption{NPPDs and their peak trajectories for (a) $\varepsilon=0.2$, (b) $\varepsilon=0.06$, (c) $\varepsilon=0.02$. (d) Convergence of these peak trajectories to the NOP. Parameters: $x_0=0$, $x_T=1$, $T=3$, $\kappa=1.5$.}
		\label{fig03}
	\end{figure*}
	\begin{figure*}[hbp]
		\centering
		\subfigure{\begin{minipage}[b]{.3\linewidth}
				\centering
				\includegraphics[scale=0.3]{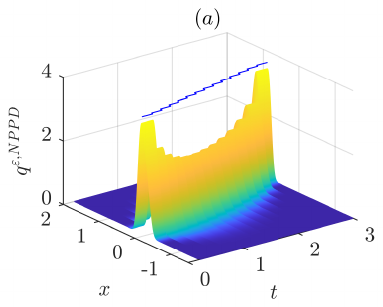}
			\end{minipage}
			\label{fig04a}
		}
		\subfigure{\begin{minipage}[b]{.3\linewidth}
				\centering
				\includegraphics[scale=0.3]{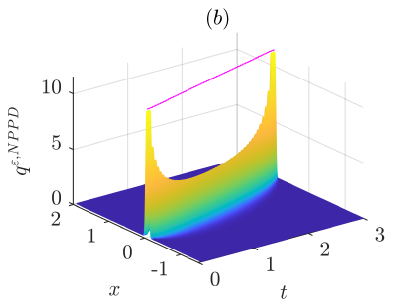}
			\end{minipage}
			\label{fig04b}
		}
		
		\subfigure{\begin{minipage}[b]{.3\linewidth}
				\centering
				\includegraphics[scale=0.3]{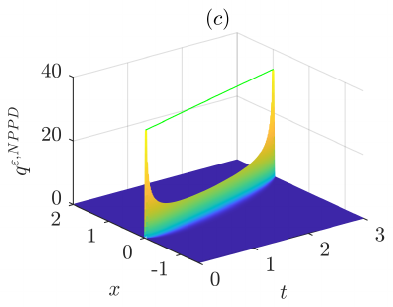}
			\end{minipage}
			\label{fig04c}
		}
		\subfigure{\begin{minipage}[b]{.3\linewidth}
				\centering
				\includegraphics[scale=0.3]{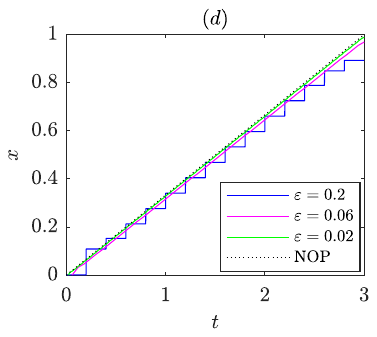}
			\end{minipage}
			\label{fig04d}
		}
		\caption{NPPDs and their peak trajectories for (a) $\varepsilon=0.2$, (b) $\varepsilon=0.06$, (c) $\varepsilon=0.02$. (d) Convergence of these peak trajectories to the NOP. Parameters: $x_0=0$, $x_T=1$, $T=3$, $\kappa=3$.}
		\label{fig04}
	\end{figure*}
\end{example}

\begin{example}\label{exam050105}
	{$b(x)=\frac{x-x^3}{1+\vert x\vert^3}$, $\kappa\neq2$.}
	
	In this particular instance, neither an explicit expression for $H(x,\alpha)$ nor an analytical form of the NOP $\{\phi_{\text{NOP}}(t;x_T,T;x_0):t\in[0,T]\}$ is available. All results must therefore be obtained through numerical approximation. 
	
	Note that the NOP $\{\phi_{\text{NOP}}(t;x_T,T;x_0):t\in[0,T]\}$ is of class $\mathcal{C}^2$. Consequently, $\alpha(t)=\nabla_xL(\phi_{\text{NOP}}(t),\dot{\phi}_{\text{NOP}}(t))$ is bounded. Together with the fact that $H(x,\alpha)$ is analytic in $\alpha$, we know that the NOP $\{\phi_{\text{NOP}}(t;x_T,T;x_0):t\in[0,T]\}$ can be approximated by the solution of Hamilton’s system (\ref{eq020203}), with the Hamiltonian $H(x,\alpha)$ replaced by its $2k$-th order approximation $H^{(2k)}(x,\alpha)$, which is given by
	\begin{equation*}
		H^{(2k)}(x,\alpha)\simeq b(x)\alpha+c_2\alpha^2+c_4\alpha^4\cdots+c_{2k}\alpha^{2k},
	\end{equation*}
	with
	\begin{equation*}
		c_2=\frac{\sigma^{2} \Gamma \left({3}/{\beta}\right)}{2 \Gamma \left({1}/{\beta}\right)},\;\;
		c_4=\frac{1}{24} \frac{\sigma^{4} \left(\Gamma \left({5}/{\beta}\right) \Gamma  \left({1}/{\beta}\right)-3 \Gamma^2  \left({3}/{\beta}\right)\right)}{\Gamma^2 \left({1}/{\beta}\right)},
	\end{equation*}
	\begin{equation*}
		c_6=\frac{1}{720} \frac{\sigma^{6} \left(\Gamma \left({7}/{\beta}\right) \Gamma^{2} \left({1}/{\beta}\right)-15 \Gamma \left({3}/{\beta}\right) \Gamma \left({5}/{\beta}\right) \Gamma \left({1}/{\beta}\right)+30 \Gamma^{3} \left({3}/{\beta}\right)\right)}{\Gamma^{3} \left({1}/{\beta}\right)},\cdots.
	\end{equation*}
	Take $x_0=-1$, $x_T=1$, and set $k=4$. Figs. \ref{fig05} and \ref{fig07} show that for each $T>0$, the approximated NOP connecting $x_0$ with $x_T$ in the time span $T$ is unique. Choosing $T=5$, the NPPDs and their peak positions for $\varepsilon=0.2$, $0.05$ and $0.01$ are plotted in Figs. \ref{fig06} and \ref{fig08}. The results clearly indicate a focusing effect of the NPPD onto the approximated NOP, and thus also onto the true NOP.
	\begin{figure*}[htp]
				\centering
				\subfigure{\begin{minipage}[b]{.3\linewidth}
								
								\includegraphics[scale=0.3]{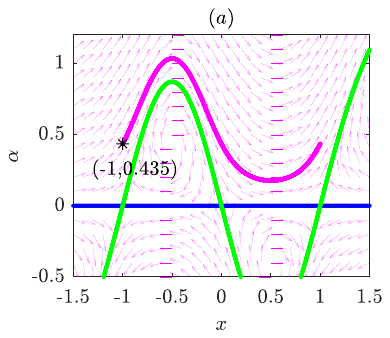}
							\end{minipage}
						\label{fig05a}
					}
				\subfigure{\begin{minipage}[b]{.3\linewidth}
								\includegraphics[scale=0.3]{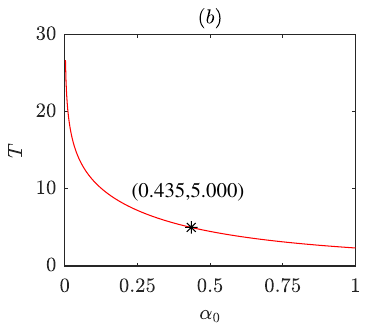}
							\end{minipage}
						\label{fig05b}
					}
				\caption{(a) The approximated Hamiltonian vector field, together with the stable (blue) and unstable (green) manifolds of the fixed points $(-1,0)$ and $(1,0)$. The unique (magenta) solution of the constrained Hamiltonian problem for $x_0=-1$, $x_T=1$, $T=5$, $\kappa=1.5$, $k=10$. (b) Plot of $T(\alpha_0)$ versus $\alpha_0$ for $x_0=-1$, $x_T=1$. The monotonicity indicates that each $T$ corresponds to a unique NOP.}
				\label{fig05}
	\end{figure*}
	\begin{figure*}[hbp]
		\centering
		\subfigure{\begin{minipage}[b]{.3\linewidth}
				\centering
				\includegraphics[scale=0.3]{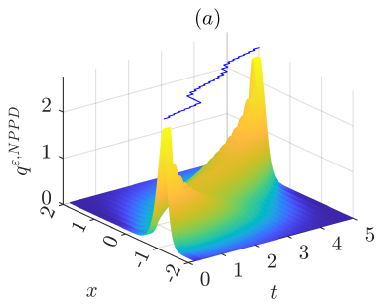}
			\end{minipage}
			\label{fig06a}
		}
		\subfigure{\begin{minipage}[b]{.3\linewidth}
				\centering
				\includegraphics[scale=0.3]{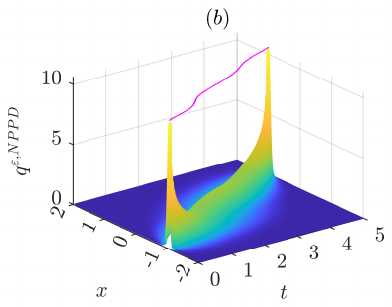}
			\end{minipage}
			\label{fig06b}
		}
		
		\subfigure{\begin{minipage}[b]{.3\linewidth}
				\centering
				\includegraphics[scale=0.3]{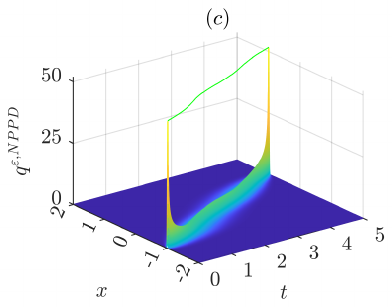}
			\end{minipage}
			\label{fig06c}
		}
		\subfigure{\begin{minipage}[b]{.3\linewidth}
				\centering
				\includegraphics[scale=0.3]{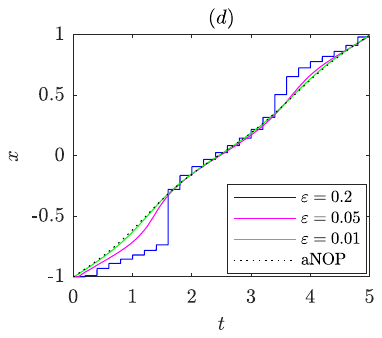}
			\end{minipage}
			\label{fig06d}
		}
		\caption{NPPDs and their peak trajectories for (a) $\varepsilon=0.2$, (b) $\varepsilon=0.05$, (c) $\varepsilon=0.01$. (d) Convergence of these peak trajectories to the approximated NOP (aNOP). Parameters: $x_0=-1$, $x_T=1$, $T=5$, $\kappa=1.5$.}
		\label{fig06}
	\end{figure*}
		\begin{figure*}[htp]
				\centering
				\subfigure{\begin{minipage}[b]{.3\linewidth}
								
								\includegraphics[scale=0.3]{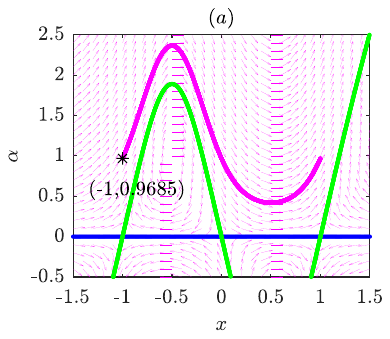}
							\end{minipage}
						\label{fig07a}
					}
				\subfigure{\begin{minipage}[b]{.3\linewidth}
								\includegraphics[scale=0.3]{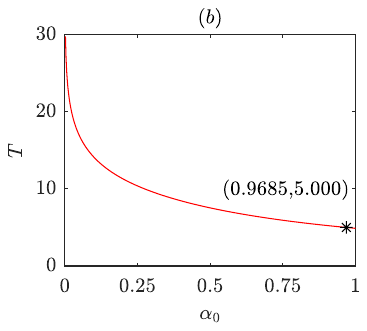}
							\end{minipage}
						\label{fig07b}
					}
				\caption{(a) The approximated Hamiltonian vector field, together with the stable (blue) and unstable (green) manifolds of the fixed points $(-1,0)$ and $(1,0)$. The unique (magenta) solution of the constrained Hamiltonian problem for $x_0=-1$, $x_T=1$, $T=5$, $\kappa=3$, $k=10$. (b) Plot of $T(\alpha_0)$ versus $\alpha_0$ for $x_0=-1$, $x_T=1$. The monotonicity indicates that each $T$ corresponds to a unique NOP.}
				\label{fig07}
			\end{figure*}
	\begin{figure*}[hbp]
		\centering
		\subfigure{\begin{minipage}[b]{.3\linewidth}
				\centering
				\includegraphics[scale=0.3]{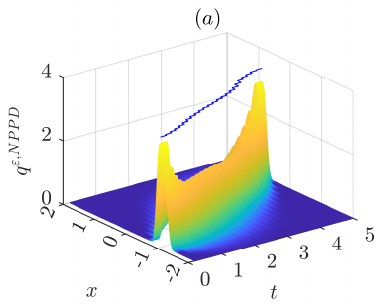}
			\end{minipage}
			\label{fig08a}
		}
		\subfigure{\begin{minipage}[b]{.3\linewidth}
				\centering
				\includegraphics[scale=0.3]{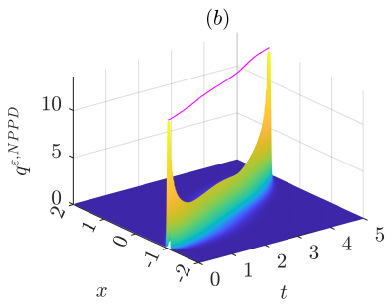}
			\end{minipage}
			\label{fig08b}
		}
		
		\subfigure{\begin{minipage}[b]{.3\linewidth}
				\centering
				\includegraphics[scale=0.3]{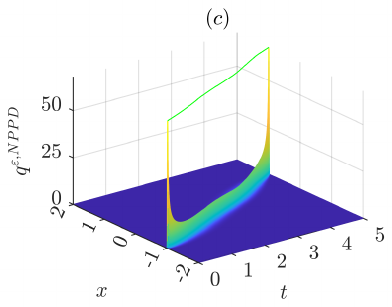}
			\end{minipage}
			\label{fig08c}
		}
		\subfigure{\begin{minipage}[b]{.3\linewidth}
				\centering
				\includegraphics[scale=0.3]{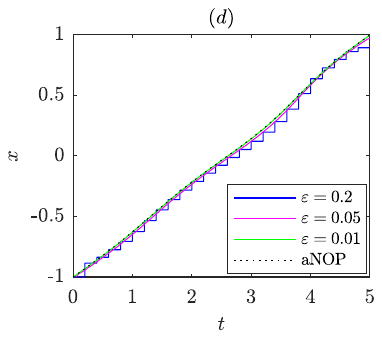}
			\end{minipage}
			\label{fig08d}
		}
		\caption{NPPDs and their peak trajectories for (a) $\varepsilon=0.2$, (b) $\varepsilon=0.05$, (c) $\varepsilon=0.01$. (d) Convergence of these peak trajectories to the approximated NOP (aNOP). Parameters: $x_0=-1$, $x_T=1$, $T=5$, $\kappa=3$.}
		\label{fig08}
	\end{figure*}
\end{example}
\begin{remark}
	Although all theoretical results in this paper require the uniqueness of the NOP, the focusing effect may still be observed when multiple NOPs coexist. To see this, take $x_0=-0.8$, $x_T=-0.2$ and $\kappa=2$. Fig. \ref{fig09b} shows that there are five primary cases for a Hamiltonian trajectory starting from $(x_0,\alpha_0)$ to reach the line $x=x_T$. For each case, the time required for the Hamiltonian trajectory to reach the target position is illustrated in Fig. \ref{fig09a}. As indicated in these figures, when $T$ is chosen sufficiently large, only Cases II and III are candidates for the possible NOP connecting $x_0$ with $x_T$ in the time span $T$. Denote by $(x^{\text{II}},\alpha^{\text{II}})$ and $(x^{\text{III}},\alpha^{\text{III}})$ the associated Hamiltonian trajectories in Cases II and III, respectively. Note that
	\begin{equation*}
		b(x)=b(-1-x),\quad\forall x\in(-1,1).
	\end{equation*}
	Then we obtain
	\begin{equation*}
		x^{\text{II}}(t)+x^{\text{III}}(T-t)=-1\quad\text{and}\quad\alpha^{\text{II}}(t)=\alpha^{\text{III}}(T-t),\quad \forall t\in[0,T].
	\end{equation*}
	Substituting these relations into (\ref{eq020204}) (or (\ref{eq020205})) gives
	\begin{equation*}
		S(x_T,T\vert x_0)=I_{[0,T]}(\{x^{\text{II}}(t):t\in[0,T]\})=I_{[0,T]}(\{x^{\text{III}}(t):t\in[0,T]\}).
	\end{equation*}
	Consequently, there exist two coexisting NOPs connecting $x_0$ with $x_T$ in the time span $T$. Set $T=10$. The two coexisting NOPs are plotted in Fig. \ref{fig10d} as black dashed lines. The NPPDs and their peak trajectories for $x_T=-0.17,\;-0.18,\;-0.19$ are illustrated in Figs. \ref{fig10a}, \ref{fig10b}, and \ref{fig10c}, respectively, revealing a clear focusing effect of the NPPD onto the coexisting NOPs. In addition, as can be seen from Figs. \ref{fig11} and \ref{fig12}, changing the parameter $\kappa$ to $1.5$ and $3$ yields the same focusing effect.
	
	This phenomenon can be explained theoretically. For given $\boldsymbol{x}_0$, $\boldsymbol{x}_T$, $T$, suppose the NOP connecting $\boldsymbol{x}_0$ with $\boldsymbol{x}_T$ in the time span $T$ is not unique, and let $\boldsymbol{\phi}^{(1)}$ and $\boldsymbol{\phi}^{(2)}$ denote two coexisting NOPs. Analogous to Proposition \ref{theo020202}, it can be shown that the following properties hold:
	\begin{itemize}
		\item $S(\boldsymbol{y},v\vert\boldsymbol{x}_0)$ and $S(\boldsymbol{x}_T,T-v\vert\boldsymbol{y})$ are smooth in $(\boldsymbol{y},v)$ at the points $(\boldsymbol{\phi}^{(1)}(t),t)$ and $(\boldsymbol{\phi}^{(2)}(t),t)$ for $t\in(0,T)$;
		\item Let $\boldsymbol{\alpha}^{(1)}$ and $\boldsymbol{\alpha}^{(2)}$ be the associated momentum functions defined in Proposition \ref{theo020202}(d). Then
		\begin{equation*}
			\begin{aligned}
				\lim_{t\to T}\nabla_{\boldsymbol{y}}S(\boldsymbol{y},v\vert\boldsymbol{x}_0)\vert_{(\boldsymbol{y},v)=(\boldsymbol{\phi}^{(1)}(t),t)}=&\boldsymbol{\alpha}^{(1)}(T),\\
				\lim_{t\to T}\nabla_{\boldsymbol{y}}S(\boldsymbol{y},v\vert\boldsymbol{x}_0)\vert_{(\boldsymbol{y},v)=(\boldsymbol{\phi}^{(2)}(t),t)}=&\boldsymbol{\alpha}^{(2)}(T),\\
			\end{aligned}
		\end{equation*}
		\begin{equation*}
			\begin{aligned}
				\lim_{t\to 0}\nabla_{\boldsymbol{y}}S(\boldsymbol{x}_T,T-v\vert\boldsymbol{y})\vert_{(\boldsymbol{y},v)=(\boldsymbol{\phi}^{(1)}(t),t)}=&-\boldsymbol{\alpha}^{(1)}(0),\\
				\lim_{t\to 0}\nabla_{\boldsymbol{y}}S(\boldsymbol{x}_T,T-v\vert\boldsymbol{y})\vert_{(\boldsymbol{y},v)=(\boldsymbol{\phi}^{(2)}(t),t)}=&-\boldsymbol{\alpha}^{(2)}(0),\\
			\end{aligned}
		\end{equation*}
		but
		\begin{equation*}
			\begin{aligned}
				\boldsymbol{\alpha}^{(1)}(T)\neq\boldsymbol{\alpha}^{(2)}(T)\quad\text{and}\quad \boldsymbol{\alpha}^{(1)}(0)\neq\boldsymbol{\alpha}^{(2)}(0).
			\end{aligned}
		\end{equation*}
	\end{itemize}
	This implies:
	\begin{itemize}
		\item The vector field in (\ref{eq020208}) is smooth at points $(\boldsymbol{\phi}^{(1)}(t),t)$ and $(\boldsymbol{\phi}^{(2)}(t),t)$ for $t\in(0,T)$, but discontinuous at the terminal point $(x_T,T)$;
		\item The vector field in (\ref{eq020209}) is smooth at points $(\boldsymbol{\phi}^{(1)}(t),t)$ and $(\boldsymbol{\phi}^{(2)}(t),t)$ for $t\in(0,T)$, but discontinuous at the initial point $(x_0,0)$.
	\end{itemize}
	Consequently, the existence and uniqueness theorem fails, and equations (\ref{eq020208}) and (\ref{eq020209}) admit multiple solutions, each corresponding to a NOP. In this situation, the focusing effect may be explained as follows:
	\begin{itemize}
		\item $\bar{\boldsymbol{x}}^{\varepsilon,\text{NPPD}}(t)$ converges in probability to a set of deterministic trajectories that are solutions of (\ref{eq040101});
		\item $\tilde{\boldsymbol{x}}^{\varepsilon,\text{NPPD}}$ converges in probability to a set of deterministic trajectories that are solutions of (\ref{eq040103}).
	\end{itemize}
	This situation is rather intricate. A rigorous proof is left for future work.
	\begin{figure*}[htp]
		\centering
		\subfigure{\begin{minipage}[b]{.3\linewidth}
				
				\includegraphics[scale=0.3]{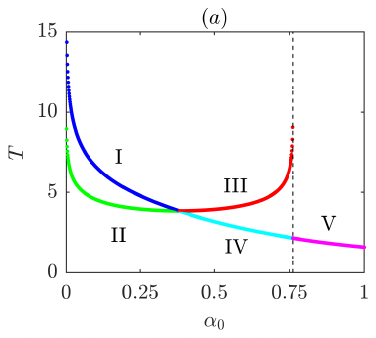}
			\end{minipage}
			\label{fig09a}
		}
		\subfigure{\begin{minipage}[b]{.3\linewidth}
				\includegraphics[scale=0.3]{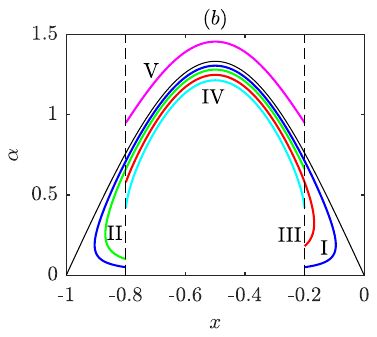}
			\end{minipage}
			\label{fig09b}
		}
		\caption{(a) Time required for a Hamiltonian trajectory starting from $(x_0,\alpha_0)$ to reach the line $x=x_T$. Five primary cases exist. Their corresponding Hamiltonian orbits are shown in (b). Parameters: $x_0=-0.8$, $x_T=-0.2$, $\kappa=2$.}
		\label{fig09}
	\end{figure*}
	\begin{figure*}[hbp]
		\centering
		\subfigure{\begin{minipage}[b]{.3\linewidth}
				\centering
				\includegraphics[scale=0.3]{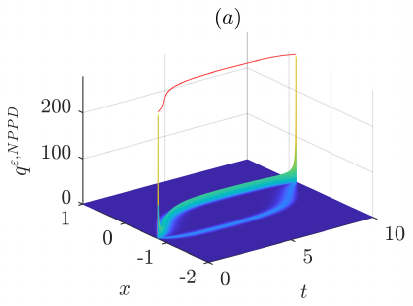}
			\end{minipage}
			\label{fig10a}
		}
		\subfigure{\begin{minipage}[b]{.3\linewidth}
				\centering
				\includegraphics[scale=0.3]{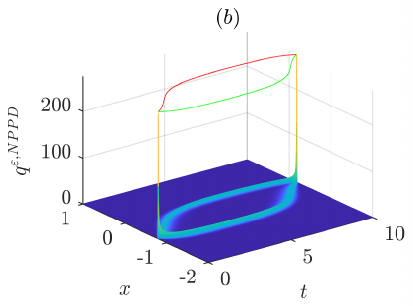}
			\end{minipage}
			\label{fig10b}
		}
		
		\subfigure{\begin{minipage}[b]{.3\linewidth}
				\centering
				\includegraphics[scale=0.3]{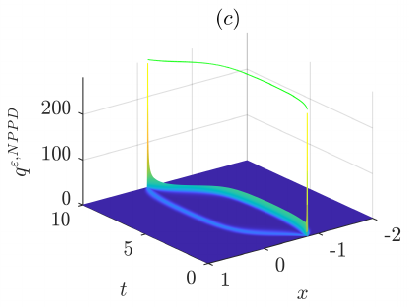}
			\end{minipage}
			\label{fig10c}
		}
		\subfigure{\begin{minipage}[b]{.3\linewidth}
				\centering
				\includegraphics[scale=0.3]{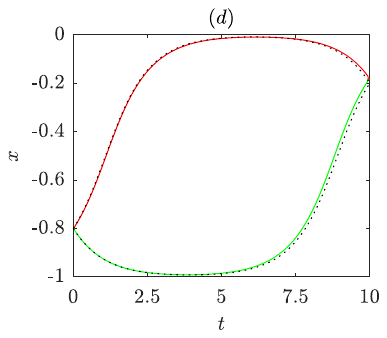}
			\end{minipage}
			\label{fig10d}
		}
		\caption{NPPDs and their peak trajectories for (a) $x_T=-0.17$, (b) $x_T=-0.18$, (c) $x_T=-0.19$, illustrating the mutation of the peak trajectory as $x_T$ varies. (d) The two peak trajectories in (b) coincide with the two coexisting NOPs from Cases II and III of Figure \ref{fig09b}. Parameters: $x_0=-0.8$, $T=10$, $\kappa=2$, $\varepsilon=0.002$.}
		\label{fig10}
	\end{figure*}
	\begin{figure*}[htp]
		\centering
		\subfigure{\begin{minipage}[b]{.3\linewidth}
				
				\includegraphics[scale=0.3]{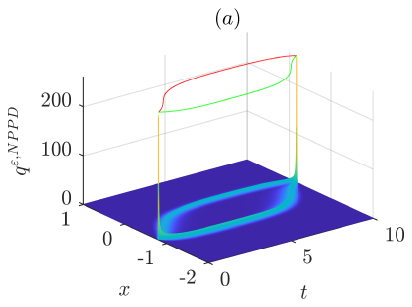}
			\end{minipage}
			\label{fig11a}
		}
		\subfigure{\begin{minipage}[b]{.3\linewidth}
				\includegraphics[scale=0.3]{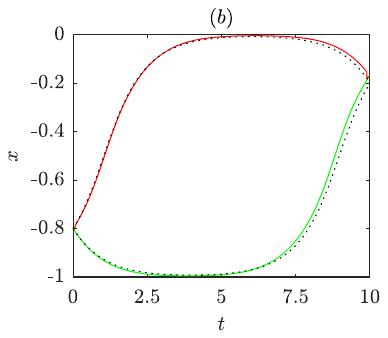}
			\end{minipage}
			\label{fig11b}
		}
		\caption{(a) NPPD and its peak trajectories for $x_T=-0.17$. (b) The two peak trajectories coincide with the two coexisting NOPs. Parameters: $x_0=-0.8$, $T=10$, $\kappa=1.5$, $\varepsilon=0.002$.}
		\label{fig11}
	\end{figure*}
	\begin{figure*}[htp]
		\centering
		\subfigure{\begin{minipage}[b]{.3\linewidth}
				
				\includegraphics[scale=0.3]{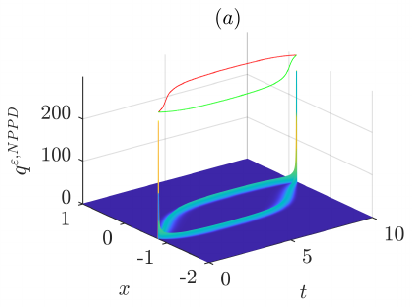}
			\end{minipage}
			\label{fig12a}
		}
		\subfigure{\begin{minipage}[b]{.3\linewidth}
				\includegraphics[scale=0.3]{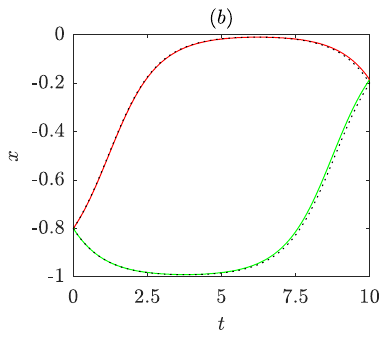}
			\end{minipage}
			\label{fig12b}
		}
		\caption{(a) NPPD and its peak trajectories for $x_T=-0.185$. (b) The two peak trajectories coincide with the two coexisting NOPs. Parameters: $x_0=-0.8$, $T=10$, $\kappa=3$, $\varepsilon=0.002$.}
		\label{fig12}
	\end{figure*}
\end{remark}

\section{Conclusions}\label{sec07}
The prehistorical description of optimal fluctuations for discrete-time Markov jump processes in non-stationary settings has been presented in this paper. Specifically, the time-reversed formula for discrete-time Markov jump processes was derived. By examining the time reversal of a specific family of probability distributions, the NPPD and its time-reversed counterpart were shown to be closely related to the time-reversed processes $\bar{\boldsymbol{x}}^{\varepsilon,\text{NPPD}}(t)$ and $\tilde{\boldsymbol{x}}^{\varepsilon,\text{NPPD}}(t)$, respectively. In the weak noise limit, both $\bar{\boldsymbol{x}}^{\varepsilon,\text{NPPD}}(t)$ and $\tilde{\boldsymbol{x}}^{\varepsilon,\text{NPPD}}(t)$ become nearly deterministic. Consequently, the focusing effect of the NPPD onto the NOP was then proved by means of the law of large numbers for the reversed processes. These results indicate that optimal fluctuations in discrete-time Markov jump processes exhibit strong qualitative similarities to those in the diffusion model (\ref{eq010103}) and the continuous-time Markov chain (\ref{eq010104}).
\section*{Declarations}
The authors declare that there is no conflict of interest related to the content of this article.

\section*{Acknowledgments}
The authors acknowledge support from the National Natural Science Foundation of China (Nos. 12172167, 12202195, 12302035, 12572038), the Jiangsu Funding Program for Excellent Postdoctoral Talent (No. 2023ZB591).

\appendix
\section{Proof of Lemma \ref{theo020203}}\label{secA1}
\begin{proof}[Proof of Lemma \ref{theo020203}]
	According to the given assumptions, for any compact set $F\subset\mathbb{R}^d$ and $\gamma>0$, there exist $\varepsilon_{1}$ and $\delta_{1}$ so that for all $\varepsilon<\varepsilon_1$ and $\delta<\delta_1$,
	\begin{equation*}
		\left\vert\frac{k^{\varepsilon,\delta}(\boldsymbol{x},t\vert\boldsymbol{x}_0,0)}{f^{\varepsilon,\delta}}-K(\boldsymbol{x},t\vert\boldsymbol{x}_0)\right\vert<\gamma,
	\end{equation*}
	\begin{equation*}
		\left\vert\frac{k^{\varepsilon,\delta}(\boldsymbol{x}+\varepsilon\boldsymbol{z},t-\varepsilon\vert\boldsymbol{x}_0,0)}{f^{\varepsilon,\delta}}-K(\boldsymbol{x}+\varepsilon\boldsymbol{z},t-\varepsilon\vert\boldsymbol{x}_0)\right\vert<\gamma,
	\end{equation*}
	and
	\begin{equation*}
		\left\vert{K}(\boldsymbol{x}+\varepsilon\boldsymbol{z},t-\varepsilon\vert\boldsymbol{x}_0)-K(\boldsymbol{x},t\vert\boldsymbol{x}_0)\right\vert<\gamma.
	\end{equation*}
	uniformly for $(\boldsymbol{x},t)\in\Sigma_{\delta_{0},[T-T^{*},T]}\times[T-T^{*},T]$ and $\boldsymbol{z}\in{F}$. Putting all these facts together, we obtain
	\begin{equation*}
		\lim_{\varepsilon\to{0},\delta\to{0}}\frac{k^{\varepsilon,\delta}(\boldsymbol{x}+\varepsilon\boldsymbol{z},t-\varepsilon\vert\boldsymbol{x}_0,0)}{k^{\varepsilon,\delta}(\boldsymbol{x},t\vert\boldsymbol{x}_0,0)}=1, 
	\end{equation*}
	uniformly for $(\boldsymbol{x},t)\in\Sigma_{\delta_{0},[T-T^{*},T]}\times[T-T^{*},T]$ and $\boldsymbol{z}\in{F}$. 
	
	Note that
	\begin{equation*}
		{P}^\varepsilon_{\boldsymbol{x}_0}\left\{\boldsymbol{x}^{\varepsilon}(t)\in{B}_{\delta}(\boldsymbol{x})\right\}=k^{\varepsilon,\delta}(\boldsymbol{x},t\vert\boldsymbol{x}_0,0)\exp\left(-\varepsilon^{-1}\inf_{\boldsymbol{y}\in{B}_{\delta}(\boldsymbol{x})}S(\boldsymbol{y},t\vert\boldsymbol{x}_0)\right),
	\end{equation*}
	\begin{equation*}
		\begin{aligned}
			{P}^\varepsilon_{\boldsymbol{x}_0}\left\{\boldsymbol{x}^{\varepsilon}(t-\varepsilon)\in{B}_{\delta}(\boldsymbol{x}+\varepsilon\boldsymbol{z})\right\}=&k^{\varepsilon,\delta}(\boldsymbol{x}+\varepsilon\boldsymbol{z},t-\varepsilon\vert\boldsymbol{x}_0,0)\\
			&\times\exp\left(-\varepsilon^{-1}\inf_{\boldsymbol{y}\in{B}_{\delta}(\boldsymbol{x}+\varepsilon\boldsymbol{z})}S(\boldsymbol{y},t-\varepsilon\vert\boldsymbol{x}_0)\right).
		\end{aligned}
	\end{equation*}
	Hence, Eq. (\ref{eq020212}) follows if we can show that
	\begin{equation}
		\begin{aligned}
			\lim_{\varepsilon\to{0}}\varepsilon^{-1}\bigg(\inf_{\boldsymbol{y}\in{B}_{\delta}(\boldsymbol{x})}S(\boldsymbol{y},t\vert\boldsymbol{x}_0)&-\inf_{\boldsymbol{y}\in{B}_{\delta}(\boldsymbol{x}+\varepsilon\boldsymbol{z})}S(\boldsymbol{y},t-\varepsilon\vert\boldsymbol{x}_0)\bigg)\\
			&=-\boldsymbol{z}\cdot\nabla_{\boldsymbol{x}}S(\boldsymbol{x},t\vert\boldsymbol{x}_0)+\frac{\partial}{\partial{t}}S(\boldsymbol{x},t\vert\boldsymbol{x}_0).
		\end{aligned}
	\end{equation}
	
	Since $S(\boldsymbol{x},t\vert\boldsymbol{x}_0)$ is at least $\mathcal{C}^{2}$ in $\Sigma_{\delta_{0},[T-T^{*},T]}\times[T-T^{*},T]$, we have
	\begin{equation*}
		\inf_{\boldsymbol{y}\in{B}_{\delta}(\boldsymbol{x})}S(\boldsymbol{y},t\vert\boldsymbol{x}_0)-S(\boldsymbol{x},t\vert\boldsymbol{x}_0)=-\left\vert\nabla_{\boldsymbol{x}}S(\boldsymbol{x},t\vert\boldsymbol{x}_0)\right\vert\delta+{O}(\delta^2)
	\end{equation*}
	\begin{equation*}
		\inf_{\boldsymbol{y}\in{B}_{\delta}(\boldsymbol{x}+\varepsilon\boldsymbol{z})}S(\boldsymbol{y},t-\varepsilon\vert\boldsymbol{x}_0)-S(\boldsymbol{x}+\varepsilon\boldsymbol{z},t-\varepsilon\vert\boldsymbol{x}_0)=-\left\vert\nabla_{\boldsymbol{x}}S(\boldsymbol{x}+\varepsilon\boldsymbol{z},t-\varepsilon\vert\boldsymbol{x}_0)\right\vert\delta+{O}(\delta^2)
	\end{equation*}
	and 
	\begin{equation*}
		\lim_{\varepsilon\to{0}}\frac{S(\boldsymbol{x},t\vert\boldsymbol{x}_0)-S(\boldsymbol{x}+\varepsilon\boldsymbol{z},t-\varepsilon\vert\boldsymbol{x}_0)}{\varepsilon}=-\boldsymbol{z}\cdot\nabla_{\boldsymbol{x}}S(\boldsymbol{x},t\vert\boldsymbol{x}_0)+\frac{\partial}{\partial{t}}S(\boldsymbol{x},t\vert\boldsymbol{x}_0).
	\end{equation*}
	Combining them with 
	\begin{equation*}
		\left\vert\nabla_{\boldsymbol{x}}S(\boldsymbol{x},t\vert\boldsymbol{x}_0)\right\vert-\left\vert\nabla_{\boldsymbol{x}}S(\boldsymbol{x}+\varepsilon\boldsymbol{z},t-\varepsilon\vert\boldsymbol{x}_0)\right\vert={O}(\varepsilon)
	\end{equation*}
	we then prove Eq. (\ref{eq020212}). The convergence is uniform for $(\boldsymbol{x},t)\in\Sigma_{\delta_{0},[T-T^{*},T]}\times[T-T^{*},T]$ and $\boldsymbol{z}\in{F}$, as the assumptions hold uniformly.
	
	Finally, since
	\begin{equation*}
		\frac{p^{\varepsilon}(\boldsymbol{x}+\varepsilon\boldsymbol{z},t-\varepsilon\vert\boldsymbol{x}_0,0)}{p^{\varepsilon}(\boldsymbol{x},t\vert\boldsymbol{x}_0,0)}=\lim_{\delta\to{0}}\frac{{P}^\varepsilon_{\boldsymbol{x}_0}\left\{\boldsymbol{x}^{\varepsilon}(t-\varepsilon)\in{B}_{{\delta}}(\boldsymbol{x}+\varepsilon\boldsymbol{z})\right\}}{{P}^\varepsilon_{\boldsymbol{x}_0}\left\{\boldsymbol{x}^{\varepsilon}(t)\in{B}_{{\delta}}(\boldsymbol{x})\right\}}.
	\end{equation*}
	Eq. (\ref{eq020213}) follows from Eq. (\ref{eq020212}). This completes the proof.
\end{proof}
\section{Proof of Proposition \ref{theo040101}}\label{secA2}
\begin{lemma}\label{theoA20101}
	Suppose that the conditions in Proposition \ref{theo040101} hold. Define
	\begin{equation*}
		\nu\triangleq\inf\left\{m\in\{0,1,\cdots,\lfloor{T/\varepsilon}\rfloor\}:(\bar{\boldsymbol{x}}^{\varepsilon,\text{NPPD}}(\Delta+m\varepsilon),\Delta+m\varepsilon)\notin\bar{\Sigma}_{\delta_{0},[0,T^{*}]}\times[0,T)\right\}.
	\end{equation*}
	Then in order to show that
	\begin{equation*}
		\lim_{\varepsilon\to{0}}P^{\varepsilon}\left\{\max_{1\leq{m}\leq\left\lceil{(T^{*}-\Delta)/{\varepsilon}}\right\rceil}\big\vert\bar{\boldsymbol{x}}^{\varepsilon,\text{NPPD}}(\Delta+(m\wedge\nu)\varepsilon)-\bar{\boldsymbol{x}}^{0}(\Delta+(m\wedge\nu)\varepsilon)\big\vert>\delta\right\}=0,
	\end{equation*}
	for every sufficiently small $\delta>0$, it suffices to show that
	\begin{equation*}
		\begin{aligned}
			\lim_{\varepsilon\to{0}}P^{\varepsilon}\bigg\{\max_{1\leq{m}\leq\left\lceil{(T^{*}-\Delta)/{\varepsilon}}\right\rceil}&\bigg\vert\bar{\boldsymbol{x}}^{\varepsilon,\text{NPPD}}(\Delta+(m\wedge\nu)\varepsilon)-\boldsymbol{x}_T\\&-\varepsilon\sum_{i=0}^{m\wedge\nu-1}\bar{\boldsymbol{b}}^{\varepsilon}(\bar{\boldsymbol{x}}^{\varepsilon,\text{NPPD}}(\Delta+i\varepsilon),\Delta+i\varepsilon)\bigg\vert>\delta\bigg\}=0
		\end{aligned}
	\end{equation*}
	for every sufficiently small $\delta>0$.
\end{lemma}
\begin{proof}
	By the definition of $\nu$, we know that $\bar{\boldsymbol{x}}^{\varepsilon,\text{NPPD}}(\Delta+\nu\varepsilon)\notin\bar{\Sigma}_{\delta_{0},[0,T^{*}]}$ if $\nu<\lfloor{T/\varepsilon}\rfloor$. Let $\hat{\boldsymbol{x}}^{\varepsilon}(\Delta+m\varepsilon)\triangleq\bar{\boldsymbol{x}}^{\varepsilon,\text{NPPD}}(\Delta+(m\wedge\nu)\varepsilon)$ for $m=0,1,\cdots,\lfloor{T/\varepsilon}\rfloor$. Then $\hat{\boldsymbol{x}}^{\varepsilon}(\Delta+m\varepsilon)$ is a Markov chain, i.e., $\hat{\boldsymbol{x}}^{\varepsilon}(\Delta)=\bar{\boldsymbol{x}}^{\varepsilon,\text{NPPD}}(\Delta)$, and for $m=1,2,\cdots,\lfloor{T/\varepsilon}\rfloor$,
	\begin{equation*}
		\begin{aligned}
			\hat{\boldsymbol{x}}^{\varepsilon}(\Delta+m\varepsilon)&=\bar{\boldsymbol{x}}^{\varepsilon,\text{NPPD}}(\Delta+(m\wedge\nu)\varepsilon)\\
			&=\bar{\boldsymbol{x}}^{\varepsilon,\text{NPPD}}(\Delta)+\varepsilon\sum_{i=0}^{m\wedge\nu-1}\bar{\boldsymbol{Z}}^{\varepsilon,\text{NPPD}}_{i}\\
			&=\bar{\boldsymbol{x}}^{\varepsilon,\text{NPPD}}(\Delta)+\varepsilon\sum_{i=0}^{m\wedge\nu-1}\bar{\boldsymbol{Z}}^{\varepsilon,\text{NPPD}}_{i\wedge\nu}\\
			&=\hat{\boldsymbol{x}}^{\varepsilon}(\Delta)+\varepsilon\sum_{i=0}^{m\wedge\nu-1}\hat{\boldsymbol{Z}}^{\varepsilon}_{i\wedge\nu}\\
			&=\hat{\boldsymbol{x}}^{\varepsilon}(\Delta)+\varepsilon\sum_{i=0}^{m-1}\hat{\boldsymbol{Z}}^{\varepsilon}_{i\wedge\nu}\\
			&=\hat{\boldsymbol{x}}^{\varepsilon}(\Delta)+\varepsilon\sum_{i=0}^{m-1}\hat{\boldsymbol{Z}}^{\varepsilon}_{i}.
		\end{aligned}
	\end{equation*}
	where $\{\hat{\boldsymbol{Z}}^{\varepsilon}_{m}\}$ is a jump process whose conditional jump probability at time $m$ is non-stationary (i.e., dependent on the variable $m$), independent of $\hat{\boldsymbol{Z}}^{\varepsilon}_{k}$ for $k<m$, dependent only on the current state of $\hat{\boldsymbol{x}}^{\varepsilon}(\Delta+m\varepsilon)$ (i.e., independent of previous states of $\hat{\boldsymbol{x}}^{\varepsilon}(\Delta+k\varepsilon)$ for $k<m$), and is given by $\hat{\mu}_{\boldsymbol{x},\Delta+m\varepsilon}^{\varepsilon}$. That is, for any Borel set $D\subset\mathbb{R}^d$ and $m=0,1,\cdots,\lfloor{T/\varepsilon}\rfloor-1$,
	\begin{equation*}
		P^{\varepsilon}\left(\hat{\boldsymbol{Z}}^{\varepsilon}_{m}\in{D}\vert\hat{\boldsymbol{x}}^{\varepsilon}(\Delta+m\varepsilon)=\boldsymbol{x},\hat{\mathcal{F}}_{<m}\right)=\hat{\mu}^{\varepsilon}_{\boldsymbol{x},\Delta+m\varepsilon}(D),
	\end{equation*}
	with
	\begin{equation*}
		\hat{\mu}_{\boldsymbol{x},\Delta+m\varepsilon}^{\varepsilon}(D)\triangleq\begin{cases}
			\bar{\mu}_{\boldsymbol{x},\Delta+m\varepsilon}^{\varepsilon,\text{NPPD}}(D), & \text{if}\;\boldsymbol{x}\in\bar{\Sigma}_{\delta_{0},[0,T^{*}]},\\
			\int_{D}\delta(\boldsymbol{z})\lambda(\mathrm{d}\boldsymbol{z}),& \text{otherwise},
		\end{cases}
	\end{equation*}
	where $\hat{\mathcal{F}}_{<m}$ denotes the history of $\hat{\boldsymbol{x}}^{\varepsilon}(\Delta+k\varepsilon)$ before $m$. (In other words, $\{\hat{\boldsymbol{Z}}^{\varepsilon}_{m}\}$ denote the increments of the stopped process $\hat{\boldsymbol{x}}^{\varepsilon}(\Delta+m\varepsilon)$. If $m<\nu$, we know that $\hat{\boldsymbol{x}}^{\varepsilon}(\Delta+m\varepsilon)\in\bar{\Sigma}_{\delta_{0},[0,T^{*}]}$, and therefore $\hat{\boldsymbol{Z}}^{\varepsilon}_{m}=\varepsilon^{-1}(\hat{\boldsymbol{x}}^{\varepsilon}(\Delta+(m+1)\varepsilon)-\hat{\boldsymbol{x}}^{\varepsilon}(\Delta+m\varepsilon))=\bar{\boldsymbol{Z}}^{\varepsilon,\text{NPPD}}_{m}$. While if $\nu\leq m<\lfloor{T/\varepsilon}\rfloor$, we have that $\hat{\boldsymbol{x}}^{\varepsilon}(\Delta+m\varepsilon)=\hat{\boldsymbol{x}}^{\varepsilon}(\Delta+\nu\varepsilon)$. Consequently, $\hat{\boldsymbol{Z}}^{\varepsilon}_{m}\equiv{0}$.)
	
	Now define
	\begin{equation*}
		\hat{\boldsymbol{b}}^{\varepsilon}(\boldsymbol{x},t)\triangleq\int_{\mathbb{R}^d}{\boldsymbol{z}\hat{\mu}^{\varepsilon}_{\boldsymbol{x},t}(\mathrm{d}\boldsymbol{z})}=\begin{cases}
			\bar{\boldsymbol{b}}^\varepsilon(\boldsymbol{x},t), &\text{if}\;\boldsymbol{x}\in\bar{\Sigma}_{\delta_{0},[0,T^{*}]},\\
			\boldsymbol{0}, & \text{otherwise}.
		\end{cases},
	\end{equation*}
	and
	\begin{equation*}
		\hat{\boldsymbol{b}}(\boldsymbol{x},t)\triangleq\begin{cases}
			\bar{\boldsymbol{b}}(\boldsymbol{x},t), &\text{if}\;\boldsymbol{x}\in\bar{\Sigma}_{\delta_{0},[0,T^{*}]},\\
			\boldsymbol{0}, & \text{otherwise}.
		\end{cases}
	\end{equation*}
	Then for $m=1,2,\cdots,\left\lceil{(T^{*}-\Delta)/{\varepsilon}}\right\rceil$,
	\begin{equation*}
		\begin{aligned}
			\hat{\boldsymbol{x}}^{\varepsilon}(\Delta+&m\varepsilon)-\bar{\boldsymbol{x}}^{0}(\Delta+(m\wedge\nu)\varepsilon)\\
			=&\bar{\boldsymbol{x}}^{\varepsilon,\text{NPPD}}(\Delta+(m\wedge\nu)\varepsilon)-\bar{\boldsymbol{x}}^{0}(\Delta+(m\wedge\nu)\varepsilon)\\
			=&\bar{\boldsymbol{x}}^{\varepsilon,\text{NPPD}}(\Delta)-\bar{\boldsymbol{x}}^{0}(\Delta)\\
			&+\varepsilon\sum_{i=0}^{m\wedge\nu-1}\bar{\boldsymbol{Z}}^{\varepsilon,\text{NPPD}}_{i}-\varepsilon\sum_{i=0}^{m\wedge\nu-1}\bar{\boldsymbol{b}}^{\varepsilon}(\bar{\boldsymbol{x}}^{\varepsilon,\text{NPPD}}(\Delta+i\varepsilon),\Delta+i\varepsilon)\\
			&+\varepsilon\sum_{i=0}^{m\wedge\nu-1}\left[\bar{\boldsymbol{b}}^{\varepsilon}(\bar{\boldsymbol{x}}^{\varepsilon,\text{NPPD}}(\Delta+i\varepsilon),\Delta+i\varepsilon)-\bar{\boldsymbol{b}}(\bar{\boldsymbol{x}}^{\varepsilon,\text{NPPD}}(\Delta+i\varepsilon),\Delta+i\varepsilon)\right]\\
			&+\varepsilon\sum_{i=0}^{m\wedge\nu-1}\left[\bar{\boldsymbol{b}}(\bar{\boldsymbol{x}}^{\varepsilon,\text{NPPD}}(\Delta+i\varepsilon),\Delta+i\varepsilon)-\bar{\boldsymbol{b}}(\bar{\boldsymbol{x}}^{0}(\Delta+i\varepsilon),\Delta+i\varepsilon)\right]\\
			&+\varepsilon\sum_{i=0}^{m\wedge\nu-1}\bar{\boldsymbol{b}}(\bar{\boldsymbol{x}}^{0}(\Delta+i\varepsilon),\Delta+i\varepsilon)-\sum_{i=0}^{m\wedge\nu-1}\int_{\Delta+i\varepsilon}^{\Delta+(i+1)\varepsilon}\bar{\boldsymbol{b}}(\bar{\boldsymbol{x}}^{0}(u),u)\mathrm{d}u\\
			=&\Theta_{1}+\Theta_{2}+\Theta_{3}+\Theta_{4}+\Theta_{5}.
		\end{aligned}
	\end{equation*}
	
	By the $\mathcal{C}^{1}$-smoothness of $\bar{\boldsymbol{b}}(\boldsymbol{x},t)$ on $\bar{\Sigma}_{\delta_{0},[0,T^{*}]}\times[0,T^{*}]$, we have
	\begin{equation*}
		\vert\Theta_{1}\vert\leq\int_{0}^{\Delta}\left\vert\hat{\boldsymbol{b}}(\bar{\boldsymbol{x}}^{0}(u),u)\right\vert\mathrm{d}u\leq{C_{1}}\Delta\leq{C_{1}}\varepsilon,
	\end{equation*}
	\begin{equation*}
		\begin{aligned}
			\vert\Theta_{4}\vert&=\left\vert\varepsilon\sum_{i=0}^{m\wedge\nu-1}\left[\hat{\boldsymbol{b}}(\hat{\boldsymbol{x}}^{\varepsilon}(\Delta+i\varepsilon),\Delta+i\varepsilon)-\hat{\boldsymbol{b}}(\bar{\boldsymbol{x}}^{0}(\Delta+(i\wedge\nu)\varepsilon),\Delta+i\varepsilon)\right]\right\vert\\
			&\leq\varepsilon\sum_{i=0}^{m\wedge\nu-1}{C_{4}}\left\vert\hat{\boldsymbol{x}}^{\varepsilon}(\Delta+i\varepsilon)-\bar{\boldsymbol{x}}^{0}(\Delta+(i\wedge\nu)\varepsilon)\right\vert\\
			&\leq\varepsilon\sum_{i=0}^{m-1}{C_{4}}\left\vert\hat{\boldsymbol{x}}^{\varepsilon}(\Delta+i\varepsilon)-\bar{\boldsymbol{x}}^{0}(\Delta+(i\wedge\nu)\varepsilon)\right\vert,
		\end{aligned}
	\end{equation*}
	\begin{equation*}
		\begin{aligned}
			\vert\Theta_{5}\vert&\leq\sum_{i=0}^{m\wedge\nu-1}\int_{\Delta+i\varepsilon}^{\Delta+(i+1)\varepsilon}\left\vert\hat{\boldsymbol{b}}(\bar{\boldsymbol{x}}^{0}(\Delta+i\varepsilon),\Delta+i\varepsilon)-\hat{\boldsymbol{b}}(\bar{\boldsymbol{x}}^{0}(u),u)\right\vert\mathrm{d}u\\
			&\leq{C_{5}}\varepsilon^2(m\wedge\nu)\leq{C_{5}}\varepsilon^2\left\lceil{(T^{*}-\Delta)/{\varepsilon}}\right\rceil\leq{C_{5}}T\varepsilon,
		\end{aligned}
	\end{equation*}
	where $C_{1},\,C_{4}$ are constants delineating the upper bound and the Lipschitz coefficient of $\bar{\boldsymbol{b}}(\boldsymbol{x},t)$ on $\bar{\Sigma}_{\delta_{0},[0,T^{*}]}\times[0,T^{*}]$ respectively, $C_{5}$ is an upper bound on the first-order derivative of $\bar{\boldsymbol{b}}(\bar{\boldsymbol{x}}^{0}(t),t)$ on the time interval $[0,T^{*}+\varepsilon]$. 
	
	By the uniform convergence of $\bar{\boldsymbol{b}}^{\varepsilon}(\boldsymbol{x},t)$ to $\bar{\boldsymbol{b}}(\boldsymbol{x},t)$ on $\bar{\Sigma}_{\delta_{0},[0,T^{*}]}\times[0,T^{*}]$, we obtain
	\begin{equation*}
		\begin{aligned}
			\vert\Theta_{3}\vert&\leq\varepsilon\sum_{i=0}^{m\wedge\nu-1}\left\vert\hat{\boldsymbol{b}}^{\varepsilon}(\hat{\boldsymbol{x}}^{\varepsilon}(\Delta+i\varepsilon),\Delta+i\varepsilon)-\hat{\boldsymbol{b}}(\hat{\boldsymbol{x}}^{\varepsilon}(\Delta+i\varepsilon),\Delta+i\varepsilon)\right\vert\\
			&\leq{T}\sup_{(\boldsymbol{x},t)\in\bar{\Sigma}_{\delta_{0},[0,T^{*}]}\times[0,T^{*}]}\left\vert\hat{\boldsymbol{b}}^{\varepsilon}(\boldsymbol{x},t)-\hat{\boldsymbol{b}}(\boldsymbol{x},t)\right\vert=C_{3}(\varepsilon)T\to{0},\;\;\text{as}\;\varepsilon\to{0}.
		\end{aligned}
	\end{equation*}
	
	In addition,
	\begin{equation*}
		\begin{aligned}
			\vert\Theta_{2}\vert&=\left\vert\varepsilon\sum_{i=0}^{m\wedge\nu-1}\hat{\boldsymbol{Z}}^{\varepsilon}_{i}-\varepsilon\sum_{i=0}^{m\wedge\nu-1}\hat{\boldsymbol{b}}^{\varepsilon}(\hat{\boldsymbol{x}}^{\varepsilon}(\Delta+i\varepsilon),\Delta+i\varepsilon)\right\vert\\
			&=\left\vert\varepsilon\sum_{i=0}^{m-1}\hat{\boldsymbol{Z}}^{\varepsilon}_{i}-\varepsilon\sum_{i=0}^{m-1}\hat{\boldsymbol{b}}^{\varepsilon}(\hat{\boldsymbol{x}}^{\varepsilon}(\Delta+i\varepsilon),\Delta+i\varepsilon)\right\vert\\
			&\leq\max_{1\leq{m}\leq\left\lceil{(T^{*}-\Delta)/{\varepsilon}}\right\rceil}\left\vert\varepsilon\sum_{i=0}^{m-1}\hat{\boldsymbol{Z}}^{\varepsilon}_{i}-\varepsilon\sum_{i=0}^{m-1}\hat{\boldsymbol{b}}^{\varepsilon}(\hat{\boldsymbol{x}}^{\varepsilon}(\Delta+i\varepsilon),\Delta+i\varepsilon)\right\vert\\
			&=C_2(\varepsilon).
		\end{aligned}
	\end{equation*}
	
	Utilizing the discrete version of Gronwall’s inequality, one can see that
	\begin{equation*}
		\begin{aligned}
			\max_{1\leq{m}\leq\left\lceil{(T^{*}-\Delta)/{\varepsilon}}\right\rceil}\big\vert\hat{\boldsymbol{x}}^{\varepsilon}&(\Delta+m\varepsilon)-\bar{\boldsymbol{x}}^{0}(\Delta+(m\wedge\nu)\varepsilon)\big\vert\\
			&\leq\left(C_{1}\varepsilon+C_2(\varepsilon)+C_3(\varepsilon)T+C_5T\varepsilon\right)(1+C_4\varepsilon)^{\left\lceil{(T^{*}-\Delta)/{\varepsilon}}\right\rceil}\\
			&\leq\left(C_{1}\varepsilon+C_2(\varepsilon)+C_3(\varepsilon)T+C_5T\varepsilon\right)e^{C_4T}.
		\end{aligned}
	\end{equation*}
	This completes the proof.
\end{proof}
\begin{lemma}\label{theoA20102}
	Let $\varphi:\mathbb{R}\mapsto\mathbb{R}$ be a nonnegative, even, convex function such that both $\varphi$ and $\varphi^{'}$ are absolutely continuous. If $\varphi(0)=\varphi^{'}(0)=0$ and $\varphi^{''}$ is nonnegative and non-increasing on $(0,+\infty)$ (for example, $\varphi(u)=\vert{u}\vert^{\theta},\,1<\theta\leq{2}$), then for each $\delta>0$,
	\begin{equation}
		\begin{aligned}
			P^{\varepsilon}\{C_2&(\varepsilon)>\delta\}\leq\frac{4}{d\varphi(\delta)}\sum_{i=0}^{\left\lceil{(T^{*}-\Delta)/{\varepsilon}}\right\rceil-1}\sum_{j=1}^{d}\\
			&E^{\varepsilon}\int_{\mathbb{R}^d}\varphi\left(\frac{d\varepsilon}{2}\left(z_j-\hat{b}^{\varepsilon}_{j}(\hat{\boldsymbol{x}}(\Delta+i\varepsilon),\Delta+i\varepsilon)\right)\right)\hat{\mu}^{\varepsilon}_{\hat{\boldsymbol{x}}(\Delta+i\varepsilon),\Delta+i\varepsilon}(\mathrm{d}\boldsymbol{z})
		\end{aligned}
	\end{equation}
	
\end{lemma}
\begin{proof}
	Let $\boldsymbol{\xi}^{\varepsilon}(\Delta)\triangleq\boldsymbol{0}$, and for $m=1,2,\cdots,\lfloor{T/\varepsilon}\rfloor$, define
	\begin{equation*}
		\begin{aligned}
			\boldsymbol{\xi}^{\varepsilon}(\Delta+m\varepsilon)&\triangleq\hat{\boldsymbol{x}}^{\varepsilon}(\Delta+m\varepsilon)-\boldsymbol{x}_T-\varepsilon\sum_{i=0}^{m-1}\hat{\boldsymbol{b}}^{\varepsilon}(\hat{\boldsymbol{x}}^{\varepsilon}(\Delta+i\varepsilon),\Delta+i\varepsilon)\\
			&=\varepsilon\sum_{i=0}^{m-1}\hat{\boldsymbol{Z}}^{\varepsilon}_{i}-\varepsilon\sum_{i=0}^{m-1}\hat{\boldsymbol{b}}^{\varepsilon}(\hat{\boldsymbol{x}}^{\varepsilon}(\Delta+i\varepsilon),\Delta+i\varepsilon),
		\end{aligned}
	\end{equation*}
	Then $\{\hat{\boldsymbol{x}}^{\varepsilon}(\Delta+m\varepsilon),\boldsymbol{\xi}^{\varepsilon}(\Delta+m\varepsilon)\}$ is a Markov chain, and  $\{\boldsymbol{\xi}^{\varepsilon}(\Delta+m\varepsilon)\}$ forms a discrete martingale. The martingale inequality implies
	\begin{equation*}
		P^{\varepsilon}\left\{\max_{1\leq{m}\leq\left\lceil{(T^{*}-\Delta)/{\varepsilon}}\right\rceil}\left\vert\boldsymbol{\xi}^{\varepsilon}(\Delta+m\varepsilon)\right\vert>\delta\right\}\leq\frac{E^{\varepsilon}{\varphi\left(\left\vert\boldsymbol{\xi}^{\varepsilon}\left(\Delta+\left\lceil{(T^{*}-\Delta)/{\varepsilon}}\right\rceil\varepsilon\right)\right\vert\right)}}{\varphi(\delta)}.
	\end{equation*} 
	The problem is to estimate the expectation on the right. 
	
	Since
	\begin{equation*}
		\vert\boldsymbol{\xi}\vert\leq\vert\xi_{1}\vert+\cdots+\vert\xi_{d}\vert\leq\sqrt{d}\vert\boldsymbol{\xi}\vert
	\end{equation*}
	we have
	\begin{equation}
		\varphi(\vert\boldsymbol{\xi}\vert)\leq\varphi\left(\vert\xi_{1}\vert+\cdots+\vert\xi_{d}\vert\right)\leq{d}^{-1}\sum_{j=1}^{d}\varphi(d\vert\xi_{j}\vert)={d}^{-1}\sum_{j=1}^{d}\varphi(d\xi_{j}).
	\end{equation}
	Define $\bar{\varphi}(\boldsymbol{\xi})\triangleq{d}^{-1}\sum_{j=1}^{d}\varphi(d\xi_{j})$. Then
	\begin{equation*}
		E^{\varepsilon}{\varphi\left(\left\vert\boldsymbol{\xi}^{\varepsilon}\left(\Delta+\left\lceil{(T^{*}-\Delta)/{\varepsilon}}\right\rceil\varepsilon\right)\right\vert\right)}\leq{E^{\varepsilon}\bar{\varphi}\left(\boldsymbol{\xi}^{\varepsilon}\left(\Delta+\left\lceil{(T^{*}-\Delta)/{\varepsilon}}\right\rceil\varepsilon\right)\right)}.
	\end{equation*}
	Let $\bar{\mathcal{F}}^{\text{NPPD}}_{\leq m}$ denote the $\sigma$-algebra generated by $\bar{\boldsymbol{x}}^{\varepsilon,\text{NPPD}}(\Delta+i\varepsilon),\;i=0,1,\cdots,m$, and define
	\begin{equation*}
		\hat{\mathcal{F}}_{\leq{m}}\triangleq\bar{\mathcal{F}}^{\text{NPPD}}_{\leq {m\wedge\nu}}=\left\{A\in\bar{\mathcal{F}}^{\text{NPPD}}_{\leq m}:A\cap\left\{\nu\leq{k}\right\}\in\bar{\mathcal{F}}^{\text{NPPD}}_{\leq k},\, k=0,1,\cdots,m\right\}.
	\end{equation*}
	We know that both $\hat{\boldsymbol{x}}^{\varepsilon}(\Delta+m\varepsilon)$ and $\boldsymbol{\xi}^{\varepsilon}(\Delta+m\varepsilon)$ are $\hat{\mathcal{F}}_{\leq{m}}$-measurable. Consequently,
	\begin{equation*}
		\begin{aligned}
			E^{\varepsilon}\bar{\varphi}(\boldsymbol{\xi}^{\varepsilon}(\Delta+&\left\lceil{(T^{*}-\Delta)/{\varepsilon}}\right\rceil\varepsilon))\\
			=&\bar{\varphi}(\boldsymbol{0})+\sum_{i=0}^{\left\lceil{(T^{*}-\Delta)/{\varepsilon}}\right\rceil-1}E^{\varepsilon}\left[\bar{\varphi}(\boldsymbol{\xi}^{\varepsilon}(\Delta+(i+1)\varepsilon))-\bar{\varphi}(\boldsymbol{\xi}^{\varepsilon}(\Delta+i\varepsilon))\right]\\
			=&\bar{\varphi}(\boldsymbol{0})+\sum_{i=0}^{\left\lceil{(T^{*}-\Delta)/{\varepsilon}}\right\rceil-1}E^{\varepsilon}\bigg\{E^{\varepsilon}\bigg[\bar{\varphi}(\boldsymbol{\xi}^{\varepsilon}(\Delta+(i+1)\varepsilon))-\bar{\varphi}(\boldsymbol{\xi}^{\varepsilon}(\Delta+i\varepsilon))\\
			&-(\boldsymbol{\xi}^{\varepsilon}(\Delta+(i+1)\varepsilon)-\boldsymbol{\xi}^{\varepsilon}(\Delta+i\varepsilon))\cdot\nabla_{\boldsymbol{\xi}}\bar{\varphi}(\boldsymbol{\xi}^{\varepsilon}(\Delta+i\varepsilon))\bigg\vert\hat{\mathcal{F}}_{\leq{i}}\bigg]\bigg\}
		\end{aligned}
	\end{equation*}
	Combined with the fact
	\begin{equation*}
		\begin{aligned}
			\varphi(u+z)-\varphi(z)-u\varphi^{'}(z)&=\int_{0}^{u}\int_{0}^{v}\varphi^{''}(w+z)\mathrm{d}w\mathrm{d}v\\
			&\leq2\int_{0}^{\vert{u}\vert}\int_{0}^{v/2}\varphi^{''}(w)\mathrm{d}w\mathrm{d}v\\
			&=4\varphi\left(\frac{\vert{u}\vert}{2}\right).
		\end{aligned}
	\end{equation*}
	we obtain
	\begin{equation*}
		\begin{aligned}
			E^{\varepsilon}\bar{\varphi}(\boldsymbol{\xi}^{\varepsilon}(\Delta&+\left\lceil{(T^{*}-\Delta)/{\varepsilon}}\right\rceil\varepsilon))\leq{4}d^{-1}\sum_{i=0}^{\left\lceil{(T^{*}-\Delta)/{\varepsilon}}\right\rceil-1}\sum_{j=1}^{d}\\
			&E^{\varepsilon}\int_{\mathbb{R}^d}\varphi\left(\frac{d\varepsilon}{2}\left(z_j-\int_{\mathbb{R}^d}w_j\hat{\mu}^{\varepsilon}_{\hat{\boldsymbol{x}}(\Delta+i\varepsilon),\Delta+i\varepsilon}(\mathrm{d}\boldsymbol{w})\right)\right)\hat{\mu}^{\varepsilon}_{\hat{\boldsymbol{x}}(\Delta+i\varepsilon),\Delta+i\varepsilon}(\mathrm{d}\boldsymbol{z}).
		\end{aligned}
	\end{equation*}
	This completes the proof.
\end{proof}
\begin{proof}[Proof of Proposition \ref{theo040101}]
	Let $\varphi(u)=\vert{u}\vert^{1+\varkappa_1}$. Utilizing the generalized mean inequality:
	\begin{equation*}
		\left(\frac{\sum_{j=1}^{d}\vert\xi_{j}\vert^{1+\varkappa_1}}{d}\right)^{1/(1+\varkappa_1)}\leq\left(\frac{\sum_{j=1}^{d}\vert\xi_{j}\vert^{2}}{d}\right)^{1/2},
	\end{equation*}
	one can see that
	\begin{equation*}
		\begin{aligned}
			P^{\varepsilon}\{C_2(\varepsilon)>\delta\}\leq&4\left(\frac{\sqrt{d}\varepsilon}{2\delta}\right)^{1+\varkappa_1}\sum_{i=0}^{\left\lceil{(T^{*}-\Delta)/{\varepsilon}}\right\rceil-1}\\
			&E^{\varepsilon}\int_{\mathbb{R}^d}\left\vert\boldsymbol{z}-\hat{\boldsymbol{b}}^{\varepsilon}(\hat{\boldsymbol{x}}(\Delta+i\varepsilon),\Delta+i\varepsilon)\right\vert^{1+\varkappa_1}\hat{\mu}^{\varepsilon}_{\hat{\boldsymbol{x}}(\Delta+i\varepsilon),\Delta+i\varepsilon}(\mathrm{d}\boldsymbol{z})\\
			\leq&4\left(\frac{\sqrt{d}\varepsilon}{2\delta}\right)^{1+\varkappa_1}\left\lceil{(T^{*}-\Delta)/{\varepsilon}}\right\rceil G_2\\
			\leq&\frac{2^{1-\varkappa_1}d^{(1+\varkappa_1)/2}TG_2}{\delta^{1+\varkappa_1}}\varepsilon^{\varkappa_1}\to{0},\quad\text{as}\;\varepsilon\to{0}.
		\end{aligned}
	\end{equation*}
	Then Lemma \ref{theoA20101} implies that for any $T^{*}\in(0,T)$, there exists a constant $\delta_0$ so that for every $\eta>0$,
	\begin{equation*}
		\lim_{\varepsilon\to{0}}P^{\varepsilon}\left\{\max_{1\leq{m}\leq\left\lceil{(T^{*}-\Delta)/{\varepsilon}}\right\rceil}\big\vert\bar{\boldsymbol{x}}^{\varepsilon,\text{NPPD}}(\Delta+(m\wedge\nu)\varepsilon)-\bar{\boldsymbol{x}}^{0}(\Delta+(m\wedge\nu)\varepsilon)\big\vert>\eta\right\}=0.
	\end{equation*}
	Note that 
	\begin{equation*}
		\left\vert\bar{\boldsymbol{x}}^{0}(\Delta+(\left\lceil{(T^{*}-\Delta)/{\varepsilon}}\right\rceil\wedge\nu)\varepsilon)-\bar{\boldsymbol{x}}^{0}(\Delta+((\left\lceil{(T^{*}-\Delta)/{\varepsilon}}\right\rceil-1)\wedge\nu)\varepsilon)\right\vert\leq{O}(\varepsilon),
	\end{equation*}
	and
	\begin{equation*}
		B_{\delta_{0}}(\bar{\boldsymbol{x}}^{0}(\Delta+(m\wedge\nu)\varepsilon))\subset\bar{\Sigma}_{\delta_{0},[0,T^{*}]},\; \text{for}\; m=0,1,\cdots,\left\lceil{(T^{*}-\Delta)/{\varepsilon}}\right\rceil-1.
	\end{equation*}
	There must exist a constant $\varepsilon_0>0$ so that for all $\varepsilon<\varepsilon_0$,
	\begin{equation*}
		B_{\delta_{0}/2}(\bar{\boldsymbol{x}}^{0}(\Delta+(m\wedge\nu)\varepsilon))\subset\bar{\Sigma}_{\delta_{0},[0,T^{*}]},\;\text{for}\; m=0,1,\cdots,\left\lceil{(T^{*}-\Delta)/{\varepsilon}}\right\rceil. 
	\end{equation*} 
	If
	\begin{equation*}
		\max_{1\leq{m}\leq\left\lceil{(T^{*}-\Delta)/{\varepsilon}}\right\rceil}\big\vert\bar{\boldsymbol{x}}^{\varepsilon,\text{NPPD}}(\Delta+(m\wedge\nu)\varepsilon)-\bar{\boldsymbol{x}}^{0}(\Delta+(m\wedge\nu)\varepsilon)\big\vert\leq\eta\leq\delta/2<\delta_{0}/2,
	\end{equation*}
	then we have
	\begin{equation*}
		\bar{\boldsymbol{x}}^{\varepsilon,\text{NPPD}}(\Delta+(m\wedge\nu)\varepsilon)\in\bar{\Sigma}_{\delta_{0},[0,T^{*}]},\; \text{for}\;m=0,1,\cdots,\left\lceil{(T^{*}-\Delta)/{\varepsilon}}\right\rceil.
	\end{equation*}
	Consequently, $\left\lceil{(T^{*}-\Delta)/{\varepsilon}}\right\rceil<\nu$. This means
	\begin{equation*}
		\max_{1\leq{m}\leq\left\lceil{(T^{*}-\Delta)/{\varepsilon}}\right\rceil}\big\vert\bar{\boldsymbol{x}}^{\varepsilon,\text{NPPD}}(\Delta+m\varepsilon)-\bar{\boldsymbol{x}}^{0}(\Delta+m\varepsilon)\big\vert\leq\delta/2,
	\end{equation*}
	and therefore for sufficiently small $\varepsilon$,
	\begin{equation*}
		\sup_{t\in[0,T^{*}]}\left\vert\bar{\boldsymbol{x}}^{\varepsilon,\text{NPPD}}(t)-\bar{\boldsymbol{x}}^{0}(t)\right\vert\leq\delta/2+O(\varepsilon)\leq\delta,
	\end{equation*}
	so we have proved the claim.
\end{proof}
\section{An Algorithm for Calculating the Non-stationary Prehistory Probability Density}\label{secA3}
\textbf{Initialization:}
\begin{itemize}
	\item[(a)] Choose a domain $D=[x_l,x_r)\subset\mathbb{R}$ such that $x_0,\,x_T\in{D}$, and 
	\begin{equation*}
		\min_{x\in\partial{D},\,0\leq{t}\leq{T}}S(x,t\vert{x}_{0})>S(x_T,T\vert x_0).
	\end{equation*}
	(This ensures that the NOP connecting $x_0$ and $x_T$ in the time span $T$ lies entirely within the domain $D$, and that neglecting trajectories escaping from $D$ before the moment $T$ has a negligible effect on the computation of the NOP.)
	\item[(b)] Partition the state space $\mathbb{R}$ into $N_x+1$ subsets: For $i=1,2,\cdots,N_x$, define $i$-th subset as $D_i\triangleq[x_l+(x_r-x_l)(i-1)/N_x,x_l+(x_r-x_l)i/N_x)$. Choose a representative element in $D_i$ as $x(i)\triangleq x_l+(x_r-x_l)(i-1/2)/N_x$. Define $(N_x+1)$-th subset as $D_{N_x+1}\triangleq\mathbb{R}\setminus{D}$. 
	\item[(c)] Let $\tau^\varepsilon$ be the first exit time of $x^\varepsilon(t)$ from $D$. Then for any $T>0$ and sufficiently small $\varepsilon$, the process $x^\varepsilon(t)$ can be uniformly approximated on $[0,T]$ by $x^{\varepsilon}(t\wedge\tau^{\varepsilon})$, which is further approximated by a discrete Markov chain with $N_x+1$ states, whose transition probabilities are given by
	\begin{equation*}
		\begin{cases}
			\begin{aligned}
				P_{i,j}&=P(x^{\varepsilon}((n+1)\varepsilon)\in D_j\vert x^{\varepsilon}(n\varepsilon)=x(i))\\
				&=\mu_{x(i)}(\varepsilon^{-1}(D_j-x(i)))
			\end{aligned} & 1\leq i\leq N_x,\;1\leq j\leq N_x,\\
			P_{i,j}=1-\sum_{j=1}^{N_x}P_{i,j}\;\; &1\leq i\leq N_x,\;j=N_x+1,\\
			P_{i,j}=0\;\;&i=N_x+1,\;1\leq j\leq N_x,\\
			P_{i,j}=1\;\;&i=N_x+1,\;j=N_x+1.\\
		\end{cases}
	\end{equation*}
\end{itemize}

\textbf{Algorithm:}
\begin{itemize}
	\item[(a)] Define 
	\begin{equation*}
		p_i(n)\triangleq p^{\varepsilon}(D_i,n\varepsilon\vert x_0,0)\simeq p^{\varepsilon}(x(i),n\varepsilon\vert x_0,0)\cdot(x_r-x_l)/N_x,
	\end{equation*}
	for $n=0,1,\cdots,\lfloor{t/\varepsilon}\rfloor$. Then $p_i(n)$ can be approximated recursively by
	\begin{equation*}
		p_j(n+1)=\sum_{i=1}^{N_x+1}p_i(n)P_{i,j},\quad j=1,\cdots,N_x+1, \; n=0,\cdots, \lfloor{T/\varepsilon}\rfloor-1,
	\end{equation*}
	with the initial condition
	\begin{equation*}
		p_i(0)=\begin{cases}
			1,\quad & x_0\in D_i,\\
			0,\quad & \text{otherwise}.\\\end{cases}
	\end{equation*}
	\item[(b)] The time-inhomogeneous transition probabilities for the reversed process $\bar{x}^{\varepsilon,\text{NPPD}}$ are given by
	\begin{equation*}
		\bar{P}^{(n)}_{i,j}\triangleq\begin{cases}
			\frac{p_j(n)P_{j,i}}{p_{i}(n+1)},\quad & p_{i}(n+1)>0,\\
			0,\quad &p_{i}(n+1)=0.
		\end{cases}
	\end{equation*}
	\item[(c)] The NPPD can be approximated as
	\begin{equation*}
		q^{\varepsilon,\text{NPPD}}(x(i),t;x_T,T;x_0,0)\simeq \bar{p}_{i}(\lfloor{t/\varepsilon}\rfloor)\cdot N_x/(x_r-x_l),
	\end{equation*}
	where $\bar{p}_{i}(n)$ is computed via the backward recursion:
	\begin{equation*}
		\bar{p}_{j}(n)=\sum_{i=1}^{N_x+1}\bar{p}_{i}(n+1)\bar{P}_{i,j}^{(m)},\quad j=1,\cdots,N_x+1,\; n=\lfloor{T/\varepsilon}\rfloor-1,\cdots,0,
	\end{equation*}
	with the terminal condition
	\begin{equation*}
		\bar{p}_{i}(\lfloor{T/\varepsilon}\rfloor)=\begin{cases}
			1,\quad & x_T\in D_i,\\
			0,\quad & \text{otherwise}.
		\end{cases}
	\end{equation*}
\end{itemize}
\bibliographystyle{elsarticle-num} 
\bibliography{References}
\end{document}